\documentclass[a4paper,12pt,twoside]{amsart}
\usepackage{vmargin}
\usepackage{amssymb}
\usepackage{amsmath}
\usepackage{mathrsfs}
\usepackage[dvips]{graphicx}
\usepackage{vmargin}
\usepackage{color}
\usepackage{amscd}
\DeclareGraphicsExtensions{.pdf, .png, .jpg, .mps}
\usepackage[colorlinks=false]{hyperref}
\definecolor{myc}{cmyk}{0.0009,0.8,0.8,0.00}
\newtheorem{theorem}{Theorem}[section]
\newtheorem{lem}[theorem]{Lemma}
\newtheorem{pro}[theorem]{Proposition}
\newtheorem{defi}[theorem]{Definition}

\newtheorem{oss}[theorem]{Remark}

\numberwithin{equation}{section}
\newtheorem*{nb}{\footnotesize {N.B}}

\def\p{\partial}

\def\io{{\infty}}
\def\range{\operatorname{ran}}
\def\diag{\operatorname{diag}}

\def\re{\operatorname{Re}}
\def\im{\operatorname{Im}}
\def\Id{\operatorname{Id}}

\def\N{\mathbb N}
\def\Z{\mathbb Z}

\def\R{\mathbb R}
\def\C{\mathbb C}
\def\poscal#1#2{\langle#1,#2\rangle}

\def\poi#1#2{\left\{#1,#2\right\}}
\def\norm#1{\Vert#1\Vert}
\def\val#1{\vert#1\vert}

\def\l2{L^2(\R^{n})}
\def\L2{L^2(\R^{2n})}
\def\supp{\operatorname{supp}}

\def\w#1{{#1^{\text {Wick}}}}

\def\RZ{\R^{2n}}
\def\trace{\operatorname{trace}}

\def\hs{{\hskip15pt}}
\def\vs{\vskip.3cm}

\def\mat22#1#2#3#4{\begin{pmatrix}#1&#2\\ #3&#4\end{pmatrix}}

\def\rank{\operatorname{rank}}

\def\ep{\varepsilon}
\def\Xint#1{\mathchoice
   {\XXint\displaystyle\textstyle{#1}}%
   {\XXint\textstyle\scriptstyle{#1}}%
   {\XXint\scriptstyle\scriptscriptstyle{#1}}%
   {\XXint\scriptscriptstyle\scriptscriptstyle{#1}}%
   \!\int}
\def\XXint#1#2#3{{\setbox0=\hbox{$#1{#2#3}{\int}$}
     \vcenter{\hbox{$#2#3$}}\kern-.5\wd0}}

\def\fint{\Xint-}
\def\beq{\begin{equation}}
\def\eeq{\end{equation}}
\begin{document}
\title[             Anisotropic Bose-Einstein Condensates ]{Fast rotating
 Bose-Einstein condensates  in an asymmetric trap}
\author[             A. Aftalion]{Amandine Aftalion}
\begin{address}{Amandine Aftalion, CMAP, Ecole Polytechnique, CNRS, 91128 Palaiseau cedex, France}
\email{amandine.aftalion@polytechnique.edu}
\urladdr{http://www.cmap.polytechnique.fr/~aftalion/}
\end{address}
\author[             X. Blanc]{Xavier Blanc}
\begin{address}{Xavier Blanc, Universit{\'e} Pierre et Marie Curie-Paris6, UMR
  7598, laboratoire Jacques-Louis Lions, 175 rue du Chevaleret, Paris
  F-75013 France}
\email{blanc@ann.jussieu.fr}
\urladdr{http://www.ann.jussieu.fr/~blanc/}
\end{address}
\author[             N. Lerner]{Nicolas  Lerner}
\begin{address}{Nicolas Lerner,
Projet analyse fonctionnelle, Institut de Math\'ematiques de Jussieu,
Universit\'e Pierre-et-Marie-Curie (Paris 6),
175 rue du Chevaleret, 75013 Paris, France.}
\email{lerner@math.jussieu.fr}
\urladdr{http://www.math.jussieu.fr/~lerner/}
\end{address}
\date{\today}
\begin{abstract} We investigate the effect of the anisotropy of a
harmonic trap on the behaviour of a fast rotating Bose-Einstein
condensate. This is done in the framework of the 2D Gross-Pitaevskii equation
 and requires a symplectic reduction of the quadratic form defining the
 energy. This reduction allows us to simplify the energy on a Bargmann space
 and study the
 asymptotics of large rotational velocity.
 We characterize two regimes of velocity and anisotropy; in the first one where
  the behaviour is similar to the isotropic case, we construct an
  upper bound: a hexagonal Abrikosov
 lattice of vortices, with an inverted parabola profile.
The second regime deals with very large velocities, a case in which
we prove that the ground state does
 not display vortices in the bulk, with a 1D limiting problem.
 In that case,
    we show that the coarse grained
atomic density behaves like an inverted parabola with large radius
in the deconfined direction but  keeps a fixed profile given by a
Gaussian in the other direction.
The features of this second regime appear as new phenomena.
\end{abstract}
\maketitle
{\footnotesize\baselineskip=0.72\normalbaselineskip
\tableofcontents
}
\section{Introduction}
Bose-Einstein condensates (BEC) are a new phase of matter where
various aspects of macroscopic quantum physics can be studied. Many
experimental and theoretical works have emerged in the past ten
years.
 We refer to the monographs by C.J.Pethick-H.Smith \cite{PS},
L.Pitaevskii-S.Stringari
 \cite{PiSi} for more details on the physics and to A.Aftalion \cite{A}  for the mathematical aspects.
 Our work is motivated by experiments in
  the  group of  J.Dalibard \cite{Madison00}  on rotating
  condensates: when  a condensate is rotated
  at a sufficiently large velocity,
   a superfluid behaviour is detected with the observation
   of quantized
  vortices. These vortices arrange themselves on a lattice, similar
 to Abrikosov lattices in superconductors \cite{Abrikosov}. This fast
 rotation regime is of interest for its analogy with Quantum Hall physics
 \cite{BSSD,tlho,wbp}.
\par
 In a previous work, A.Aftalion, X.Blanc and F.Nier \cite{MR2271933}
 have addressed
 the mathematical aspects of fast rotating condensates in harmonic isotropic
 traps and gave a mathematical description of the observed vortex
 lattice. This was done through the minimization of the Gross-Pitaevskii energy and the introduction of Bargmann spaces to describe
 the lowest Landau level sets of states. Nevertheless, the experimental device
 leading to the realization of a rotating condensate
  requires an anisotropy of the  trap holding the atoms, which was
  not taken into account in \cite{MR2271933}. Several physics
 papers have addressed the behaviour of anisotropic condensates
 under rotation and its similarity or differences
  with isotropic traps. We refer the reader to the paper by
 A.Fetter \cite{F}, and to the related works \cite{oktel,sanchez-lotero,sinha}.
 The aim of the present article is
 to analyze the effect of anisotropy on the energy
 minimization and the vortex pattern, and in particular to
 derive a mathematical study of some of Fetter's computations
  and conjectures.
  Two different situations emerge according to the values of
  the parameters: in one case,
 the behaviour is similar to the isotropic case with a triangular
 vortex lattice; in the other case, for very large velocities, we have found a new regime
 where there are no vortices, and a full mathematical analysis
 can be performed, reducing the minimization  to a 1D problem.
The existence of this new regime was apparently not predicted in the physics
  literature.
  This feature relies on the analysis of the bottom of the
   spectrum of a specific  operator whose positive lower bound
    prevents the
   condensate from shrinking in one direction,
   contradicting some heuristic explanations present in \cite{F}.
   Our analysis is based on the symplectic
 reduction of the quadratic form defining the Hamiltonian (inspired
 by the computations of Fetter \cite{F}), the characterization
 of a lowest Landau level adapted to the anisotropy and finally the study of
 the reduced energy in this space.
 \subsection{The physics problem and its mathematical formulation}Our problem comes from  the study  of the
3D Gross-Pitaevskii energy functional for a fast rotating
Bose-Einstein condensate with $N$ particles of mass $m$ given by
\begin{equation}
\mathcal E_{GP} (\phi)=\poscal{\mathcal H
\phi}{\phi}_{L^2(\R^3)}+\frac{g_{3d}N}{2}\norm{\phi}^4_{L^4(\R^3)},
\end{equation}
where the operator $\mathcal H$ is
\begin{equation}
\mathcal H=\frac{1}{2m}(h^2D_{x}^2+h^2
D_{y}^2+h^2D_{z}^2)+\frac{m}{2}\bigl(\omega_{x}^2
x^2+\omega_{y}^2y^2+\omega_z^2z^2\bigr)-\Omega(x hD_{y}-yh D_{x}),
 \end{equation}
 where  $h$ is the Planck constant, $D_{x}=(2i\pi)^{-1}\p_{x}$,
$\omega_{j}$ is the frequency along the $j$-axis, $\Omega$ is the
rotational velocity, and the coupling constant $g_{3d}$ is a positive
parameter.
\par
In the particular case where $\omega_x=\omega_y$, the fast rotation
regime corresponds to the case where $\Omega$ tends to $\omega_x$
 and the condensate expands in the transverse direction. It
has been proved \cite{aft-bl} that the minimizer can be described at
leading order by a 2D function $\psi(x,y)$,  multiplied by the
ground state of the harmonic oscillator in the $z$-direction (the
operator $h^2/(2m)D_z^2+m\omega_z^2z^2/2$), which is equal to
$({2m\omega_{z} h^{-1}})^{1/4} e^{-\pi m\omega_z h^{-1}z^2}.$ This
property is still true in the anisotropic case if $\omega_y\ll
\omega_z$. The reduced 2D energy to study is thus
\begin{equation}
\mathcal E (\psi)=\poscal{\mathcal H_{0}
\psi}{\psi}_{L^2(\R^2)}+\frac{g_{2d}N}{2}\norm{\psi}^4_{L^4(\R^2)},
\end{equation}
where the operator $\mathcal H_{0}$ is
\begin{equation}
\mathcal H_{0}=\frac{1}{2m}(h^2D_{x}^2+h^2
D_{y}^2)+\frac{m}{2}\bigl(\omega_{x}^2
x^2+\omega_{y}^2y^2\bigr)-\Omega(x hD_{y}-yh D_{x}),
 \end{equation}
 and the coupling constant $g_{2d}$ takes into account the integral of the
 ground state in the $z$-direction:
\begin{equation}\label{1.constg}
g_{2d}N=\frac{gh^2}{m},\quad\text{where $g$ is dimensionless (and $>0$).}
\end{equation}
Since $h$ has the dimension energy $\times$ time, it is consistent
to assume that the wave function $\psi$ has the dimension {
1/length}, with the normalization $ \norm{\psi}_{L^2(\R^2)}=1. $ We
define the mean square oscillator frequency $\omega_{\perp}$ by
$$
\omega_{\perp}^2=\frac12(\omega_{x}^2+\omega_{y}^2)
$$
and the function  $u$ by
\begin{equation}
\psi(x,y)=h^{-1/2}m^{1/2}\omega_{\perp}^{1/2}
u(h^{-1/2}m^{1/2}\omega_{\perp}^{1/2}  x,h^{-1/2}m^{1/2}
\omega_{\perp}^{1/2} y),
\end{equation}
so that
$$
\norm{u}_{L^2(\R^2)}=\norm{\psi}_{L^2(\R^2)}=1,\quad
g_{2d}N\norm{\psi}_{L^4(\R^2)}^4=gh\omega_{\perp}\norm{u}_{L^4(\R^2)}^4.
$$
We also note  that the dimension of
$h^{-1/2}m^{1/2}\omega_{\perp}^{1/2}$ is 1/length, so that
$$
x_{1}=h^{-1/2}m^{1/2}\omega_{\perp}^{1/2} x,\quad
x_{2}=h^{-1/2}m^{1/2}\omega_{\perp}^{1/2} y,\quad
u(x_{1},x_{2})\quad\text{are dimensionless.}
$$
Assuming $\omega_{x}^2\le \omega_{y}^2$, we use the dimensionless parameter
$\nu$ to write
$$\omega_{x}^2=(1-\nu^2)\omega_{\perp}^2,\quad
\omega_{y}^2=(1+\nu^2)\omega_{\perp}^2,
$$
and we get immediately
\begin{multline*}
\frac 1{h\omega_{\perp}}\mathcal E(\psi)=
\frac12\norm{D_{1}u}_{L^2(\R^2)}^2+\frac12\norm{D_{2}u}_{L^2(\R^2)}^2
+\frac12(1-\nu^2) \norm{x_{1}u}_{L^2(\R^2)}^2 +\frac12(1+\nu^2)
\norm{x_{2}u}_{L^2(\R^2)}^2
\\-\frac{\Omega}{\omega_{\perp}}
\poscal{(x_{1}D_{2}-x_{2}D_{1})u}{u}_{L^2(\R^2)} +\frac g2
\norm{u}_{L^4(\R^2)}^4.
\end{multline*}
Finally, we have
\begin{gather}
\frac 1{h\omega_{\perp}}\mathcal E(\psi):= E_{GP}(u)=
\poscal{Hu}{u}+\frac g2 \norm{u}_{L^4(\R^2)}^4,\label{1.gpen}
\\
2H= D_{1}^2+D_{2}^2+(1-\nu^2) x_{1}^2+(1+\nu^2) x_{2}^2-2\omega
(x_{1}D_{2}-x_{2}D_{1}),\quad \omega=\frac{\Omega}{\omega_{\perp}},
\label{1.hami}
\end{gather}
where  $\omega, \nu, u, g$ are all dimensionless and
$\norm{u}_{L^2(\R^2)}=1$. The minimization of this functional is the mathematical problem that we
address in this paper. The Euler-Lagrange equation for the
minimization of $ E_{GP}(u)$, under the constraint
 $\norm{u}_{L^2(\R^2)}=1$,
 is
 \begin{equation}\label{1.eule}
H u+g \val {u}^2 u=\lambda u,
\end{equation}
where $\lambda$ is the Lagrange multiplier. We shall always assume
that $ \Omega^2\le \omega_{x}^2, $ i.e. $\omega^2+\nu^2\le 1$ and
define the dimensionless parameter $\ep$ by
\begin{equation}\label{1.thre} \omega^2+\nu^2+\ep^2=1.
\end{equation} The fast
rotation regime occurs when the ratio $\Omega^2/\omega_{x}^2$ tends
to $1_{-}$, i.e. $\ep$ tends to 0.\par\noindent
{\it {\bf Summarizing and reformulating
our reduction,} we have
\begin{equation}\label{1.engrpi}
E_{GP}(u)=\frac12\poscal{q_{\omega,\nu,\ep}^w
u}{u}_{L^2(\R^2)}+\frac g2\int_{\R^2} \val u^4 dx,
\end{equation}
where $q_{\omega,\nu, \ep}$
 is the quadratic form
\begin{equation}\label{1.fetter}
q_{\omega,\nu,\ep}(x_{1},x_{2},\xi_{1},\xi_{2})=
\xi_{1}^2+\xi_{2}^2+(1-\nu^2) x_{1}^2+ (1+\nu^2)
x_{2}^2-2\omega(x_{1}\xi_{2}-x_{2}\xi_{1}),
\end{equation}
which depends on the real parameters $\omega, \nu, \ep$ such
that\footnote{Of course there is no loss of generality assuming that
$\epsilon,\nu$ are nonnegative parameters; we may also assume that
$\omega\ge 0$, since the change of function $u(x_{1},x_{2})\mapsto
u(-x_{1},x_{2})$ preserves the $L^4$-norm, is unitary in $L^2$,
corresponding to the symplectic transformation $(x_{1},
x_{2},\xi_{1},\xi_{2})\mapsto (-x_{1}, x_{2},-\xi_{1},\xi_{2})$ and
leads to the same problem where $\omega$ is replaced by $-\omega$.}
 \eqref{1.thre} holds.
Here  $q^w_{\omega,\nu,\ep}$  is the  operator  with Weyl  symbol
$q_{\omega,\nu,\ep},$ that is:
\begin{equation}\label{1.fetter1}
q^w_{\omega,\nu,\ep}=D_{1}^2+D_{2}^2+(1-\nu^2) x_{1}^2+ (1+\nu^2)
x_{2}^2-2\omega(x_{1}D_{2}-x_{2}D_{1}),
\end{equation}
where $D_j = \partial_j / (2i\pi).$ We would like  to minimize the
energy $E_{GP}(u)$ under the constraint $\norm{u}_{L^2}=1$ and
understand what is happening when  $\ep\rightarrow 0$.}
\subsection{The isotropic Lowest Landau Level}
When the harmonic trap is isotropic, i.e. when $\nu=0$, it turns out
that, since $\omega^2+\ep^2=1$,
\begin{equation}\label{1.firdia}
q = q_{\omega,0,\ep} = (\xi_{1}+\omega x_{2})^2+ (\xi_{2}-\omega x_{1})^2
+\ep^2(x_{1}^2+x_{2}^2)
\end{equation}
so that
$$
E_{GP}(u)=
\frac12\norm{(D_{1}+\omega x_{2})\psi+i(D_{2}-\omega x_{1}) u}^2
+\frac{\omega}{2\pi}\norm{u}^2+\frac{\ep^2}{2}\norm{ \val x u}^2+\frac{g}{2}\int \val u^4 dx.
$$
We note that, with $z=x_{1}+ix_{2}$,
$$
D_{1}+\omega x_{2}+i(D_{2}-\omega x_{1})=\frac{1}{i\pi}\bar\p-i\omega z=\frac{1}{i\pi}(\bar \p+\pi\omega z),
$$
hence the first term of the energy is minimized (and equal to $0$) if $u\in
LLL_{\omega^{-1}}$, where
\begin{equation}\label{1.isolll}
LLL_{\omega^{-1}}=\{u\in L^2(\R^2), u(x)= f(z) e^{-\pi \omega\val z^2}\}=\ker(\bar \p+\pi\omega z)\cap L^2(\R^2),
\end{equation}
with $f$ holomorphic. We expect the condensate to have a large
expansion, hence the term $\int |u|^4$ to be small. Thus, it is
natural to minimize the energy $E_{GP}$ in $LLL_{\omega^{-1}}.$ It
has been proved in \cite{aft-bl} that the restriction to $
LLL_{\omega^{-1}}$ is a good approximation of the original problem,
i.e. the minimization of $E_{GP}$ in $L^2(\R^2)$.
We get for $u\in
LLL_{\omega^{-1}}, \norm{u}_{L^2}=1,$
$$
E_{GP}(u)=
\frac12\norm{\underbrace{(D_{1}+\omega x_{2})u+i(D_{2}-\omega x_{1})u}_{(i\pi)^{-1}(\bar\p+\pi\omega z)u=0}}^2
+\frac{\omega}{2\pi}+\frac{\ep^2}{2}\norm{ \val x u}^2+\frac g2\int \val u^4 dx,
$$
and with $u(x)= \upsilon\bigl((\omega\ep)^{1/2} x\bigr)(\omega\ep)^{1/2}$
(unitary change in $L^2(\R^2)$),
$$
E_{GP}(u)=
\frac{\omega}{2\pi}+\frac{\ep}{2\omega}\left(
\int \val y^2\val{\upsilon(y)}^2dy
+{\omega^2  g}\int {\val{ \upsilon(y)} }^4dy\right).$$
The minimization problem of $E_{GP}(u)$
in the space $LLL_{\omega^{-1}}$ is thus reduced to study
\begin{equation}\label{enabn}
E_{LLL}(\upsilon)=\norm{\val x \upsilon}^2_{L^2}+ \omega^2 g\norm{\upsilon}^4_{L^4},
 \quad
\upsilon\in LLL_{\ep},
\end{equation}
i.e. with $z=x_{1}+ix_{2}$, $ v(x_{1}, x_{2})= f(z) e^{-\pi
\ep^{-1}\val z^2}, $ $f$ entire (and $v\in L^2(\R^2)$). This program has
been carried out in the paper \cite{MR2271933} by A. Aftalion, X.
Blanc, F. Nier.
In the isotropic case, a key point is the fact that the symplectic
diagonalisation of the quadratic Hamiltonian is rather simple: in
fact revisiting the formula \eqref{1.firdia}, we obtain easily
\begin{multline}\label{1.isodia}
q= \overbrace{(\frac{1-\omega}{2})(\xi_{1}-x_{2})^2}^{\eta_{1}^2}
+\overbrace{(\frac{1-\omega}{2})(\xi_{2}+x_{1})^2}^{\mu_{1}^2
y_{1}^2}
\\+\underbrace{(\frac{1+\omega}{2})(\xi_{1}+x_{2})^2}_{\eta_{2}^2}
+\underbrace{(\frac{1+\omega}{2})(\xi_{2}-x_{1})^2}_{\mu_{2}^2y_{2}^2},
\end{multline}
with
{\small\begin{equation}
\label{1.isosym}\left\{
\begin{matrix}
\eta_{1}=2^{-1/2}(1-\omega)^{1/2}(\xi_{1}-x_{2}),& \mu_{1}=1-\omega,
&y_{1}=
2^{-1/2}(1-\omega)^{-1/2}(\xi_{2}+x_{1}),\\ \\
\eta_{2}=2^{-1/2}(1+\omega)^{1/2}(\xi_{1}+x_{2}),& \mu_{2}=1+\omega,
&y_{2}= 2^{-1/2}(1+\omega)^{-1/2}(x_{1}-\xi_{2}),
\end{matrix}\right.
\end{equation}}
so that  the linear forms $(y_{1},y_{2},\eta_{1}, \eta_{2})$ are
symplectic coordinates in $\R^4$, i.e.
$$
\poi{\eta_{1}}{y_{1}}=\poi{\eta_{2}}{y_{2}}=1, \quad
\poi{\eta_{1}}{\eta_{2}}=\poi{\eta_{1}}{y_{2}}=\poi{\eta_{2}}{y_{1}}=\poi{y_{1}}{y_{2}}=0.
$$
In \cite{MR2271933}, an upper bound for the energy is constructed with a
test function which is also an ``almost" solution to the Euler-Lagrange equation
corresponding to the minimization of \eqref{enabn} in
$LLL_{\ep}$. This almost solution displays a triangular
vortex lattice in a central region of the condensate and is
constructed using a Jacobi Theta function, which is modulated by an inverted
parabola profile and projected onto $LLL_\ep$.
\subsection{Sketch of some preliminary reductions in the anisotropic case}
The analysis of the reduced energy in the anisotropic case
yields two different  situations: one is  similar to the isotropic
case and the other one  is quite different, without vortices.
To tackle the non-isotropic case where $\nu>0$ in \eqref{1.fetter1},
 one would like to determine a space playing the role of
 the $LLL$ and taking into account the anisotropy.
\subsubsection*{Step 1. Symplectic reduction of the quadratic form $q_{\omega,\nu,\epsilon}$.}
Given the quadratic form $q_{\omega,\nu,\ep}$ \eqref{1.fetter},
identified with a $4\times 4$ symmetric matrix, we define its
fundamental matrix by the identity
$F=-\sigma^{-1}q_{\omega,\nu,\ep}=\sigma q_{\omega,\nu,\ep}$ where
$$\sigma=\begin{pmatrix}
0&I_{2}\\
-I_{2}&0
\end{pmatrix}\quad\text{ is the symplectic matrix given in $2\times 2$ blocks.}
$$ The properties of the eigenvalues
and eigenvectors of $F$ allow to find a symplectic reduction for $q_{\omega,\nu,\ep}$.
\subsubsection*{Step 2. Determination of the anisotropic $LLL$.}
The anisotropic
equivalent of the $LLL$
can be determined explicitely,
thanks to the results of the first step.
We find that it is the subspace
of functions $u$ of $L^2(\R^2)$ such that
 \begin{equation*}
f\bigl(x_{1}+i\beta_{2}x_{2}\bigr)
\exp{\left(-\frac{\gamma\pi}{4\beta_{2}}\Bigl[
x_{1}^2(1-\frac{\nu^2}{2\alpha})+
(\beta_{2}x_{2})^2(1+\frac{\nu^2}{2\alpha})\Bigr]
\right)}
\exp{(-i\frac{\pi\nu^2\gamma }{4\alpha }x_{1}x_{2})},
\end{equation*} where $f$ is entire.
The positive parameters $\alpha,\gamma,\beta_{2}$ are defined in the
text and
are explicitely known in terms of $\omega, \nu$. We also
 determine an
operator $M$, which  can be used to give an explicit expression for
the isomorphism between $L^2(\R)$ and the anisotropic $LLL$ as well as
to express the Gross-Pitaevskii energy in the new symplectic coordinates.
\subsubsection*{Step 3. Rescaling.}
Introducing a new set of parameters
($\omega, \nu,\epsilon$ are positive satisfying
\eqref{1.thre}, $g>0$ given by \eqref{1.constg}),
 \begin{equation}\label{1.kappa2}
\kappa_{1}^2=(2\nu^2+\epsilon^2)\bigl(1+\frac{2\nu^2}{\alpha-\nu^2+\omega^2}\bigr),\quad \alpha= \sqrt{\nu^4+ 4\omega^2},
\quad
g_{1}=g\frac{\alpha+2\omega^2+\nu^2}{2\alpha},
\end{equation}
\begin{equation}\label{1.morek}
\kappa=\frac{\kappa_1}{\beta_2},\
g_0=\frac{g_1\gamma^2}{4\beta_2},
\
\gamma=\frac{2\alpha}{\omega},
\beta_{2}=\frac{2\omega \mu_{2}}{\alpha+2\omega^2+\nu^2},
\quad
\mu_{2}=1+\omega^2+\alpha,
\end{equation}
we show that,
after some rescaling, the minimization of
the full energy $E_{GP}(u)$
of \eqref{1.engrpi}
can be reduced to the minimization of\ \beq
 \label{eq:nrjLLL2}
 {E}(u)=\int_{\mathbb R ^2} \frac12 (\ep^2 x_1^2+ \kappa^2 x_2^2 )|u|^2
 +\frac {g_{0}}2
  |u|^4.\eeq
   on the
 space  \beq
 \label{lll} \Lambda_{0}=\{ u\in L^2(\R^2),\ u(x_1,x_2)= f(z)e^{-\pi |z|^2/2},
\ f \hbox{ holomorphic, } z=x_1+ix_2\}.\eeq
The point is that,
after some scaling, we are able to come back to an isotropic space.
The orthogonal projection $\Pi_{0}$
of $L^2(\R^2)$ onto $\Lambda_{0}$ is explicit and simple:
\begin{equation}
  \label{eq:projLLL}
 (\Pi_{{0}} u)(x) =  \int_{\R^2} e^{-\frac{\pi}2 |x-y|^2 + i\pi
    \left(x_2 y_1 - y_2 x_1\right)} u(y) dy.
\end{equation}
We  are thus reduced to
the following problem: with $E(u)$ given by
\eqref{eq:nrjLLL2},
study
\begin{equation}
  \label{eq:minLLL}
  I(\ep,\kappa) = \inf\bigl\{ E(u), \ u \in \Lambda_{0}, \
   \norm{u}_{L^2(\R^2)} = 1 \bigr\}.
\end{equation}
The minimization of $E$ without the holomorphy constraint
yields \begin{equation}\label{invpar}|u|^2 = \frac{2}{\pi R_1 R_2}(1
- \frac{x_1^2}{R_1^2} -
  \frac{x_2^2}{R_2^2} ),\hbox{
where }R_1 =\left(\frac{4g_0 \kappa}{\pi\ep^3}\right)^{1/4},\quad
R_2 = \left(\frac{4g_0 \ep}{\pi
\kappa^3}\right)^{1/4}.\end{equation} As $\ep$ tends to 0, $R_1$
always tends to infinity
(in fact $R_{1}\gtrsim \ep^{-1/2}$), but the behaviour of $R_2$ depends on the
respective values of $\ep$ and $\kappa$, that is of $\ep$
   and  $\nu$.
   \subsubsection*{Step 4. Sorting out the various regimes.}
   Recalling that the positive parameter $\nu$ stands for the anisotropy,
   we find
two  regimes:
\par
$\bullet$ $\nu\ll \ep^{1/3}$ (weak anisotropy): $R_2\to \infty$
(in fact, $R_{2}^{4/3}\approx \min(\ep^{-2/3}, \ep^{1/3}\nu^{-1})$). Numerical simulations ({\sc Figure} \ref{figanv})
 show a triangular vortex lattice. The behaviour is similar to the isotropic case except that the
inverted parabola profile (\ref{invpar}) takes into account the
anisotropy. We will construct  an approximate minimizer.
\begin{figure}[h]
\includegraphics[width=1\textwidth]{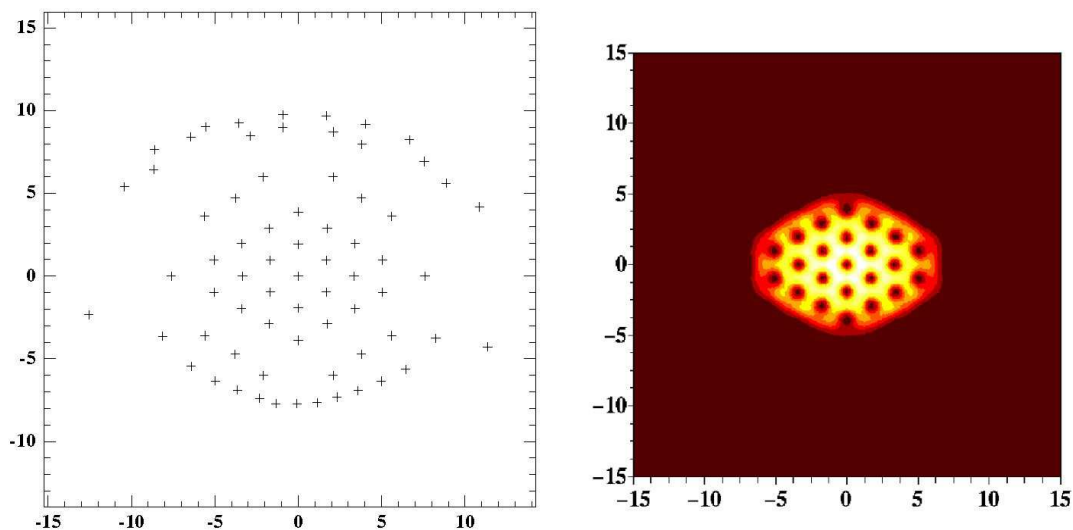}
\caption{\small Plot of the zeroes of the
 minimizer (left) and the density (right) for $\ep^2=0.002$, $\nu=0.03$. Triangular vortex
 lattice in an anisotropic trap.}\label{figanv}
\end{figure}
\par
$\bullet$  $\nu\gg \ep^{1/3}$ (strong anisotropy): $R_2\to 0$ (in fact
$R_{2}^{4/3}\approx \ep^{1/3}\nu^{-1}$). Numerical simulations
({\sc Figure} \ref{figannv})
 show that there are no vortices in the bulk, the behaviour is an inverted parabola
in the $x_1$ direction and a fixed Gaussian in the $x_2$ direction.
 Thus, the size of the condensate does not shrink in the $x_2$
 direction and (\ref{invpar}) is not a good approximation of the
 minimizer.
The shrinking of the condensate in the $x_2$ direction is not
allowed in $\Lambda_{0}$ (see \eqref{lll})
 because the operator $x_2^2$ is
bounded from below in that space by a positive constant
 and the first eigenfunction is  a Gaussian in the
$x_2$ direction.
 We find an asymptotic 1D problem  (upper and
lower bounds match) which yields a separation of variables.
\begin{figure}[h]
\includegraphics[width=1\textwidth]{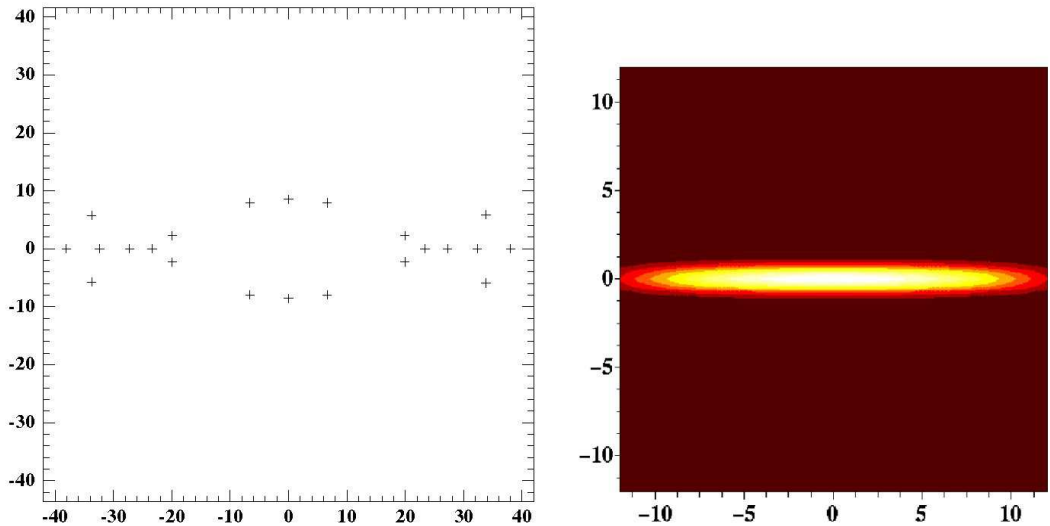}
\caption{\small Plot of the zeroes of the
 minimizer (left) and the density (right) for $\ep^2=0.002$, $\nu= 0.73$.  No vortex
 in the visible region.}\label{figannv}\end{figure}
 \eject
 \subsection{Main results}
\subsubsection{Weakly anisotropic case}
In a first step\footnote{\label{footk}We shall see that $\kappa\approx \nu+\ep$
in the sense that
the ratio ${\kappa}/{(\nu+\ep)}$ is bounded above and below by some fixed positive constants, so that the weakly anisotropic case is indeed $\nu\ll \ep^{1/3}$.}, we assume that,
with $\kappa$ given by
\eqref{1.morek},  \begin{equation}
  \label{eq:faible}
 {\ep \leq \kappa \ll \ep^{1/3}}.
\end{equation}  The isotropic case is recovered by assuming $\kappa = \ep.$
 This case  is similar  to  the isotropic  case and  we
derive similar results to  the paper \cite{MR2271933}, namely an
upper
 bound given by the Theta function but we lack a good lower bound.
\par
  We recall that the Jacobi Theta function
$\Theta(z,\tau)$
associated to a lattice $\Z\oplus \Z\tau$ is a holomorphic
function which vanishes exactly once in any lattice cell and is
defined by
 \begin{equation}\label{thetaintro}\Theta(z,\tau) =
\frac{1}{i}\sum_{n=-\infty}^{+\infty}(-1)^{n}e^{i\pi\tau(n+1/2)^{2}}
e^{(2n+1)\pi iz},\quad z\in \C\,.\end{equation}  This function
allows us to construct a periodic function on the same lattice:
$u_\tau$ is defined by\begin{equation}\label{eq:utau}
u_\tau(x_1,x_2) = e^{\frac{\pi}2 \left(z^2 - |z|^2\right)}
\Theta\left
  (\sqrt{\tau_I}z, \tau\right), \quad z = x_1 + ix_2, \quad \tau =
\tau_R + i\tau_I,
\end{equation}
$|u_\tau|$ is periodic over the lattice $\Z\oplus \tau\Z$, and
$u_\tau$ satisfies
\begin{equation}
      \label{eq:AbrikEulLag}
    \Pi_{0} \left(|u_{\tau}|^2 u_{\tau} \right)
    = \lambda_{\tau} u_{\tau},
\end{equation}
with
\begin{equation}
  \label{eq:Abrikmult}
  \lambda_{\tau} =
  \frac{\fint |u_{\tau}|^{4}}{\fint |u_{\tau}|^{2}} =
\frac{\gamma(\tau)}{\sqrt{2\tau_I}},
\end{equation}
and
\begin{equation}
  \label{eq:dfgamma}
  \gamma(\tau) := \frac{\fint |u_\tau|^4}{\left(\fint |u_\tau|^2\right)^2}.
\end{equation}
 The minimization of $\gamma(\tau)$ on all possible
$\tau$ corresponds to the Abrikosov problem. It turns out that the
properties of the Theta function allow to derive that  $$\gamma(\tau)
= \sum_{(j,k)\in \Z^2} e^{-\frac{\pi}{\tau_I} |j\tau - k|^2}$$ and
prove (see \cite{MR2271933}) that $\tau \mapsto \gamma(\tau)$ is
minimized for $\tau = j=e^{2i\pi/3}$, which corresponds to the
hexagonal lattice. The minimum is
\begin{equation}\label{gammadej}
b=\gamma (j)\approx 1.1596.
\end{equation}
The function $u_\tau$ allows us to construct the vortex lattice and
we multiply it by the proper inverted parabola to get a good upper
bound:
\begin{theorem}\label{th:cv-faible}
 We have
   for $I(\ep,\kappa)$ defined in \eqref{eq:minLLL},
   $b$ given in \eqref{gammadej}, $\kappa$ in \eqref{1.morek},
  \begin{equation}
    \label{eq:limnrjfaible}
    \frac23 \sqrt{\frac{2g_{0}\ep \kappa}{\pi}} < I(\ep,\kappa) \leq \frac23
    \sqrt{\frac{2g_{0}b\ep \kappa}{\pi}} + O\left(\sqrt{\ep\kappa}
        \left(\frac{\kappa^3}{\ep}\right) ^{1/8}\right),
  \end{equation}
  when $(\epsilon,\kappa\epsilon^{-1/3}) \to (0,0)$. Moreover, the following function provides the
upper bound:
\begin{equation}
  \label{eq:ansatz}
  v = \Pi_{0}\left( u_\tau \rho \right),
\end{equation}
where $u_\tau$ is defined by $(\ref{eq:utau})$ with $
 \tau =e^{\frac{2i\pi}3}$ and
$$\rho(x)^2 = \frac{2}{\pi\sqrt b R_1 R_2} \left(1-
  \frac{x_1^2}{\sqrt b R_1^2}
- \frac{x_2^2}{ \sqrt b R_2^2}\right)_+, \quad R_1 =
\left(\frac{4g_0\kappa}{\pi\ep^3}\right)^{1/4}, \ R_2 =
\left(\frac{4g_0\ep}{\pi\kappa^3}\right)^{1/4}.$$
\end{theorem}We expect $v$ to be a good approximation
 of the minimizer and the energy asymptotics to match the right-hand side of
  \eqref{eq:limnrjfaible}. Thus, the lower bound is not optimal
  ( it
 does not include $b$).
  In addition, the test function \eqref{eq:ansatz} (with a general $\tau
\neq j$ a priori) gives the upper bound of
\eqref{eq:limnrjfaible} with $\gamma(\tau)$ instead of $b$.
The proof is a refinement of that in \cite{MR2271933}.
\subsubsection{Strong anisotropy}
In the case where the rotation is fast enough
 in the sense that
 \begin{equation}
  \label{eq:fort}
  {\kappa \gg \ep^{1/3}}
\end{equation} we have found a regime  unknown by physicists where
 vortices disappear and the problem can be reduced in fact to a 1D energy.
\begin{theorem}
  \label{th:cv-fort}
  For $I(\ep,\kappa)$ defined in \eqref{eq:minLLL},
   $b$ given in \eqref{gammadej}, $\kappa$ in \eqref{1.morek},
   we have
\begin{equation}
  \label{eq:nrjlimfort}
\lim_{ (\epsilon,\epsilon^{1/3}\kappa^{-1}) \to (0,0)}\left( \frac{I(\ep,\kappa) - \frac{\kappa^2}{8\pi}}{\ep^{2/3}} \right) =
  J,
\end{equation}
where
\begin{equation}
  \label{eq:pb1d}
  J = \inf\bigl\{\int\frac12 t^2 p(t)^2 + \frac{g_0}{2}\int
    p(t)^4, \ p \text{  real-valued }\in L^2(\R)\cap L^4(\R),
    \norm{p}_{L^2(\R)}=1
    \bigr\}.
\end{equation}
In addition, if $u$ is a minimizer of $I(\ep,\kappa)$, then
\begin{equation}
  \label{eq:ulimfort}
  \frac 1 {\ep^{1/3}} \left| u\left(\frac{x_1}{\ep^{2/3}},x_2\right) \right| \longrightarrow 2^{1/4} e^{-\pi x_2^2} p(x_1) ,
\end{equation}
in $L^2(\R^2) \cap L^4(\R^2)$, where $p$ is the minimizer of $J.$
\end{theorem}
Note that the minimizer $p$ of \eqref{eq:pb1d} is explicit:
$$p(t)^2 = \frac{3}{4R} \left(1-\frac{t^2}{R^2} \right)_+ , \quad R =
\left( \frac{3g_{0}}{2}\right)^{1/3}.$$
A few words about the proof of Theorem~\ref{th:cv-fort}.
The
first point is that the operator $\Pi_{0}x_2^2\Pi_{0}$ (see \eqref{lll}, \eqref{eq:projLLL})
is bounded
from below by a positive constant:
$$\forall u \in \Lambda_0, \quad \int_{\R^2} x_2^2 |u|^2 \geq \frac1{4\pi}
\int_{\R^2} |u|^2.$$ This is proven in Lemma~\ref{lm:b_inf_y2} below.
Actually, the spectrum of this operator is purely continuous, and
any Weyl sequence associated with the value $1/(4\pi)$ converges
(up to renormalization) to the function
\begin{equation}\label{eq:lim}
u_0(x_1,x_2) = \exp\left(-\pi x_2^2 + i\pi x_1x_2\right),
\end{equation}
which satisfies the equation $\Pi_{0}(u_0) = \frac 1 {4\pi} u_0.$
This gives the lower bound
$$I(\ep, \kappa) \geq \frac{\kappa^2}{8 \pi},$$
and indicates that in order to be close to this lower bound, a test
function should be close to \eqref{eq:lim}. Thus, the second point
is to construct a test function having the same behaviour as
\eqref{eq:lim} in $x_2$, and  a large extension in $x_1$. This is
done by using the function
$$u_1(x_1,x_2) = \frac 1 {2^{1/4}} e^{-\frac{\pi}2 x_2^2}
\int_{\R}e^{-\frac{\pi}{2} \left((x_1-y_1)^2 - 2iy_1x_2 \right)}
\rho(y_1) dy_1,$$ which is equal to $\Pi_{0}(\rho(x_1)
\delta_0(x_2))$, where $\delta_0$ is the Dirac delta function and
$\rho$ any real-valued function of one variable. This test function is then
proved to be close to $2^{1/4}e^{-\pi x_2^2} \rho(x_1),$ which
allows to compute its energy, and gives the upper bound, provided
that $\rho(t) = \ep^{1/3}p(\ep^{2/3}t)$, where $p$ is the minimizer
of \eqref{eq:pb1d}. Finally, in order to prove the lower bound, we
first extract bounds on the minimizer from the energy, which allow
to pass to the limit in the equation (after rescaling as in
\eqref{eq:ulimfort}), hence prove that the limit is the right-hand
side of \eqref{eq:ulimfort}. This uses the fact that the energy
appearing in \eqref{eq:pb1d} is strictly convex, hence that any
critical point is the unique minimizer.
\par
The paper is organized as follows: in section 2, we  review some
standard facts on positive definite quadratic forms in a symplectic
space. This allows us, in section 3, to construct a symplectic
mapping $\chi$, which yields a simplification of the quadratic form
$q$.  In section 4,
quantizing that symplectic mapping
in a metaplectic transformation,
we find the expression of the $LLL$ and manage
to reach the reduced form of the energy (Proposition 4.5).
 Section 5 is devoted to the proof of Theorem 1.1 and section 6 to
 Theorem 1.2.
\subsubsection*{Open questions}We have no information on the
the intermediate regime where, for instance, $\ep^{1/3}/ \kappa$ converges to some
constant $R_{0}^{4/3}$
(in that case, $R_{1}\approx \ep^{-2/3}, R_{2}\approx R_{0}$).
 We expect that the extension in the $x_2$ direction depends on $R_{0}$ and wonder whether
  the condensate has a finite number of vortex lines. We have not
  determined the limiting problem.
 \vs
{\bf Acknowledgements.}
We would like to thank A.L.Fetter and J.Dalibard for
very useful comments on the physics of the problem. We also acknowledge
support from the French ministry grant {\it ANR-BLAN-0238, VoLQuan} and
express our gratitude to our colleagues  participating to this
ANR-project, in particular T.Jolic\oe ur and S.Ouvry.
\section{Quadratic Hamiltonians}
We first review some standard facts on positive definite
quadratic forms in a symplectic space.
\subsection{On positive definite quadratic forms on symplectic spaces}
We consider the phase space
$\R^n_{x}\times \R^n_{\xi}$, equipped with its canonical symplectic structure: the symplectic form $\sigma$ is a bilinear alternate form on $\RZ$  given by
\begin{gather}\label{2.symfor}
\sigma\bigl((x,\xi);(y,\eta)\bigr)=\xi\cdot y-\eta\cdot x=\poscal{\sigma X}{Y},\\
\text{with\quad}
{X=\begin{pmatrix}x\\\xi\end{pmatrix}, Y=\begin{pmatrix}y\\\eta\end{pmatrix},\sigma=
\begin{pmatrix}
0&I_{n}\\
-I_{n}&0
\end{pmatrix},}
\end{gather}
where the form $\sigma$ is identified with
the $2n\times 2n$ matrix above
given in $n\times n$ blocks.
The symplectic group $Sp(n)$ (a subgroup of $Sl(2n,\R)$), is defined by the equation
on the $2n\times 2n$ matrix $\chi$,
\begin{equation}
\chi^* \sigma\chi =\sigma,\quad\text{
i.e. $\forall X,Y\in \RZ,\ \poscal{\sigma \chi X}{\chi Y}
=\poscal{\sigma X}{Y}$}.
\end{equation}
The following lemma is classical
(see e.g. the chapter XXI in \cite{MR2304165}, or \cite{MR1698616}).
\begin{lem} \label{2.lem.symgen}
Let $B\in GL(n,\R)$ and let $A, C$ be  $n\times n$ real symmetric matrices.  \
Then the matrix $\Xi$, given by $n\times n$
blocks
\begin{equation}\label{2.symgen}
\Xi_{A,B,C}=\mat22{B^{-1}}{-B^{-1}C}{AB^{-1}}{B^*-AB^{-1}C}
=
\mat22{I}{0}{A}{I}
\mat22{B^{-1}}{0}{0}{B^*}
\mat22{I}{-C}{0}{I}
\end{equation}
belongs to $Sp(n)$.
Any element of $Sp(n)$
can be written
as a product $$\Xi_{A_{1},B_{1},C_{1}}\Xi_{A_{2},B_{2},C_{2}}.$$
\end{lem}
\begin{nb} {\rm\small The first statement is easy to verify directly
and we shall not use the last statement, which is nevertheless an interesting piece of information.
For a symplectic mapping $\Xi$, to be of the form above
is equivalent
to the assumption that the mapping
$x\mapsto pr_{1}\Xi(x\oplus 0)$
is invertible from $\R^n$ to $\R^n$.}\end{nb}
Given a quadratic form $Q$ on $\RZ$, identified with a symmetric
$2n\times 2n$ matrix,
we define its {\it fundamental matrix} $F$ by the identity
$$
F=-\sigma^{-1} Q=\sigma Q,\quad\text{so that for $X,Y\in \RZ$ }\quad
\poscal{\sigma Y}{FX}=\poscal{QY}{X}.
$$
The following proposition is classical
(see e.g.  the theorem
21.5.3 in \cite{MR2304165}).
\begin{pro}\label{2.pro.basdef}
Let $Q$ be  a positive definite quadratic form on the symplectic $\R^n_{x}\times\R^n_{\xi}$.
One can find $\chi\in Sp(n)$ such that with
$$\RZ\ni X=\chi Y, \quad
Y=(y_{1},\dots,y_{n},\eta_{1},\dots,\eta_{n}),$$
$$
\poscal{QX}{X}=\poscal{Q\chi Y}{\chi Y}
=\sum_{1\le j\le n}(\eta_{j}^2+\mu_{j}^2y_{j}^2),\quad \mu_{j}>0.
$$
The $\{\pm i\mu_{j}\}_{1\le j\le n}$ are the
$2n$
eigenvalues of the fundamental matrix,
related to the $2n$
eigenvectors
$\{e_{j}\pm i \ep_{j}\}_{1\le j\le n}.
$
The $\{e_{j}, \ep_{j}\}_{1\le j\le n}$
make a symplectic basis of  $\RZ$:
$$
\sigma(\ep_{j}, e_{k})=\delta_{j,k},\quad \sigma(\ep_{j}, \ep_{k})=
\sigma(e_{j}, e_{k})=0,
$$
and the symplectic planes $\Pi_{j}=\R e_{j}\oplus \R \ep_{j}$
are orthogonal for $Q$.
\end{pro}
\begin{nb}{\rm \small
A one-line-proof of these classical facts:
on $\C^{2n}$ equipped with the dot-product  given by $Q$, diagonalize the sesquilinear Hermitian form $i\sigma$.}
\end{nb}
\subsection{Generating functions}
We define on $\R^n\times \R^n$ the {\it generating function} $S$
of the symplectic mapping of the form $\Xi_{A,B,C}$
given in the lemma \ref{2.lem.symgen}
by the identity
\begin{equation}\label{4.genfun}
S(x,\eta)=\frac12\bigl(\poscal{Ax}{x}+2\poscal{Bx}{\eta}
+\poscal{C\eta}{\eta}
\bigr).
\end{equation}
We have
\begin{equation}\label{2.genfin}
\Xi_{A,B,C}\underbrace{\Bigl(\frac{\p S}{\p\eta}, \eta\Bigr)}_{\in \R^n\times \R^n}=
\underbrace{\Bigl(x,\frac{\p S}{\p x}\Bigr)}_{\in \R^n\times \R^n}.
\end{equation}
In fact, we see directly
$$
\mat22{I}{0}{A}{I}
\mat22{B^{-1}}{0}{0}{B^*}
\mat22{I}{-C}{0}{I}
\begin{pmatrix}
Bx+C\eta\\\eta
\end{pmatrix}
=\mat22{I}{0}{A}{I}
\begin{pmatrix}
x\\B^*\eta
\end{pmatrix}
=
\begin{pmatrix}
x\\Ax+B^*\eta
\end{pmatrix}.
$$
Given a positive definite quadratic form
$Q$ on $\RZ$, identified with a symmetric $2n\times 2n$
matrix,
we know from the proposition
\ref{2.pro.basdef}
that there exists $\chi \in Sp(n)$ such that
$$
\chi^* Q \chi=
\begin{pmatrix}\mu^2&0\\0&I_{n}
\end{pmatrix},\quad \mu^{2}=\diag(\mu_{1}^2,\dots,\mu_{n}^2)
$$
Looking for $\chi=\Xi_{A,B,C}$ given by a generating function $S$ as above,
we end-up
(using the notation $q(X)=\poscal{QX}{X}$ with $X\in \RZ$)
 with the equation
$$
q(\underbrace {x,\p_{x}S}_{\R^n\times\R^n})=
\norm{ \mu\p_{\eta} S}^2+\norm{\eta}^2,\qquad
 \mu\p_{\eta}  S=(\mu_{j}\p_{\eta_{j}}S)_{1\le j\le n}\in \R^n,
$$
where $\norm{\cdot}$ stands for the standard Euclidean norm on $\R^n$.
This means
\begin{equation}\label{2.initia}
q(x,  Ax+B^* \eta)=\norm{ \mu(Bx+C\eta)}^2+\norm{\eta}^2.
\end{equation}
We want now to go back to the study of our quadratic form \eqref{1.fetter}.
\subsection{Effective diagonalization}\label{sec.effective}
\begin{lem}\label{2.lem.effect}
Let $q$ be the quadratic form on $\R^4$
given by
\eqref{1.fetter}, where
$\omega, \nu, \ep$ are nonnegative  parameters such that
$\omega^2+\nu^2+\ep^2=1$.
The eigenvalues of the fundamental matrix are $\pm i\mu_{1},\pm i\mu_{2}$ with
\begin{align}\label{2.parone}
0&\le \mu_{1}^2=1+\omega^2-\alpha\le \mu_{2}^2=1+\omega^2+\alpha,\quad \alpha=\sqrt{\nu^4+4\omega^2},
\\\label{2.partwo}
\mu_{1}^2&= \frac{2\nu^2\ep^2+\ep^4}{\mu_{2}^2}.
\end{align}
In the isotropic case $\nu=0$, we recover
$\mu_{1}=1-\omega, \mu_{2}=1+\omega.$
When $\ep>0$, we have
$0<\mu_{1}^2\le \mu_{2}^2$
and $q$ is positive-definite.
When $\ep=0$,
we have $\mu_{1}=0<\mu_{2}$,
and
$q$ is positive semi-definite
with rank  $2$ if $\nu=0$ and with rank $3$ if $\nu>0$.
\end{lem}
\begin{proof}
The matrix $Q$ of  $q$ is thus
{\small \begin{equation}
\label{2.quamat}
{          Q=
\begin{pmatrix}
1-\nu^2&0&0&-\omega\\
0&1+\nu^2&\omega&0\\
0&\omega&1&0\\
-\omega&0&0&1
\end{pmatrix}},\text{\ and }
F=\sigma Q=
\begin{pmatrix}
0&\omega&1&0\\
-\omega&0&0&1\\
\nu^2-1&0&0&\omega\\
0&-\nu^2-1&-\omega&0
\end{pmatrix}.
\end{equation}}
The characteristic polynomial  $p$ of $F$ is easily seen to be even and we calculate
$$
p(\lambda)=\det (F-\lambda I_{4})= \lambda^4+2(1+\omega^2) \lambda^2+(1-\omega^2)^2-\nu^4
=(\lambda^2+1+\omega^2)^2-(\nu^4+4\omega^2).
$$
The four eigenvalues of $F$ are thus
$
\pm i\sqrt{1+\omega^2\pm\sqrt{\nu^4+4\omega^2} },
$
proving the first statement of
the lemma.
Since $(1+\omega^2)^2-\alpha^2=(1-\omega^2)^2-\nu^4=\ep^2 (2\nu^2+\ep^2)$,
we get
$
\mu_{1}^2={\ep^2 (2\nu^2+\ep^2)}/{\mu_{2}^2}.
$
The statements on the cases $\nu=0, \ep>0$ are now obvious.
When $\ep=0=\nu$, we have $\omega=1$, and  $\rank q=2$
as it is obvious on \eqref{1.isodia}.
When $\ep=0, \nu>0$, we consider the following minor determinant in $F$, cofactor of $f_{31}$
$$
\left\vert\begin{matrix}\omega&1&0\ \\ 0&0&1\ \\ -\nu^2-1&-\omega&0\ \end{matrix}\right\vert=(-1)(-\omega^2+\nu^2+1)=-2\nu^2\not=0,
$$
so that $\rank Q=\rank F=3$ in that case.
\end{proof}
\begin{nb} {\rm\small   We may note here that the condition $\omega^2+\nu^2\le 1$ is an iff condition on the real parameters $\nu, \omega$
for the quadratic form \eqref{1.fetter} to be positive semi-definite.
This is obvious on the expression \eqref{1.isodia}
in the isotropic case $\nu=0$,
and more generally,
the (non-symplectic) decomposition in independent linear forms
\begin{equation*}
q=(\xi_{1}+\omega x_{2})^2+(\xi_{2}-\omega x_{1})^2+x_{1}^2(1-\nu^2-\omega^2)
+x_{2}^2(1+\nu^2-\omega^2),
\end{equation*}
shows that $q$ has exactly one negative eigenvalue when
$\omega^2+\nu^2>1\ge \omega^2-\nu^2$,
and exactly
two negative eigenvalues when $\omega^2-\nu^2>1$.
As a result, when $\omega^2+\nu^2>1$, the operator $q^w$ is unbounded from below.
  }
\end{nb}
Using now the equations
\eqref{2.initia}, \eqref{1.fetter} and assuming that we may find a linear symplectic transformation given by a generating function \eqref{4.genfun},
we have to find
$A,B,C$ like in the lemma \ref{2.lem.symgen} with $n=2$,
so that for all $(x,\eta)\in \R^2\times \R^2$,
$$
\norm{Ax+B^* \eta}^2+\norm x^2+\nu^2(x_{2}^2-x_{1}^2)-2\omega
\bigl(x\wedge (Ax+B^* \eta)\bigr)
=\norm{ \mu(Bx+C\eta)}^2+\norm{\eta}^2,
$$
with
$
x\wedge \xi=x_{1}\xi_{2}-x_{2}\xi_{1}, \mu=\diag({\mu_{1},\mu_{2}}).
$
At this point, we see that the previous identity forces
some relationships between the matrices $A,B,C$.
However, the algebra is somewhat complicated
and assuming that $B$ is diagonal, $A, C$ are (symmetrical) with zeroes on the diagonal lead to some simplifications and to  the following results.
We introduce first some  parameters:
{\small\begin{align}
\beta_{1}&=\frac{2\omega \mu_{1}}{\alpha-2\omega^2+\nu^2}=\frac{\alpha-2\omega^2-\nu^2}{2\omega \mu_{1}}\  \text{\tiny              since $ (\alpha-2\omega^2)^2-\nu^4=4\omega^2+4\omega^4-4\omega^2\alpha=4\omega^2\mu_{1}^2$ },\\
\beta_{2}&=\frac{2\omega \mu_{2}}{\alpha+2\omega^2+\nu^2}
=\frac{\alpha+2\omega^2-\nu^2}{2\omega \mu_{2}}\
 \text{\tiny                          since $ (\alpha+2\omega^2)^2-\nu^4=4\omega^2+4\omega^4+4\omega^2\alpha=4\omega^2\mu_{2}^2$ },\label{formulebeta2}\\
&\gamma=\frac{2\alpha}{\omega},\label{formulegamma}\\
&\lambda_{1}^2=\frac{\mu_{1}}{\mu_{1}+\beta_{1}\beta_{2}\mu_{2}}=\frac{1}{1+\frac{\beta_{1}\beta_{2}\mu_{2}}{\mu_{1}}}=\frac{1}{1+\frac{\alpha+2\omega^2-\nu^2}{\alpha-2\omega^2+\nu^2}}=\frac{\alpha-2\omega^2+\nu^2}{2\alpha},\\
&\lambda_{2}^2=\frac{\mu_{2}}{\mu_{2}+\beta_{1}\beta_{2}\mu_{1}}
=\frac{1}{1+\frac{\beta_{1}\beta_{2}\mu_{1}}{\mu_{2}}}=
\frac{1}{1+\frac{\alpha-2\omega^2-\nu^2}{\alpha+2\omega^2+\nu^2}}=\frac{\alpha+2\omega^2+\nu^2}{2\alpha},
\end{align}
}and we have
\begin{align}
&\lambda_{1}^2+\lambda_{2}^2=1+\frac{\nu^2}{\alpha},\qquad
\lambda_{1}^2\lambda_{2}^2=\frac{(\alpha+\nu^2)^2-4\omega^4}{4\alpha^2}.
\end{align}
We define also
\begin{align}
&
d=\frac{\gamma \lambda_{1}\lambda_{2}}{2},\quad c=\frac{\lambda_{1}^2+\lambda_{2}^2}
{2\lambda_{1}\lambda_{2}}\quad
\text{ \footnotesize            which gives }\quad
cd=\frac{2\alpha(1+\nu^2/\alpha)}{4\omega}=\frac{\alpha+\nu^2}{2\omega}.
\end{align}
\begin{lem}\label{2.lem.calcul}
We define the $2\times 2$ matrices
$$
B=\begin{pmatrix}
\lambda_{1}^{-1}&0\\
0&\lambda_{2}^{-1}
\end{pmatrix},\quad
C=\begin{pmatrix}
0&d^{-1}\\
d^{-1}&0
\end{pmatrix},\quad
A=\begin{pmatrix}
0&\frac{d}{\lambda_{1}\lambda_{2}}-cd\\
\frac{d}{\lambda_{1}\lambda_{2}}-cd&0
\end{pmatrix}.
$$
The $4\times 4$
matrix given  with $2\times 2$ blocks by
$$\chi=\Xi_{A,B,C}  =\mat22{I_{2}}{0}{A}{I_{2}}
\mat22{B^{-1}}{0}{0}{B^*}
\mat22{I_{2}}{-C}{0}{I_{2}}
$$
belongs to $Sp(2)$
and
\begin{equation}\label{2.symmat}
\chi
=
\begin{pmatrix}
\lambda_{1}&0&0&-\frac{\lambda_{1}}{d}\\
0&\lambda_{2}&-\frac{\lambda_{2}}{d}&0\\
0&\frac{d}{\lambda_{1}}-\lambda_{2}cd&c\lambda_{2}&0\\
\frac{d}{\lambda_{2}}-\lambda_{1}cd&0&0&c\lambda_{1}\end{pmatrix},
\end{equation}
\begin{equation}
\chi^{-1}=
\begin{pmatrix}
c\lambda_{2}&0&0&\frac{\lambda_{2}}{d}\\
0&c\lambda_{1}&\frac{\lambda_{1}}{d}&0\\
0&-\frac{d}{\lambda_{2}}+\lambda_{1}cd&\lambda_{1}&0\\
-\frac{d}{\lambda_{1}}+\lambda_{2}cd&0&0&\lambda_{2}\end{pmatrix}.
\end{equation}
\end{lem}
\begin{proof}
The lemma \ref{2.lem.symgen}
gives that $\chi\in Sp(2)$
and we have also
$$
\chi^{-1}=\mat22{I_{2}}{C}{0}{I_{2}}
\mat22{B}{0}{0}{{B^*}^{-1}}
\mat22{I_{2}}{0}{-A}{I_{2}}.
$$
The remaining part of the proof depends
on the formula giving $\Xi_{A,B,C}$
in the lemma \ref{2.lem.symgen}
and a direct computation
whose verification is left to the reader.
\end{proof}
\begin{lem}\label{2.lem.keydiag}
Let $\chi$ be the symplectic matrix given by \eqref{2.symmat}
and $Q$ be the matrix given in \eqref{2.quamat}.
Then, with $\mu_{j}$ given by \eqref{2.parone}, we have
\begin{equation}
\chi^* Q \chi=\diag(\mu_{1}^2, \mu_{2}^2, 1,1).
\end{equation}
\end{lem}
The (tedious)
proof of that lemma is given in the appendix \ref{8.app.somec}.\par
Using the expression of $\chi^{-1}$ in \eqref{2.symmat},
defining
\begin{equation}\label{2.lastcomp}
\begin{pmatrix}
y_{1}\\ y_{2}\\\eta_{1}\\\eta_{2}
\end{pmatrix}
=\begin{pmatrix}
c\lambda_{2}&0&0&\frac{\lambda_{2}}{d}\\
0&c\lambda_{1}&\frac{\lambda_{1}}{d}&0\\
0&-\frac{d}{\lambda_{2}}+\lambda_{1}cd&\lambda_{1}&0\\
-\frac{d}{\lambda_{1}}+\lambda_{2}cd&0&0&\lambda_{2}\end{pmatrix}
\begin{pmatrix}
x_{1}\\x_{2}\\\xi_{1}\\\xi_{2}
\end{pmatrix},
\end{equation}
we get from the lemma \ref{2.lem.keydiag} the following result.
\begin{lem}\label{2.lem.effectq}
For $(x_{1},x_{2},\xi_{1},\xi_{2})\in \R^4$, $(y_{1}, y_{2},\eta_{1},\eta_{2})\in \R^4$
given by \eqref{2.lastcomp}, we have the following identity,
{\small
\begin{align*}
\mu_{1}^2y_{1}^2+\mu_{2}^2y_{2}^2+\eta_{1}^2+\eta_{2}^2&=
\mu_{1}^2
\bigl(
c\lambda_{1}x_{2}+\lambda_{2}d^{-1}\xi_{2}
\bigr)^2
+\mu_{2}^2
\bigl(
c\lambda_{2}x_{1}+\lambda_{1}d^{-1}\xi_{1}
\bigr)^2
\\&\hs +\bigl(
(-d\lambda_{2}^{-1}+\lambda_{1}cd)x_{2}+\lambda_{1}\xi_{1}\bigr)^2
+\bigl(
(-d\lambda_{1}^{-1}+\lambda_{2}cd)x_{1}+\lambda_{2}\xi_{2}\bigr)^2
\\&=
\xi_{1}^2+\xi_{2}^2+(1-\nu^2) x_{1}^2+
(1+\nu^2) x_{2}^2-2\omega(x_{1}\xi_{2}-x_{2}\xi_{1}),
\end{align*}
}where the parameters $c, \lambda_{2}, d, \lambda_{1}$ are defined above
(note that all these parameters are well-defined when $(\omega,\nu)$
are both positive with $\omega^2+\nu^2<1$).
\end{lem}
We have achieved an explicit diagonalization of the quadratic form \eqref{1.fetter} and, most importantly, that diagonalization is performed via a symplectic mapping.
That feature will be of particular importance in our next section.
Expressing  the parameters in terms of
$\alpha, \omega,\nu, \ep$
(cf. section \ref{notations}), we obtain
{\small\begin{align*}
q&=\Bigl(2^{-1/2}\alpha^{-1/2}(\alpha-2\omega^2+\nu^2)^{1/2}\xi_{1}
-2^{-3/2}\omega^{-1}\alpha^{-1/2}(\alpha-2\omega^2+\nu^2)^{1/2}(\alpha-\nu^2)x_{2}\Bigr)^2
\\&
+
\Bigl(2^{-1/2}\alpha^{-1/2}\bigl(\frac{\alpha+2\omega^2-\nu^2}{2\nu^2+\ep^2}\bigr)^{1/2}
\frac{(2\nu^2\ep^2+\ep^4)^{1/2}}{\mu_{2}}
\xi_{2}
\\&
\hskip84pt
+\frac{(2\nu^2\ep^2+\ep^4)^{1/2}}{\mu_{2}}
(\alpha^{1/2}+\nu^2\alpha^{-1/2})
2^{-3/2}\omega^{-1}
\bigl(\frac{\alpha+2\omega^2-\nu^2}{2\nu^2+\ep^2}\bigr)^{1/2}x_{1}
\Bigr)^2
\\&
+
\Bigl((1+\omega^2+\alpha)^{1/2}
2^{1/2}\alpha^{-1/2}\omega
(\alpha+2\omega^2+\nu^2)^{-1/2}
\xi_{1}
\\&
\hskip84pt
+
(1+\omega^2+\alpha)^{1/2}
(1+\alpha^{-1}\nu^2) 2^{-1/2}\alpha^{1/2}
(\alpha+2\omega^2+\nu^2)^{-1/2}
x_{2}\Bigr)^2
\\&+
\Bigl(
2^{-1/2}\alpha^{-1/2}
(\alpha+2\omega^2+\nu^2)^{1/2}
\xi_{2}
-2^{-3/2}
\omega^{-1}\alpha^{-1/2}
(\alpha-\nu^2)
(\alpha+2\omega^2+\nu^2)^{1/2}x_{1}
\Bigr)^2,
\end{align*}}
so that
\begin{align}\label{mastereqhere}\begin{split}
q&=\overbrace{\Bigl(\frac{\alpha-2\omega^2+\nu^2}{2\alpha}\Bigr)
\Bigl[\xi_{1}-\bigl(\frac{\alpha-\nu^2}{2\omega}\bigr)x_{2}\Bigr]^2}^{\eta_{1}^2}
+\overbrace{
\Bigl(\frac{\alpha+2\omega^2-\nu^2}{2\alpha\mu_{2}^2}\Bigr)\ep^2
\Bigl[\xi_{2}+\bigl(\frac{\alpha+\nu^2}{2\omega}\bigr)x_{1}\Bigr]^2}^{\mu_{1}^2y_{1}^2}
\\&+
\underbrace{
2\omega^2
\Bigl(\frac{1+\omega^2+\alpha}{\alpha(\alpha+2\omega^2+\nu^2)}\Bigr)
\Bigl[\xi_{1}+\bigl(\frac{\alpha+\nu^2}{2\omega}\bigr)x_{2}\Bigr]^2}_{\mu_{2}^2 y_{2}^2}
+
\underbrace{\Bigl(\frac{\alpha+2\omega^2+\nu^2}{2\alpha}\Bigr)
\Bigl[\xi_{2}-\bigl(\frac{\alpha-\nu^2}{2\omega}\bigr)x_{1}\Bigr]^2}_{\eta_{2}^2}.
\end{split}
\end{align}
The equation \eqref{mastereqhere}
encapsulates most of our previous work on the diagonalization of $q$.
In the appendix \ref{a.double},
we provide another way of checking the symplectic relationships between the
linear forms, $y_{j},\eta_{l}$.
\par
We have seen in Lemma \ref{2.lem.effect}
that when $\ep=0,\nu>0$, the rank of  $q$ is  3, whereas its symplectic rank is 2.
Indeed,  $\ep=0$ and  $\nu>0$, we have
\begin{align}\label{master0h}\begin{split}
q&=\overbrace{\Bigl(\frac{\alpha-2\omega^2+\nu^2}{2\alpha}\Bigr)
\Bigl[\xi_{1}-\bigl(\frac{\alpha-\nu^2}{2\omega}\bigr)x_{2}\Bigr]^2}^{\eta_{1}^2}
+
\\&+
\underbrace{
2\omega^2
\Bigl(\frac{1+\omega^2+\alpha}{\alpha(\alpha+2\omega^2+\nu^2)}\Bigr)
\Bigl[\xi_{1}+\bigl(\frac{\alpha+\nu^2}{2\omega}\bigr)x_{2}\Bigr]^2}_{\mu_{2}^2 y_{2}^2}
+
\underbrace{\Bigl(\frac{\alpha+2\omega^2+\nu^2}{2\alpha}\Bigr)
\Bigl[\xi_{2}-\bigl(\frac{\alpha-\nu^2}{2\omega}\bigr)x_{1}\Bigr]^2}_{\eta_{2}^2}.
\end{split}
\end{align}
\section{Quantization}
\subsection{The Irving E. Segal formula}
Let  $a$ be
defined on  $\R^n_{x}\times \R^n_{\xi}$
(say a tempered distribution on $\RZ$).
Its Weyl quantization is the operator,
acting for instance on $u\in \mathscr S(\R^n)$,
\begin{equation}\label{3.weylf}
(a^w u)(x)=\iint e^{2i\pi(x-x')\xi} a(\frac{x+x'}{2},\xi) u(x') dx' d\xi.
\end{equation}
In fact, the weak formula
$
\poscal{a^w u}{v}=\int_{\RZ} a(x,\xi)\mathcal H(u,v)(x,\xi) dx d\xi$
makes sense for $a\in \mathscr S'(\RZ), u,v\in \mathscr S(\R^n)$
since the { Wigner function} $\mathcal H(u,v)$ defined by
$$\mathcal H(u,v)(x,\xi) =\int e^{-2i\pi x' \xi} u(x+\frac{ x'}2)\bar v(x-\frac {x'}2)dx'
$$
belongs to
$\mathscr S(\RZ)$ for $u,v\in \mathscr S(\R^n)$ .
Note also our definition of the Fourier transform
$\hat u (\xi)=\int e^{-2i\pi x\cdot \xi} u(x) dx$
(so that  $ u (x)=\int e^{2i\pi x\cdot \xi} \hat u(\xi) d\xi$)
and
$$
\xi_{j}^wu=\frac{1}{2i\pi}\frac{\p u}{\p x_{j}}=D_{j}u,      \quad x_{j}^w u=x_{j}u,
\quad
(x_{j}\xi_{j})^w=\frac12\bigl(x_{j}D_{j}+D_{j}x_{j}\bigr).
$$
Let  $\chi$ be a linear symplectic transformation $\chi (y,\eta)=(x,\xi)$. The
Segal  formula (see e.g. the theorem 18.5.9 in
\cite{MR2304165}) asserts that there exists a unitary transformation  $M$
of  $L^2(\R^n)$,
uniquely determined apart from a constant factor of modulus one,
which is also an automorphism of $\mathscr S(\R^n)$
and  $\mathscr S'(\R^n)$
such that, for all $a\in \mathscr S'(\RZ)$,
\begin{equation}\label{3.segalf}
(a\circ \chi)^w=M^* a^w M,
\end{equation}
providing the following commutative diagrams
$$
\begin{CD}
\mathscr S(\R^n_{x}) @>a^w>> \mathscr S'(\R^n_{x})\\
@AMAA @VVM^*V\\
\mathscr S(\R^n_{y}) @>>(a\circ \chi)^w>      \mathscr S'(\R^n_{y})
\end{CD}
\hs\text{\small and if $a^w\in \mathcal L(L^2(\R^n))$}\hs
\begin{CD}
L^2(\R^n_{x}) @>a^w>> L^2(\R^n_{x})\\
@AMAA @VVM^*V\\
L^2(\R^n_{y}) @>>(a\circ \chi)^w>      L^2(\R^n_{y})
\end{CD}
$$
\subsection{The metaplectic group and the generating functions}
For a given $\chi$, how can we determine $M$ ?
We shall not need here the rich algebraic structure of the two-fold covering
$Mp(n)$ (the metaplectic group in which live the transformations  $M$) of the symplectic group $Sp(n)$.
The following lemma is classical
(and also easy to prove directly using the factorization of the lemma \ref{2.lem.symgen}) and provides a simple expression for $M$ when the transformation $\chi$
has a generating function.
\begin{lem}\label{3.lem.metapl}
Let $\chi=\Xi_{A,B,C}$ be the  symplectic mapping
given by \eqref{2.symgen}.
 Then the Segal formula \eqref{3.segalf} holds with
\begin{equation}\label{gen}
(Mv)(x)=\int e^{2i\pi S(x,\eta)}\widehat v(\eta) d\eta\val{\det B}^{1/2},
\end{equation}
where $S$ is given by  \eqref{4.genfun}.
\end{lem}
\subsection{Explicit expression for $M$}
\begin{lem}
Let $\chi$ be the symplectic transformation of $\R^4$
given by \eqref{2.symmat}.
Then the Segal formula \eqref{3.segalf} holds with
$M$ given by
\begin{multline}
(Mv)(x_{1},x_{2})=(\lambda_{1}\lambda_{2})^{-1/2}
e^{2i\pi d\left((\lambda_{1}\lambda_{2})^{-1} -c\right)x_{1}x_{2}}
      \\
      \times \iint e^{2i\pi d^{-1}\eta_{1}\eta_{2}}
\widehat v(\eta_{1},\eta_{2}) e^{2i\pi(\lambda_{1}^{-1}x_{1}\eta_{1}+\lambda_{2}^{-1}x_{2}\eta_{2})} d\eta_{1} d\eta_{2},
\end{multline}
\begin{equation}\label{3.metapl}
(Mv)(x_{1},x_{2})=(\lambda_{1}\lambda_{2})^{-1/2}
e^{2i\pi d\left((\lambda_{1}\lambda_{2})^{-1} -c\right)x_{1}x_{2}}
(e^{2i\pi d^{-1}D_{1}D_{2}}v)(\lambda_{1}^{-1}x_{1},\lambda_{2}^{-1}x_{2}).
\end{equation}
\end{lem}
 \begin{proof}
We apply the lemmas \ref{3.lem.metapl} and \ref{2.lem.calcul},
along with the fact that the mapping $Mp(n)\ni M\mapsto \chi\in Sp(n)$
 is an homomorphism
or more elementarily that \eqref{3.segalf} implies for $\chi_{j}\in Sp(n)$,
 $$
(a\circ \chi_{2}\circ \chi_{1})^w=M_{1}^* (a\circ \chi_{2})^w M_{1}
=
M_{1}^* M_{2}^*a^w M_{2} M_{1}.
$$
The factorization of the lemma \ref{2.lem.calcul}
 implies that
$$
(Mv)(x)= e^{i\pi\poscal{Ax}{x}}\int_{\R^2} e^{2i\pi\poscal{Bx}{\eta}}
e^{i\pi\poscal{C\eta}{\eta}}\widehat v(\eta) d\eta,
$$
 which gives readily the formulas above.
\end{proof}
Summing-up, we have proven the following result.
\begin{theorem}\label{3.thm.algwork}
Let $q$ be the quadratic form on $\R^4$
given by \eqref{1.fetter}.
We define the symplectic mapping $\chi$ by \eqref{2.symmat}
and the metaplectic mapping $M$ by \eqref{3.metapl}.
We have
\begin{align}
(q\circ\chi)(y,\eta)&=\mu_{1}^2y_{1}^2+\mu_{2}^2y_{2}^2+\eta_{1}^2+\eta_{2}^2,\quad\text{\small $($the $\mu_{j}^2$ are given by \eqref{2.parone}$)$,}
\\
(q\circ\chi)^w&= M^* q^w M.
\end{align}
\end{theorem}
We can also explicitly quantize the formulas of the lemma \ref{2.lem.effectq}, to obtain\footnote{Note that for a linear form  $L$ on $\RZ$, $L^wL^w=(L^2)^w$.}
\begin{multline}\label{diagopnew}
q^w=
\overbrace{\Bigl((\lambda_{1}cd-d\lambda_{2}^{-1})x_{2}+\lambda_{1}D_{x_{1}}\Bigr)^2}^{(\eta_{1}^2)^w}
+
\overbrace{\mu_{1}^2\Bigl(\lambda_{2}d^{-1}D_{x_{2}}+c\lambda_{2}{x_{1}}\Bigr)^2}^{\mu_{1}^2(y_{1}^2)^w}
\\
+
\underbrace{\Bigl((\lambda_{2}cd-d\lambda_{1}^{-1})x_{1}+\lambda_{2}D_{x_{2}}\Bigr)^2}_{(\eta_{2}^2)^w}
+
\underbrace{\mu_{2}^2\Bigl(\lambda_{1}d^{-1}D_{x_{1}}+c\lambda_{1}{x_{2}}\Bigr)^2}_{\mu_{2}^2(y_{2}^2)^w}.
\end{multline}
\section{The Fock-Bargmann space and the anisotropic $LLL$ }
\subsection{Nonnegative quantization and entire functions}
\begin{defi}
For  $X,Y\in \RZ$ we set
\begin{equation}\label{4.project}
\Pi(X,Y)=e^{-\frac{\pi}{2}\val{X-Y}^2}e^{-i\pi[X,Y]},
\end{equation}
where $[X,Y]=\poscal{\sigma X}{Y}$ is the symplectic form
\eqref{2.symfor}.
For $v\in L^2(\R^n)$,
we define
\begin{equation}\label{4.doubleu}
(Wv)(y,\eta)=\poscal{v}{\varphi_{y,\eta}}_{L^2(\R^n)},\quad\text{with\hs}
\varphi_{y,\eta}(x)=2^{n/4}e^{-\pi(x-y)^2}e^{2i\pi(x-\frac y2)\eta}.
\end{equation}
We define also
\begin{equation}\label{4.lllpi}
\Lambda_{0} \text{ =$\{u\in L^2(\R^{2n}_{y, \eta})$ such that $u=f(z) e^{-\frac{\pi}{2}\val z^2}$, $z=\eta+iy$ , $f$ entire.\}}
\end{equation}
\end{defi}
\begin{pro}\label{4.pro.pizer0}
The operator $\Pi_{0}$ with kernel
 $\Pi(X,Y)$is the orthogonal projection
in $L^2(\RZ)$
on  $\Lambda_{0}$,
which is a  proper closed subspace of $L^2(\RZ)$,
canonically isomorphic to $L^2(\R^n)$.
We  have
\begin{align}
\Lambda_{0}&=\text{\rm ran} W=L^2(\RZ)\cap \ker( \bar \p+\frac\pi 2 z),\\
W^*W&=\Id_{L^2(\R^n)}\quad\text{\small (reconstruction formula
$u(x)=\int_{\RZ}(Wu)(Y) \varphi_{Y}(x)dY$),}      \
\\
WW^*&=\Pi_{0},\quad\text{\small            ($W$ is an isomorphism from  $L^2(\R^n)$
onto  $\Lambda_{0}$).}      \
\end{align}
\end{pro}
\begin{proof}
These statements are  classical
(see e.g. \cite{MR1957713}) ; however, since we shall need some extension of that proposition, it is useful  to examine the proof.
We note that
$e^{-i\pi y\eta}(Wv)(y,\eta)$ is the partial Fourier transform w.r.t. $x$
of $$\R^n\times \R^n \ni(x,y)\mapsto v(x)
2^{n/4}e^{-\pi(x-y)^2},$$
whose $L^2(\RZ)$-norm is $\norm{v}_{L^2(\R^n)}$
so that
$W$
is isometric from $L^2(\R^n)$ into $L^2(\RZ)$,
thus with a closed range.
As a result, we have $W^*W=\Id_{L^2(\R^n)}$,
$WW^*$ is selfadjoint and such that
$WW^*WW^*=WW^*$: $WW^*$ is indeed the orthogonal projection
on $\range W$
($\range WW^*\subset\range W$ and $Wu=WW^*Wu$).
The straightforward computation
of the kernel of $WW^*$ is left to the reader.
Let us prove that $\Lambda_{0}=\range W$
is indeed defined by \eqref{4.lllpi}.
For  $v\in L^2(\R^n)$, we have
\begin{multline}\label{4.isoww}
(Wv)(y,\eta)=\int_{\R^n} v(x)
2^{n/4}e^{-\pi(x-y)^2}e^{-2i\pi(x-\frac y2)\eta}dx
\\
=
\int_{\R^n} v(x)
2^{n/4}e^{-\pi(x-y+i\eta)^2}
dx
e^{-\frac\pi 2 (y^2+\eta^2)}
e^{-\frac\pi 2 (\eta+iy)^2}
\end{multline}
and we see that $Wv\in L^2(\RZ)\cap \ker( \bar \p+\frac\pi 2 z)$.
Conversely,
if $\Phi\in L^2(\RZ)\cap \ker( \bar \p+\frac\pi 2 z)$,
we have $\Phi(x,\xi)=e^{-\frac \pi 2(x^2+\xi^2)} f(\xi+ix)$
with $\Phi\in L^2(\RZ)$ and $f$ entire. This gives
\begin{align*}
(WW^*\Phi)&(x,\xi)=\iint  e^{-\frac \pi 2\bigl((\xi-\eta)^2+(x-y)^2+2i \xi y-2i \eta x\bigr)}
\Phi(y,\eta)dy d\eta
\\
&=e^{-\frac \pi 2(\xi^2+x^2)}\iint  e^{-\frac \pi 2(\eta^2-2\xi \eta+y^2-2x y+2i \xi y-2i \eta x)}
\Phi(y,\eta)dy d\eta
\\
&=e^{-\frac \pi 2(\xi^2+x^2)}\iint  e^{-\frac \pi 2
\bigl(\eta^2+y^2+2iy (\xi +ix)-2 \eta (\xi+i x)\bigr)}
\Phi(y,\eta)dy d\eta
\\
&=e^{-\frac \pi 2(\xi^2+x^2)}\iint  e^{-\pi(y^2+\eta^2)}e^{\pi(\eta-i y)
(\xi +ix)}
f(\eta +iy)dy d\eta
\\
&=e^{-\frac \pi 2\val z^2}\iint
e^{-\pi\val \zeta^2} e^{\pi \bar \zeta z} f(\zeta) dy d\eta\quad (\zeta=\eta+iy,\ z=\xi+ix)
\\
&=e^{-\frac \pi 2\val z^2}\iint{{f(\zeta) }\prod_{1\le j\le n}\frac{1}{\pi (z_{j}-\zeta_{j})}\frac{\p }{\p \bar \zeta_{j}}\Bigl(
e^{-\pi\val \zeta^2} e^{\pi \bar \zeta z}\Bigr)} dy d\eta
\\
&=e^{-\frac \pi 2\val z^2}\poscal {f(\zeta)\prod_{1\le j\le n}\frac{\p }{\p \bar \zeta_{j}}\Bigl(
\frac{1}{\pi (\zeta_{j}-z_{j})} \Bigr)}
{
e^{-\pi\val \zeta^2} e^{\pi \bar \zeta z}}_{\mathscr S'(\RZ), \mathscr S(\RZ)}
\\&=e^{-\frac \pi 2\val z^2} f(z),
\end{align*}
since $f$ is entire.
This implies  $WW^* \Phi=\Phi$
and $\Phi\in \range W$.
The proof of the proposition is complete.
\end{proof}
\begin{pro}\label{4.pro.focksprime}
Defining
\begin{equation}\label{4.focksprime}
\mathscr K=\ker (\bar\p+\frac{\pi}{2} z)\cap \mathscr S'(\RZ),
\end{equation}
the operator $W$
given by \eqref{4.doubleu}
can be extended as a continuous mapping
from
$\mathscr S'(\R^n)$ onto $\mathscr K$
(the $L^2(\R^n)$
dot-product is replaced by a bracket of (anti)duality).
The operator $\widetilde{\Pi}$ defined by its kernel
$\Pi$ given by
\eqref{4.project}
defines a continuous mapping from $\mathscr S(\RZ)$
into itself
and can be extended as a continuous mapping
 from $\mathscr S'(\RZ)$
onto $\mathscr K$.
It verifies
\begin{equation}\label{4.fockagain}
\widetilde{\Pi}^2=\widetilde{\Pi},\quad{\widetilde{\Pi}}_{\vert \mathscr K}=\Id_{\mathscr K}.
\end{equation}
\end{pro}
\begin{proof}
As above we use that
$e^{-i\pi y\eta}(Wv)(y,\eta)$ is the partial Fourier transform w.r.t. $x$
of the tempered distribution on $\RZ_{x,y}$
$$ v(x)
2^{n/4}e^{-\pi(x-y)^2}.$$
Since $e^{\pm i\pi y\eta}$ are in the space $\mathscr O_{M}(\RZ)$
of multipliers
of $\mathscr S(\RZ)$,
that transformation
is continuous and injective from
$\mathscr S'(\R^n)$ into
$\mathscr S'(\RZ)$.
Replacing in \eqref{4.isoww}
the integrals by brackets of duality,
we see that
$W(\mathscr S'(\R^n))\subset\mathscr K$.
Conversely, if $\Phi\in \mathscr K$,
the same calculations as above
give \eqref{4.fockagain} and \eqref{4.focksprime}.
\end{proof}
For a Hamiltonian $a$ defined on $\RZ$,
for instance a bounded function on $\R^{2n}$,
we define
$
a^{\text{Wick}}=W^* a W:
$
$$
\begin{CD}
L^2(\RZ) @>a >{\text{             (multiplication by $a$)}}> L^2(\RZ)\\
@AWAA @VVW^*V\\
L^2(\R^n) @>>a^{\text{Wick}}>      L^2(\R^n)
\end{CD}
$$
we note that
$
a(x,\xi)\ge 0\Longrightarrow  a^{\text{Wick}}=W^* a W\ge 0$,
as an operator.
There are many useful applications of the Wick quantization
due to that non-negativity property,
but for our purpose here,
it will be more important to relate  explicitely that quantization
to   the usual Weyl quantization (as given by \eqref{3.weylf}) for quadratic forms.
\begin{lem}\label{lm:b_inf_y2}
Let $q(X)=\poscal{QX}{X}$ be a quadratic form
on $\RZ$ ($Q$ is a $2n\times 2n$ symmetric matrix).
Then we have
\begin{equation}\label{wiwequ}
q^{\text{Wick}}= q^w+\frac{1}{4\pi}\trace Q.
\end{equation}
Let $L(y,\eta)=\tau\cdot y- t\cdot \eta$ be a real linear form on $\RZ$; then, for all $\Phi\in \Lambda_{0}$, we have
\begin{equation}\label{unprll}
\iint {L(y,\eta)}^2\val{\Phi(y,\eta)}^2 dy d\eta\ge \frac{\val{\tau}^2+\val t^2}{4\pi }\norm{\Phi}_{L^2(\RZ)}^2.
\end{equation}
\end{lem}
\begin{proof}
A straightforward computation
shows that
\begin{equation}
\w{q}=(q\ast \Gamma)^w,
\quad\text{where $\Gamma(X)=2^n e^{-2\pi\val X^2}$
($X\in\RZ$).}
\end{equation}
By Taylor's formula, we have
$
(q\ast \Gamma)(X)= q(X)+\int_{\RZ}2^ne^{-2\pi\val{Y}^2}\poscal{QY}{Y} dY,
$
we can use
the formula
$
\int_{\R}2^{1/2} t^2 e^{-2\pi t^2} dt=\frac{1}{4\pi}$
to get the first result.
For $\Phi\in \Lambda_{0}$, we have $\Phi=Wu$ with $u\in L^2(\R^n)$
and thus
\begin{align*}
\norm{L\Phi }_{L^2(\RZ)}^2&=\poscal{L^2 Wu}{Wu}_{L^2(\RZ)}=
\poscal{W^*L^2 Wu}{u}_{L^2(\R^n)}
\\&=
\poscal{\w{{(L^2)}}u}{u}_{L^2(\R^n)}=\poscal{{{(L^2)^w}}u}{u}_{L^2(\R^n)}+
\frac {\trace(L^2)}{4\pi}\norm{u}^2_{L^2(\R^n)},
\end{align*}
and since $L^wL^w=(L^2)^w$ for a linear form, we get since $L$ is real-valued,
$$\norm{L\Phi }_{L^2(\RZ)}^2
=\norm{L^w u}^2_{{L^2(\R^n)}}+
\frac {\val{\tau}^2+\val t^2}{4\pi}\norm{\Phi}^2_{L^2(\RZ)},$$
which implies \eqref{unprll}.
\end{proof}
\begin{nb}{\rm\small
The inequality \eqref{unprll}
looks like an uncertainty principle related  to the localization in
$\RZ$
for the functions of
$\Lambda_{0}$.
Moreover the equality
\eqref{wiwequ}
provides a simple way to saturate approximately
the inequality \eqref{unprll};
for instance if $L(y,\eta)=y_{1}$,
we consider the sequence
$\Phi_{\ep}=W u_{\ep}$
with
$
u_{\ep}(x)=\varphi(x_{1}/\ep)\ep^{-1/2}\psi(x'),\quad\norm{\varphi}_{L^2(\R)}
=\norm{\psi}_{L^2(\R^{n-1})}=1,
$
and we get, provided $x\varphi(x)\in L^2(\R)$,
$$
\iint y_{1}^2\val{\Phi_{\ep}(y,\eta)}^2 dy d\eta=
\int_{\R} x_{1}^2\val{\varphi(x_{1}/\ep)}^2\ep^{-1}dx_{1}+\frac{1}{4\pi}=O(\ep^2)+\frac{1}{4\pi}.
$$
}\end{nb}
\subsection{The anisotropic $LLL$}
Going back to the Gross-Pitaevskii energy \eqref{1.engrpi},
with $q$ given by \eqref{1.fetter1},
we see, using the theorem \ref{3.thm.algwork} and \eqref{diagopnew} that,
with $u=Mv$,
\begin{align*}
2E_{GP}(u)&=\poscal{q^w u}{u}_{L^2(\R^2)}
+g\int \val{u}^4dx\\
&=
\poscal{M^*q^w Mv}{v}_{L^2(\R^2)}+g\int \val{(Mv)(x)}^4dx
\\&=
\poscal{(D_{y_{1}}^2+\mu_{1}^2 y_{1}^2+D_{y_{2}}^2
+\mu_{2}^2 y_{2}^2)v}{v}_{L^2(\R^2)}
+g\int \val{(Mv)(x)}^4dx
\\&=
\bigl\langle
{\bigl((\lambda_{1}cd-d\lambda_{2}^{-1})x_{2}+\lambda_{1}D_{x_{1}}\bigr)^2u+
\mu_{1}^2\bigl(\lambda_{2}d^{-1}D_{x_{2}}+c\lambda_{2}{x_{1}}\bigr)^2u},
{u}\bigr\rangle\\
&\hs +\bigl\langle
{\bigl((\lambda_{2}cd-d\lambda_{1}^{-1})x_{1}+\lambda_{2}D_{x_{2}}\bigr)^2 u},{u}\bigr\rangle+
\bigl\langle
{\mu_{2}^2\bigl(\lambda_{1}d^{-1}D_{x_{1}}+c\lambda_{1}{x_{2}}\bigr)^2u},{u}\bigr\rangle
\\
&\hs \hs+g\int \val{u}^4dx.
\end{align*}
The question at hand is the determination of
$\inf_{\norm u_{L^2}=1}E_{GP}(u)$, which is equal to
$
\inf_{\norm v_{L^2}=1}E_{GP}(Mv).
$
Since $\mu_{1}=0$ at $\ep=0$ (see \eqref{2.partwo}) and $\mu_{2}\in [1,4]$
(see \eqref{12.rem.obvious}),
it is natural to modify our minimization problem,
and in the $(y,\eta)$ coordinates,
to restrict our attention to the {\it Lowest Landau Level},
i.e. the groundspace of $D_{y_{2}}^2+\mu_{2}^2 y_{2}^2$,
that is  the subspace of $L^2(\R^2)$
\begin{equation}
LLL_{y}=\{v_{1}(y_{1})\otimes 2^{1/4}\mu_{2}^{1/4}e^{-\pi \mu_{2}y_{2}^2}\}_{v_{1}\in L^2(\R)}
=\ker (D_{y_{2}}-i \mu_{2}y_{2})\cap L^2(\R^2).
\end{equation}
If we want to stay in the physical coordinates $(x,\xi)$ we reach the following definition,
obtained by  using Segal's formula \eqref{3.segalf}
with $M, \chi$ given in the lemma \ref{3.lem.metapl}
so that
$$LLL_{x}=M(LLL_{y}).$$
\begin{pro}\label{4.pro.LLL}
 Let $q$ be the quadratic form on $\R^4$ given by \eqref{1.fetter1}.
We define the $LLL$ as
\begin{align}
LLL&=(\ker \mathcal L)\cap L^2(\R^2),\qquad{\text{with}}
\\
\mathcal L=
(\lambda_{2}cd-d\lambda_{1}^{-1})x_{1}&+\lambda_{2}D_{x_{2}}
-i
\mu_{2}\lambda_{1}d^{-1}D_{x_{1}}-i\mu_{2}c\lambda_{1}{x_{2}}=\eta_{2}^w-i\mu_{2}y_{2}^w.
\end{align}
The $LLL$
is the subspace of $L^2(\R^2)$ of functions of type
{\small \begin{equation}\label{4.df.LLL}
F\bigl(x_{1}+i\beta_{2}x_{2}\bigr)
\exp{\left(-\frac{\gamma\pi}{4\beta_{2}}\Bigl[
x_{1}^2(1-\frac{\nu^2}{2\alpha})+
(\beta_{2}x_{2})^2(1+\frac{\nu^2}{2\alpha})\Bigr]
\right)}
\exp{(-i\frac{\pi\nu^2\gamma }{4\alpha }x_{1}x_{2})},
\end{equation}}
where $F$ is entire on $\C$, and the parameters $\gamma,\beta_{2},\nu, \alpha$
are given in the section $\ref{notations}$.
The real part of the phase of  the Gaussian function multiplying $F\bigl(x_{1}+i\beta_{2}x_{2}\bigr)$ is a negative definite quadratic form when $(\omega,\nu)\not=(0,0)$.\end{pro}
\begin{proof}
We have
\begin{multline*}
i\mathcal L=
\overbrace{\Bigl(
\mu_{2}\lambda_{1}d^{-1}D_{x_{1}}+\mu_{2}c\lambda_{1}{x_{2}}
\Bigr)}^{\mu_{2}y_{2}}
+i\overbrace{
\Bigl(
\lambda_{2}D_{x_{2}}
-(d\lambda_{1}^{-1}-\lambda_{2}cd)x_{1}
\Bigr)}^{\eta_{2}}      \
\\=\frac{1}{2i\pi}
\Bigl(
\mu_{2}\lambda_{1}d^{-1}\p_{1}+i\lambda_{2}\p_{2}
+2i\pi\mu_{2}c\lambda_{1}{x_{2}}
+2\pi (d\lambda_{1}^{-1}-\lambda_{2}cd)x_{1}
\Bigr)      \
\\=
\frac{1}{i\pi}
\Bigl(\frac 12
\mu_{2}\lambda_{1}d^{-1}\p_{1}+i\frac 1 2\lambda_{2}\p_{2}
+i\pi\mu_{2}c\lambda_{1}{x_{2}}
+\pi (d\lambda_{1}^{-1}-\lambda_{2}cd)x_{1}
\Bigr).      \
\end{multline*}
We set
\begin{equation}
t_{1}=\mu_{2}^{-1}\lambda_{1}^{-1} d x_{1},
\quad t_{2}=\lambda_{2}^{-1} x_{2},
\end{equation}
and we get for $z=t_{1}+it_{2}$,
\begin{align*}
&\frac{\p}{\p \bar z}+i\pi\mu_{2}c\lambda_{1}\lambda_{2}t_{2} +\pi
(d\lambda_{1}^{-1}-\lambda_{2}cd)\mu_{2}\lambda_{1}d^{-1}t_{1}
\\&=
\frac{\p}{\p \bar z}+i\pi\mu_{2}c\lambda_{1}\lambda_{2}\frac{z-\bar z}{2i}
+\pi
(d\lambda_{1}^{-1}-\lambda_{2}cd)\mu_{2}\lambda_{1}d^{-1}\frac{z+\bar z}{2}
\\&=
\frac{\p}{\p \bar z}+z\pi\frac{\mu_{2}}{2}
+\bar z \pi \frac{\mu_{2}}2(1-2\lambda_{1}\lambda_{2}c)
=
\frac{\p}{\p \bar z}+z\pi\frac{\mu_{2}}{2}-
\bar z \pi \frac{\mu_{2}}2\nu^2\alpha^{-1}
\\&
=e^{-\pi\frac{\mu_{2}}{2}z\bar z} e^{\pi\frac{\nu^2\mu_{2}}{4\alpha}(\bar z)^2}
\frac{\p}{\p \bar z} e^{\pi\frac{\mu_{2}}{2}z\bar z} e^{-\pi\frac{\nu^2\mu_{2}}{4\alpha}(\bar z)^2}.
\end{align*}
As a consequence,
the $LLL$ is  the subspace of $L^2(\C)$ of functions
$$
f(z) e^{-\pi\frac{\mu_{2}}{2}z\bar z} e^{\pi\frac{\nu^2\mu_{2}}{4\alpha}(\bar z)^2},\quad\text{with $f$ holomorphic.}
$$
We note that the real part of the exponent
is
$$
-\frac{\pi \mu_{2}}{2}(t_{1}^2+t_{2}^2-\frac{\nu^2}{2\alpha}(t_{1}^2-t_{2}^2))
=-\frac{\pi \mu_{2}}{2}\Bigl[t_{1}^2(\frac{2\alpha-\nu^2}{2\alpha})+
t_{2}^2(\frac{2\alpha+\nu^2}{2\alpha})\Bigr]
$$
and that
$$
2\alpha-\nu^2>0\Longleftrightarrow (\omega,\nu)\not=(0,0).
$$
Leaving the $t$-coordinates for the original $x$-coordinates, we get
with $f$ entire,
{\small $$
f(\mu_{2}^{-1}\lambda_{1}^{-1} dx_{1}+i \lambda_{2}^{-1}x_{2})
\exp{\left(-\frac{\pi \mu_{2}}{2}\Bigl[t_{1}^2(\frac{2\alpha-\nu^2}{2\alpha})+
t_{2}^2(\frac{2\alpha+\nu^2}{2\alpha})\Bigr]
\right)}
\exp{(-i\frac{\pi\mu_{2}\nu^2}{2\alpha}t_{1}t_{2})},
$$}
i.e.
{\small$$
f(\mu_{2}^{-1}\lambda_{1}^{-1} dx_{1}+i \lambda_{2}^{-1}x_{2})
\exp{\left(-\frac{\pi \mu_{2}}{2}\Bigl[x_{1}^2d^2(\frac{2\alpha-\nu^2}{2\alpha\lambda_{1}^2\mu_{2}^2})+
x_{2}^2(\frac{2\alpha+\nu^2}{2\alpha\lambda_{2}^2})\Bigr]
\right)}
\exp{(-i\frac{\pi\mu_{2}\nu^2 d}{2\alpha\lambda_{1}\lambda_{2}\mu_{2}}x_{1}x_{2})},
$$}
and since
\begin{align*}
\mu_{2}\lambda_{1}d^{-1}\lambda_{2}^{-1}&=\mu_{2}\lambda_{1}2\gamma^{-1}\lambda_{1}^{-1}\lambda_{2}^{-1}
\lambda_{2}^{-1}=\mu_{2}2\gamma^{-1}\lambda_{2}^{-2}=
\mu_{2}2\gamma^{-1}\gamma\beta_{2}(2\mu_{2})^{-1}=\beta_{2},
\\
2^{-1}\mu_{2}d^2\lambda_{1}^{-2}\mu_{2}^{-2}&
=2^{-1}\mu_{2}\gamma^2 4^{-1}\lambda_{2}^2\mu_{2}^{-2}
=2^{-1}\mu_{2}\gamma^2 4^{-1}2\mu_{2}\gamma^{-1}\beta_{2}^{-1}\mu_{2}^{-2}=
\frac{\gamma}{4\beta_{2}},
\\
2^{-1}\mu_{2}\lambda_{2}^{-2}&=2^{-1}\mu_{2}\frac{\gamma\beta_{2}}{2\mu_{2}}
=\frac{\gamma\beta_{2}}{4},\\
\frac{\pi\mu_{2}\nu^2 d}{2\alpha\lambda_{1}\lambda_{2}\mu_{2}}&=
\frac{\pi\nu^2 d}{2\alpha\lambda_{1}\lambda_{2}}
=\frac{\pi\nu^2\gamma }{2\alpha 2},
\end{align*}
we obtain
\begin{multline*}
f\bigl(\mu_{2}^{-1}\lambda_{1}^{-1} d[x_{1}+i \underbrace{
\mu_{2}\lambda_{1}d^{-1}\lambda_{2}^{-1}}_{=\beta_{2}}x_{2}]\bigr)
\\\times\exp{\left(-\frac{\pi \mu_{2}}{2}\Bigl[x_{1}^2d^2(\frac{2\alpha-\nu^2}{2\alpha\lambda_{1}^2\mu_{2}^2})+
x_{2}^2(\frac{2\alpha+\nu^2}{2\alpha\lambda_{2}^2})\Bigr]
\right)}
\exp{(-i\frac{\pi\mu_{2}\nu^2 d}{2\alpha\lambda_{1}\lambda_{2}\mu_{2}}x_{1}x_{2})},
\end{multline*}
that is, with $F$ entire on $\C$,
{\small \begin{equation}
F\bigl(x_{1}+i\beta_{2}x_{2}\bigr)
\exp{\left(-\frac{\gamma\pi}{4\beta_{2}}\Bigl[
x_{1}^2(1-\frac{\nu^2}{2\alpha})+
(\beta_{2}x_{2})^2(1+\frac{\nu^2}{2\alpha})\Bigr]
\right)}
\exp{(-i\frac{\pi\nu^2\gamma }{4\alpha }x_{1}x_{2})}.
\end{equation}}
The proof of the proposition is complete.
\end{proof}
\begin{oss}{\rm
We note that in the isotropic case $\nu=0$, we have
$\beta_{2}=1,\gamma=4$, recovering  \eqref{1.isolll}
($ f(x_{1}+ix_{2}) e^{-\pi(x_{1}^2+x_{2}^2)}$) for $\omega=1$.
On the other hand, the reader may have noticed that
it seems difficult to guess the above definition without going through the explicit computations
on the diagonalization of $q$ of the previous sections.}
\end{oss}
\subsection{The energy in the anisotropic $LLL$}
\begin{lem}
The $LLL$ is defined by the proposition \ref{4.pro.LLL} and the
Gross-Pitaevskii energy by \eqref{1.engrpi}. For $u\in LLL$, we have
\begin{multline}\label{4.egpdeu}
E_{GP}(u)=\frac 12\int_{\R^2}
\left(\frac{2\alpha}{\alpha+2\omega^2+\nu^2} \ep^2 x_{1}^2
+\frac{2\alpha(2\nu^2+\ep^2)}{\alpha-\nu^2+2\omega^2}x_{2}^2\right)
\val{u(x_{1}, x_{2})}^2 dx_{1} dx_{2}
\\+\frac g2\int_{\R^2}\val{u(x_{1}, x_{2})}^4 dx_{1} dx_{2}
 +\frac{\mu_2}{4\pi} - \frac{\mu_1}{8\pi}\left(\beta_1\beta_2 +
  \frac1{\beta_1\beta_2} \right).
\end{multline}
\end{lem}{\em Proof.}
In the $LLL$,
one can simplify the energy. We define
\begin{align*}
A_{2}=M(\eta_{2}-i\mu_{2}y_{2})^w M^{*} &= \mu_2\left(\lambda_1d^{-1} D_{x_1} + c \lambda_1 x_2
\right) + i \left(\lambda_2 D_{x_2} - (d\lambda_1^{-1} - \lambda_2
    cd )x_1 \right),\\
A_{1}=M(\eta_{1}-i\mu_{1}y_{1})^w M^{*}  &= \mu_1\left(\lambda_2d^{-1}D_{x_2} + c\lambda_2x_1 \right) +
i \left((\lambda_1 cd - d\lambda_2^{-1})x_2 + \lambda_1 D_{x_1} \right),
\end{align*}
which satisfy the canonical commutation relations: $\left[A_j ,
A_j^*\right] = \mu_j/\pi,$ while all other commutators vanish.
We have proven that
$$
q^w= A_{1}^* A_{1}+A_{2}^* A_{2}+\frac{\mu_{1}+\mu_{2}}{2\pi}=(\re A_{1})^2
+(\im A_{1})^2
+(\re A_{2})^2
+(\im A_{2})^2
$$
and the $LLL$ is defined by the equation $A_{2}u=0$.
On the other hand, we have
$$
   d\mu_{1}^{-1}\re A_{1} -\im A_{2}= d\lambda_1^{-1}x_1, \quad
   d {\mu_2}^{-1} \re A_{2} - \im A_{1}= d\lambda_2^{-1}x_2,
$$
and thus for $u\in LLL$, since $A_{2}u=0$ , using the commutation relations of the $A_{j}$'s,
one gets
\begin{align*}
d^2\lambda_1^{-2}x_1^2&=d^2\mu_{1}^{-2}(\re A_{1})^2+
((A_{2}-A_{2}^*)/2i)^2+2d\mu_{1}^{-1}(\re A_{1})(A_{2}-A_{2}^*)/2i
\\
&=d^2\mu_{1}^{-2}(\re A_{1})^2+\frac{\mu_{2}}{4\pi},
\end{align*}
and similarly,
\begin{align*}
d^2\lambda_2^{-2}x_2^2&=d^2\mu_{2}^{-2}((A_{2}+A_{2}^*)/2)^2+
(\im A_{1})^2
\\&=(\im A_{1})^2+\frac{d^2}{4\pi\mu_{2}}.
\end{align*}
As a result, we get on the $LLL$,
$$\mu_1^2 \lambda_1^{-2}x_1^2 + d^2 \lambda_2^{-2} x_2^2
 =  (\re A_{1})^2
+(\im A_{1})^2
+\frac{d^2}{4\pi\mu_{2}}
+\frac{\mu_{2}\mu_{1}^2}{4\pi d^2},
$$
and
$
q^w=\mu_1^2 \lambda_1^{-2}x_1^2 + d^2 \lambda_2^{-2} x_2^2
-\frac{d^2}{4\pi\mu_{2}}
-\frac{\mu_{2}\mu_{1}^2}{4\pi d^2}+\frac{\mu_{2}}{2\pi},
$
so that
\begin{eqnarray*}
2E_{GP}(u)&=&
  \frac \gamma 2  \int_{\R^2} \left(
  \mu_1\beta_1 x_1^2 + \frac{\mu_1}{\beta_1} x_2^2\right) |u(x_1,x_2)|^2
  dx_1 dx_2 \\
&&  +g\int_{\R^2} \val{u(x_{1},x_{2})}^4 dx_{1}d x_{2}
  \\
 && +\frac{\mu_2}{2\pi} - \frac{\mu_1}{4\pi}\left(\beta_1\beta_2 +
  \frac1{\beta_1\beta_2} \right),
\end{eqnarray*}
for any $u\in LLL$, that is, satisfying \eqref{4.df.LLL}.
We note that
\begin{equation*}
\frac{\gamma\mu_{1}\beta_{1}}{2}=
\frac{2\alpha}{\alpha+2\omega^2+\nu^2} \ep^2,
\text{ (coefficient of $x_{1}^2$)}\qquad
\frac{\gamma\mu_{1}}{2\beta_{1}}=
\frac{2\alpha(2\nu^2+\ep^2)}{\alpha-\nu^2+2\omega^2},
\text{  (coefficient of $x_{2}^2$)   }.
\end{equation*}
\begin{defi}\label{4.def.lllnew}
For $u\in LLL$ (see the proposition \ref{4.pro.LLL}), we define
\begin{equation}
 {\mathcal E}_{LLL}(u)=\frac 12\int_{\R^2}(\ep^2 x_{1}^2+ \kappa_1^2 x_{2}^2)
\val{u(x_{1},x_{2})}^2
dx_{1}dx_{2}+\frac{g_{1}}2\int_{\R^2}\val{u(x_{1}, x_{2})}^4 dx_{1}
dx_{2},
\end{equation}
with
\begin{equation}\label{once421}
\kappa_1^2=\frac{(\alpha+2\omega^2+\nu^2)(2\nu^2+\ep^2)}{\alpha-\nu^2+2\omega^2},\quad
g_{1}=g\frac{\alpha+2\omega^2+\nu^2}{2\alpha},
\quad
\alpha=\sqrt{\nu^4+4\omega^2}.
\end{equation}We note that,
from \eqref{4.egpdeu},
\begin{equation}\label{4.egpuu}
E_{GP}(u)=\frac{2\alpha}{\alpha+2\omega^2+\nu^2}\mathcal E_{LLL}(u)
+\frac{\mu_2}{4\pi} - \frac{\mu_1}{8\pi}\left(\beta_1\beta_2 +
  \frac1{\beta_1\beta_2}\right).
\end{equation}
\end{defi}
\begin{oss}{\rm
Since $\alpha^2=\nu^4+4\omega^2$, we see
 that
\begin{equation}
(2\nu^2+\ep^2)\bigl(1+
\frac{2\nu^2}{\alpha-\nu^2+2\omega^2}
\bigr)=
\kappa^2=\frac{(\alpha+2\omega^2+\nu^2)(2\nu^2+\ep^2)}{\alpha-\nu^2+2\omega^2}\ge 2\nu^2+\ep^2,
\end{equation}
and
$
\kappa^2=\ep^2\Longleftrightarrow \nu=0.
$
}
\end{oss}
\begin{oss}{ \rm We stay away from the case where $\omega=0$
and shall always assume $\omega>0$.
In the case $\omega=0$,
the quadratic part of the energy is diagonal
and the $LLL$
is,
$$
v_{1}(x_{1})\otimes 2^{1/4}(2-\ep^2)^{1/8} e^{-\pi (2-\ep^2)^{1/2} x_{2}^2},
$$
and we get a 1D problem on the function $v_{1}$.
}
\end{oss}
\subsection{The (final)
reduction to a simpler lowest Landau level}
Given the fact that in \eqref{4.df.LLL}, we can write
$F(x_{1}+i\beta_{2}x_{2})$ as a holomorphic function times
$e^{-\delta z^2}$, with $\delta=\gamma\pi\nu^2/(8\beta_2\alpha)$,
and that the energy ${\mathcal E}_{LLL}$  depends only on the modulus
of $u$ and not on its phase, it is equivalent to minimize ${\mathcal
E}_{LLL}$ on the $LLL$ or on the space
$$f\bigl(x_{1}+i\beta_{2}x_{2}\bigr)
\exp{\left(-\frac{\gamma\pi}{4\beta_{2}}\Bigl[ x_{1}^2+
(\beta_{2}x_{2})^2\Bigr] \right)}
,\quad\text{with $f$ entire.}
$$ A rescaling in $x_1$ and $x_2$
yields the space of the introduction with
\begin{equation}\label{4.changeu}
u(x_1,x_2)=\sqrt{\frac
\gamma 2} v(y_1,y_2),\quad y_1=x_1\sqrt{\frac \gamma {2\beta_2}},\
y_2=x_2\sqrt{\frac {\gamma\beta_2} {2}},
\end{equation}
and, with $\Lambda_{0}$ given by \eqref{4.lllpi},
 the mapping
$LLL\ni u\mapsto v\in \Lambda_{0}$
is bijective and isometric.
With $\kappa_{1}, g_{1}$
given in the definition \ref{4.def.lllnew},
$\beta_{2}$ in \eqref{formulebeta2},
$\gamma$ in
\eqref{formulegamma},
we introduce
\begin{equation}\label{some424}
\kappa=\frac{\kappa_1}{\beta_2},\
g_0=\frac{g_1\gamma^2}{4\beta_2},
\end{equation}
and
\begin{equation}
E(v)=\frac12\int_{\R^2}(\epsilon^2y_{1}^2+\kappa^2 y_{2}^2)\val{v(y_{1},y_{2})}^2 dy_{1}dy_{2}+\frac{g_{0}}{2}\norm{v}_{L^4(\R^2)}^4.
\end{equation}
Using the transformation \eqref{4.changeu}, we have
\begin{equation}\label{newener}{\mathcal E}_{LLL}(u)=\frac{2\beta_2}\gamma E(v),
\end{equation}
so that,
via the definition \ref{4.def.lllnew}, we are indeed reduced to the minimization of
\eqref{eq:nrjLLL2} in the space $\Lambda_{0}$ (given in \eqref{lll}) under  the constraint
$\norm{u}_{L^2(\R^2)}=1$.
We note also that the quantities
\begin{align}
\frac{2\alpha}{\alpha+2\omega^2+\nu^2}, \frac{2\beta_{2}}{\gamma},\quad
\text{\footnotesize (factors of $\mathcal E_{LLL}(u)$ in \eqref{4.egpuu} and $E(v)$ in \eqref{newener})},
\\
\text{and }
\beta_{2},
\frac{\gamma^2}{\beta_{2}},\quad \frac{\alpha+2\omega^2+\nu^2}{2\alpha}
 \text{\footnotesize (factors of $\kappa $ in \eqref{some424}, of  $g_{1}$ in \eqref{some424}), of
 $g$ in \eqref{once421} },
 \end{align}
 are bounded and away from zero as long as $\omega$ stays away from zero,
 a condition that we shall always assume, say $0<\omega_{0}\le \omega\le 1$.
\section{Weak anisotropy}
This section is devoted to the proof of
Theorem~\ref{th:cv-faible}.  We assume $\ep \leq \kappa \ll \ep^{1/3}$.
 The isotropic case is recovered by assuming $\kappa = \ep.$
We  first give some approximation results in
subsection~\ref{ssec:approx-faible}, and prove the theorem  in
subsection~\ref{ssec:nrj-bounds-faible}.
\par
We recall that the space $\Lambda_0$, the operator $\Pi_0$, the energy
$E$ and the minimization problem $I(\ep,\kappa)$ are defined by
\eqref{lll}, \eqref{eq:projLLL}, \eqref{eq:nrjLLL2} and
\eqref{eq:minLLL}, respectively.
An important test function will be \eqref{eq:utau}, namely
\begin{equation}\label{eq:utau2}
u_\tau(x_1,x_2) = e^{\frac{\pi}2 \left(z^2 - |z|^2\right)} \Theta\left
  (\sqrt{\tau_I}z, \tau\right), \quad z = x_1 + ix_2,  .
\end{equation} for $\tau =
\tau_R + i\tau_I = e^{\frac{2i\pi}3}$.
\subsection{Approximation results}
\label{ssec:approx-faible}
\begin{lem}
\label{lm:approx1}
Let $u(x) = f(x_1 + ix_2) e^{-\frac{\pi}{2}|x|^2} \in L^\infty(\R^2)$,
with $f$ holomorphic. Assume  $0\leq \beta
\leq 1$ and let $p\in C^{0,\beta}(\R^2)$ be such that $\supp(p) \subset
B_S$ the Euclidean ball of radius $S>0$ and of center $0$. Define
\begin{equation}
  \label{eq:pdilate}
  \rho (x) = \frac 1 {\sqrt{R_1 R_2}} p\left(\frac{x_1}{R_1},
    \frac{x_2}{R_2} \right).
\end{equation}
Then, for any $r\geq 1,$ there exists a constant $C_{S,r}>0$ depending only on
$S$ and $r$ such that, setting $R =
\min(R_1,R_2)$, we have,
\begin{equation}
  \label{eq:approx}
  \left\| \Pi_{0}\left(\rho u \right) - \rho u\right\|_{L^r(\R^2)} \leq
    C_{S,r} \|u\|_{L^\infty(\R^2)} \|p\|_{C^{0,\beta}(\R^2)}  \frac
    {\left(R_1 R_2\right)^{\frac 1 r - \frac 1 2}} {R^{\beta}}.
\end{equation}
\end{lem}
\begin{proof}We first prove the lemma in the case
$\beta=0$. For this purpose, we write
$$\left| \Pi_{0}(\rho u) \right| \leq \int_{\R^2} e^{-\frac{\pi}2 |x-y|^2}
|u(y)| |\rho(y)| dy.$$
Young's inequality implies, for any $r\geq 1$ and any $p,q\geq 1$
such that $1/p  + 1/q = 1+ 1/r,$
$$\left\| \Pi_{0}(\rho u) \right\|_{L^r} \leq \left\|
e^{-\frac{\pi}2|x|^2} \right\|_{L^p} \left\| u \rho \right\|_{L^q} \leq
\|u\|_{L^\infty}\left\|
e^{-\frac{\pi}2|x|^2} \right\|_{L^p} \left\| \rho \right\|_{L^q} .$$
Fixing $q=r$, hence $p=1$, we find
\begin{equation}
  \label{eq:approx1}
  \left\| \Pi_{0}(\rho u) \right\|_{L^r} \leq 2 \|u\|_{L^\infty}
  \left\| \rho \right\|_{L^r} =2 \|u\|_{L^\infty} \left(R_1
    R_2\right)^{\frac 1 r - \frac 1 2} \|p\|_{L^r}.
\end{equation}
This proves \eqref{eq:approx} for $\beta = 0.$
\par
Next, we assume $\beta = 1$. We use a Taylor expansion
of $\rho(y) = \rho(x + y-x)$ around $x$:
\begin{multline*}
\rho(y) = \rho(x) \\+
  \frac 1 {\sqrt{R_1 R_2}} \int_0^1 \nabla p\left(\frac{x_1}{R_1} + t
    \frac{y_1 - x_1}{R_1} ,
    \frac{x_2}{R_2} + t\frac{y_2 - x_2}{R_2}  \right) \cdot \left(
  \frac{y_1 - x_1}{R_1},\frac{y_2 - x_2}{R_2} \right) dt.
\end{multline*}
We then notice that, although $u\notin \Lambda_0$ a priori,
it belongs to $\mathscr K$ (see the proposition \ref{4.pro.focksprime})
and we have
$\Pi_{0}(u) = u$ since  $u\in L^\io$ and $u(x) = f(x_1 +
ix_2) \exp(-\pi |x|^2/2)$
with $f$ holomorphic. Hence, we have
\begin{multline*}
  \Pi_{0}(\rho u) - \rho u = \int_{B_{S+1}^{R_1,R_2}} e^{-\frac{\pi}2 |x-y|^2
    + i\pi \left(x_2 y_1 - y_2 x_1\right)} u(y_1,y_2)\\
\times  \frac 1 {\sqrt{R_1 R_2}} \int_0^1 \nabla p\left(\frac{x_1}{R_1} + t
    \frac{y_1 - x_1}{R_1} ,
    \frac{x_2}{R_2} + t\frac{y_2 - x_2}{R_2}  \right) \cdot \left(
  \frac{y_1 - x_1}{R_1},\frac{y_2 - x_2}{R_2} \right) dt dy, \\
- \rho(x) \int_{\left(B_{S+1}^{R_1,R_2}\right)^c} u(y) e^{-\frac{\pi}2 |x-y|^2
    + i\pi \left(x_2 y_1 - y_2 x_1\right)}dy
\end{multline*}
where the set $B_{S+1}^{R_1,R_2}$ is
\begin{equation}\label{eq:BSR}
B_{S+1}^{R_1,R_2} = \left\{ (y_1,y_2) = (R_1 t_1, R_2 t_2), \quad t\in
  B_{S+1}\right\}.
\end{equation}
We thus have, with $R=\min(R_{1},R_{2})$,
\begin{multline}\label{eq:approx1-1}
\left|  \Pi_{0}(\rho u) - \rho u \right| \leq\left\| \nabla p
  \right\|_{L^\infty} \int_{B_{S+1}^{R_1,R_2}}
e^{-\frac{\pi}2 |x-y|^2} |u(y)| \frac 1 {\sqrt{R_1 R_2}} \frac{|y-x|}R
  dy \\
 + |\rho(x)| \int_{\left(B_{S+1}^{R_1,R_2}\right)^c} |u(y)| e^{-\frac{\pi}{2}|x-y|^2} dy.
\end{multline}
We bound the first term of the right-hand side of \eqref{eq:approx1-1}
using Young's inequality, while for the second term, we have, $\forall
x\in \supp(\rho)\subset B_S^{R_1,R_2},$
\begin{multline*}
\int_{\left(B_{S+1}^{R_1,R_2}\right)^c} |u(y)|
e^{-\frac{\pi}{2}|x-y|^2} dy \leq \|u\|_{L^\infty}
e^{-\frac{\pi}4 R^2} \int_{\R^2} e^{-\frac \pi 4 |x-y|^2} dy \\
=
  4\|u\|_{L^\infty}e^{-\frac{\pi}4 R^2} \leq \|u\|_{L^\infty} \frac C R,
\end{multline*}
where $C$ is a universal constant. Hence, we have
\begin{eqnarray*}
\left\|  \Pi_{0}(\rho u) - \rho u \right\|_{L^r} &\leq& \frac 1 R
\left\| \nabla p \right\|_{L^\infty}
\left\| |y| e^{-\frac{\pi}2 |y|^2}\right\|_{L^1} \| u \|_{L^\infty}
\frac{1}{\sqrt{R_1R_2}} |B_{S+1}^{R_1,R_2}|^{1/r} \\
&&\hs\hs\hs\hs + \frac{C}{R}\|u\|_{L^\infty}\left\|\rho \right\|_{L^r} \\
&=&\frac 1 R \left\| \nabla p
  \right\|_{L^\infty} \sqrt 2  \| u \|_{L^\infty}
(R_1R_2)^{\frac 1 r - \frac 1 2} |B_{S+1}|^{1/r} \\
&& \hs\hs\hs\hs + \frac{C} R \|u\|_{L^\infty} \|p\|_{L^\infty} (R_1R_2)^{\frac 1 r -
  \frac 1 2} |B_{S}|^{1/r}.
\end{eqnarray*}
This gives \eqref{eq:approx} for $\beta=1$. We then conclude by a real
interpolation argument between $C^0$ and $C^{0,1}$.
\end{proof}
A comment is in order here: we have chosen to state Lemma~\ref{lm:approx1} with a
general function $p$. However, since our aim is to apply the above result
with the special case $p(x) = \left(1-|x|^2 \right)_+^{1/2}$, it is also
possible to use explicitly this value of $p$ in order to give a simpler
proof of the above result. The method would then be to prove the
estimate for $r=+\infty$ first, then for $r=1$, and then use an
interpolation argument between $L^1$ and $L^\infty.$ For instance, the
proof of the $r=+\infty$ case would go as follows:
\begin{eqnarray*}
\left| \Pi_{0}(\rho u )(x) - \rho(x) u(x)\right| &=&\left| \int_{\R^2}
e^{-\frac{\pi}2 |x-y|^2 + i\pi(x_2y_1 - y_2x_1)} \left( \rho(y)u(y) -
  \rho(x)u(y) \right)dy \right| \\
&\leq&
\|u\|_{L^\infty} \int_{\R^2} e^{-\frac{\pi}2 |x-y|^2} \left|
\rho(y) - \rho(x) \right|dy \\
&\leq & \|u\|_{L^\infty} \int_{\R^2} e^{-\frac{\pi}2 |x-y|^2}
\sqrt{\frac{|x-y|}{R}} dy \\
&=& \frac{\|u\|_{L^\infty}}{\sqrt R} \int_{\R^2}
e^{-\frac\pi 2 |y|^2}\sqrt{|y|} dy.
\end{eqnarray*}
The proof of the case $r=1$ is slightly more involved, but is based on
the same idea.
\bigskip
We now prove
\begin{lem}
\label{lm:approx2}
With the same hypotheses as in Lemma~\ref{lm:approx1}, we have, for any
$s\geq 1,$
\begin{equation}
  \label{eq:approx21}
  \left(\int_{\R^2} x_1^{2s} \left|\Pi_{0}(\rho u) - \rho u\right|^2
  \right)^{1/2} \!\leq C_{S,s} \|u\|_{L^\infty(\R^2)}
  \|p\|_{C^{0,\beta}(\R^2)} \frac{1+R_1^s S^s} {R^{\beta}},
\end{equation}
and
\begin{equation}
  \label{eq:approx22}
  \left(\int_{\R^2} x_2^{2s} \left|\Pi_{0}(\rho u) - \rho u\right|^2
  \right)^{1/2} \! \leq C_{S,s} \|u\|_{L^\infty(\R^2)} \|p\|_{C^{0,\beta}(\R^2)}
  \frac{ (1+R_2^s S^s)} {R^{\beta}},
\end{equation}
where $C_{S,s}$ depends only on $S$ and $s$.
\end{lem}
\begin{proof}
  Here again, we first deal with the case $\beta=0$. For this purpose,
  we write:
\begin{multline}\label{eq:approx21-1}
|x_1|^s \left| \Pi_{0}(\rho u) \right| \leq 2^{s-1} \int_{\R^2} |x_1 - y_1|^s
e^{-\frac{\pi}2 |x-y|^2} |u(y)| \rho(y) dy\\
+ 2^{s-1}\int_{\R^2} |y_1|^s
e^{-\frac{\pi}2 |x-y|^2} |u(y)| \rho(y) dy ,
\end{multline}
where we have used the inequality $(a+b)^s \leq 2^{s-1}(a^s+b^s)$, valid
for any $a,b\geq 0, s\ge 1.$
The first line of \eqref{eq:approx21-1} is dealt with exactly as in the
proof of Lemma~\ref{lm:approx1}, leading to \eqref{eq:approx1} with
$r=2$, which reads here
\begin{multline}\label{eq:approx21-2}
\left\|\int_{\R^2} |x_1 - y_1|^s
e^{-\frac{\pi}2 |x-y|^2} |u(y)| \rho(y) dy \right\|_{L^2} \leq
\|u\|_{L^\infty} \left\| |x|^s e^{-\frac{\pi}2 |x|^2}\right\|_{L^1}
\|\rho \|_{L^2} \\ \leq C_s \|u\|_{L^\infty} \|p\|_{L^2},
\end{multline}
where $C_s$ depends only on $s$.
The second line of \eqref{eq:approx21-1} is treated in the same way, but
$\rho(y)$ is replaced by $|y_1|^s\rho(y)$, that is, $p(y)$ is replaced
by $R_1^s |y_1|^s p(y).$ Hence,
we have
\begin{equation}\label{eq:approx21-3}
\left\|\int_{\R^2} |y_1|^s
e^{-\frac{\pi}2 |x-y|^2} |u(y)| \rho(y) dy \right\|_{L^2} \leq 2R_1^s
\|u\|_{L^\infty} \left\||y_1|^s p\right\|_{L^2}.
\end{equation}
Collecting \eqref{eq:approx21-1}, \eqref{eq:approx21-2} and
\eqref{eq:approx21-3}, we find
$$\left\||x_1|^s \Pi_{0}(\rho u) \right\|_{L^2} \leq
C_s(1+R_1^s S^s) \|u\|_{L^\infty} \|p\|_{C^0}|B_S|^{1/2}.
$$
This proves \eqref{eq:approx21} for $\beta=0$.
\par
Next, we consider the case $\beta = 1.$ Here again, we use a Taylor
expansion to obtain \eqref{eq:approx1-1}. This implies
\begin{align*}
|x_1|^s \left| \Pi_{0}(\rho u) - \rho u \right| \leq 2^{s-1} \frac{\left\| \nabla p
  \right\|_{L^\infty}}R \int_{B_{S+1}^{R_1,R_2}}
e^{-\frac{\pi}2 |x-y|^2} |u(y)| \frac 1 {\sqrt{R_1 R_2}}|y-x||y_1-x_1|^s
  dy \\
+  2^{s-1}\frac{\left\| \nabla p
  \right\|_{L^\infty}}R \int_{B_{S+1}^{R_1,R_2}}
e^{-\frac{\pi}2 |x-y|^2} |u(y)| \frac 1 {\sqrt{R_1 R_2}}|y-x||y_1|^s
  dy \\
+ |x_1|^s  |\rho(x)| \int_{\left(B_{S+1}^{R_1,R_2}\right)^c} |u(y)| e^{-\frac{\pi}{2}|x-y|^2} dy,
\end{align*}
where $B_{S+1}^{R_1,R_2}$ is defined by \eqref{eq:BSR}.
We use Young's inequality again, finding
\begin{align*}
  \left\| |x_1|^s  \left|\Pi_{0}(\rho u) - \rho u\right| \right\|_{L^2}
    \leq 2^{s-1} \frac{\left\| \nabla p \right\|_{L^\infty}}R \left\| |y|^{s+1}
      e^{-\frac{\pi}2 |y|^2} \right\|_{L^1} \left( \frac{|B_{S+1}^{R_1,R_2}|}{R_1R_2}
    \right)^{1/2} \|u\|_{L^\infty} \\
 +2^{s-1}\frac{\left\| \nabla p \right\|_{L^\infty}}R \left\| |y|
      e^{-\frac{\pi}2 |y|^2} \right\|_{L^1} \left(\int_{B_{S+1}^{R_1,R_2}}
      \frac{|y_1|^{2s}}{R_1R_2} dy  \right)^{1/2} \|u\|_{L^\infty} \\
+ \frac C R \|u\|_{L^\infty} \left\| |x_1|^s \rho \right\|_{L^2},
\end{align*}
where $C$ is a universal constant.
Hence,
\begin{eqnarray*}
\left\| |x_1|^s  \left|\Pi_{0}(\rho u) - \rho u\right| \right\|_{L^2}
&\leq & C_{S,s}  \frac{\left\|  p \right\|_{C^1}}R
\left(1 + R_1^s S^s\right)\|u\|_{L^\infty}.
\end{eqnarray*}
This gives \eqref{eq:approx21} in the case $\beta=1$. Here again, we
conclude with a real interpolation argument. The proof of
\eqref{eq:approx22} follows  the same lines.
\end{proof}
\subsection{Energy bounds}
\label{ssec:nrj-bounds-faible}
\begin{pro}\label{pr:bsup}
  Let $\tau \in \C\setminus\R$, let $p\in C^{0,1/2}(\R^2)$ be such that
  $\supp(p) \subset K$ for some compact set $K$, and $\int |p|^2 = 1.$
  Consider $u_\tau$ as defined by \eqref{eq:utau}, and define
  \begin{equation}
    \label{eq:ansatz2}
    v = \left\| \Pi_{0}(\rho u_\tau )\right\|_{L^2(\R^2)}^{-1}
    \Pi_{0}(\rho u_\tau),
  \end{equation}
where $\rho$ is given by
\begin{equation}
  \label{eq:alpha}
  \rho(x) = \frac 1 {\sqrt{R_1 R_2}} p\left(\frac{x_1}{R_1},
    \frac{x_2}{R_2} \right), \quad  R_1 =
\left(\frac{4g_0\kappa}{\pi\ep^3}\right)^{1/4}, \quad R_2 =
\left(\frac{4g_0\ep}{\pi\kappa^3}\right)^{1/4}.
\end{equation}
Then we have, with $E(u)$ defined by \eqref{eq:nrjLLL2}
\begin{equation}
  \label{eq:bsup}
  E(u) = \sqrt{\frac{2g \ep \kappa}{\pi}}\left( \int_{\R^2}
    \frac12 |x|^2 |p(x)|^2 + \frac{\pi \gamma(\tau)}4 |p(x)|^4\right) +
  O\left(\sqrt{\ep\kappa}\left(\frac{\kappa^3}{\ep}\right)^{1/8}\right),
\end{equation}
for $(\ep, \kappa\ep^{-1/3})\rightarrow (0,0)$,
where $\gamma(\tau)$ is given by \eqref{eq:dfgamma}.
\end{pro}
\begin{nb}
{\rm  The $L^\io$
function $\rho u_{\tau}$
does not belong to $\Lambda_{0}$
since it is compactly supported and not identically 0;
as a result, $\norm{\Pi_{0}(\rho u_{\tau})}_{L^2}\not=0$
and $v$ makes sense.}
\end{nb}
\begin{proof}
  First note that $R = \min(R_1,R_2) = R_2$, and
  that Lemma~\ref{lm:approx1} with $r=2$ implies
\begin{equation}\label{eq:bsup4}
\left| \|\Pi_{0}(\rho
    u_\tau)\|_{L^2} - \|\rho u_\tau\|_{L^2} \right| \leq C R^{-1/2} =
  C \left(\frac{\kappa^3}\ep\right)^{1/8}.
\end{equation}
We then apply Lemma~\ref{lm:approx2} for $s=1, \beta=1/2$, finding
\begin{eqnarray*}
\left| \int_{\R^2} x_1^2 \left|\Pi_{0}(\rho u_\tau)\right|^2 -
  \int_{\R^2} x_1^2 |\rho|^2 |u_\tau|^2 \right| &\leq& C \bigl(\|x_1
\Pi_{0}(\rho u_\tau)\|_{L^2}
+ \| x_1 \rho u_\tau\|_{L^2}\bigr)
\frac{1+R_1}{ R^{1/2}} \\
&\leq &  C \left(2\| x_1 \rho u_\tau\|_{L^2} + C\frac{1+R_1}{R^{1/2}}\right)
\frac{1+R_1}{ R^{1/2}}.
\end{eqnarray*}
We also compute
$$\int_{\R^2} x_1^2 |\rho(x)|^2 |u_\tau(x)|^2 dx \leq R_1^2
\|u_\tau\|_{L^\infty}^2 \int_{\R^2}
x_1^2 |p(x)|^2 dx \leq C R_1^2.$$
Hence, we get
\begin{equation}
  \label{eq:bsup1}
  \frac{\ep^2}2\left| \int_{\R^2} x_1^2 \left|\Pi_{0}(\rho u_\tau)\right|^2 -
  \int_{\R^2} x_1^2 |\rho|^2 |u_\tau|^2 \right| \leq C\ep^2
\frac{1+R_1^2}{R^{1/2}} \leq C \sqrt{\ep \kappa}
\left(\frac{\kappa^3}{\ep}\right)^{1/8}.
\end{equation}
A similar argument allows to show that
\begin{equation}
  \label{eq:bsup2}
  \frac{\kappa^2}2\left| \int_{\R^2} x_2^2 \left|\Pi_{0}(\rho u_\tau)\right|^2 -
  \int_{\R^2} x_2^2 |\rho|^2 |u_\tau|^2 \right| \leq C\kappa^2
\frac{1+R_2^2}{R^{1/2}} \leq C \sqrt{\ep \kappa}
\left(\frac{\kappa^3}{\ep}\right)^{1/8}.
\end{equation}
Turning to the last term of the energy, we apply Lemma~\ref{lm:approx1}
again, with $r=4, \beta=1/2$, finding
\begin{eqnarray*}
  \left| \int_{\R^2} \left|\Pi_{0}(\rho u_\tau)\right|^4 -
    \int_{\R^2} |\rho u_\tau|^4 \right| &\leq & 2
  \left(\left\|\Pi_{0}(\rho u_\tau)\right\|_{L^4}^3 + \left\|\rho
      u_\tau\right\|_{L^4}^3 \right) 
\left\| \Pi_{0}(\rho u_\tau) -
    \rho u_\tau \right\|_{L^4} \\
&\leq & C  \left\|\rho u_\tau\right\|_{L^4}^3
(R_1R_2)^{-1/4} R^{-1/2}.
\end{eqnarray*}
In addition, we have
$$\int_{\R^2} |\rho u_\tau|^4 \leq \|u_\tau\|_{L^\infty}^4
\int_{\R^2}|\rho|^4 = \|u_\tau\|_{L^\infty}^4
\left(R_1R_2\right)^{-1} \int_{\R} p^4.$$
Hence, we obtain
\begin{equation}
  \label{eq:bsup3}
  \left| \int_{\R^2} \left|\Pi_{0}(\rho u_\tau)\right|^4 -
    \int_{\R^2} |\rho u_\tau|^4 \right| \leq C
  \left(R_1R_2\right)^{-1} R^{-1/2} \leq C \sqrt{\ep \kappa}
\left(\frac{\kappa^3}{\ep}\right)^{1/8}.
\end{equation}
Combining \eqref{eq:bsup1}, \eqref{eq:bsup2} and \eqref{eq:bsup3}, we
have
$$E\left(\Pi_{0}(\rho u_\tau)\right) = E(\rho u_\tau) \left[
1+ O\left(\left(\frac{\kappa^3}{\ep}\right)^{1/8}\right) \right].$$
Hence, with the help of \eqref{eq:bsup4}, we get
$$E(v) = E\left(\frac{\rho u_\tau}{\|\rho
    u_\tau\|_{L^2}}\right) \left[1+  O\left( \left(\frac{\kappa^3}{\ep}\right)^{1/8}\right) \right].$$
Finally, we estimate the terms of $E(\rho u_\tau / \|\rho
    u_\tau\|_{L^2} )$: using real interpolation between $C^0$ and $C^{0,1}$,
    we obtain
\begin{multline}\label{eq:bsup8}
\| \rho u_\tau \|_{L^2}^2 = \int_{\R^2} |p(x)|^2 |u_\tau(R_1 x_1,
R_2 x_2)|^2 dx \\= \fint |u_\tau|^2 + O\left(\frac 1 {R^{1/2}}\right) =
\fint |u_\tau|^2 + O\left( \left(\frac{\kappa^3}\ep\right)^{1/8}
\right).
\end{multline}
Moreover, we have
\begin{eqnarray}
\label{eq:bsup5}
\frac{\ep^2}2 \int_{\R^2} x_1^2 |\rho|^2 |u_\tau|^2 &=&\frac{\ep^2}2
R_1^2\left[\fint |u_\tau|^2 + O\left(
    \left(\frac{\kappa^3}\ep\right)^{1/8}\right)  \right]\int_{\R^2}
x_1^2 |p(x)|^2dx, \\
\label{eq:bsup6}
\frac{\kappa^2}2 \int_{\R^2} x_2^2 |\rho|^2 |u_\tau|^2 &=&
\frac{\kappa^2}2 R_2^2  \left[ \fint |u_\tau|^2 + O\left(
  \left(\frac{\kappa^3}\ep\right)^{1/8}\right) \right]\int_{\R^2} x_2^2 |p(x)|^2dx, \\
\label{eq:bsup7}
\frac{g}{2} \int_{\R^2} |\rho|^4 |u_\tau|^4 &=& \frac{g}{2R_1R_2} \left[\fint
|u_\tau|^4 + O\left(
  \left(\frac{\kappa^3}\ep\right)^{1/8}\right) \right] \int_{\R^2} |p|^4.
\end{eqnarray}
Thus, collecting \eqref{eq:bsup8}, \eqref{eq:bsup5}, \eqref{eq:bsup6}
and \eqref{eq:bsup7},
\begin{eqnarray*}
E(u) &=& \left[\frac{\ep^2}2 R_1^2\int_{\R^2}
x_1^2 |p(x)|^2dx + \frac{\kappa^2}2 R_2^2\int_{\R^2} x_2^2 |p(x)|^2dx
\right. \\
&&\left. +\frac{\fint |u_\tau|^4}{\left(\fint |u_\tau|^2\right)^2}
\frac{g_0}{2R_1
  R_2} \int_{\R^2} |p|^4 \right]\left[1+  O\left(\left(\frac{\kappa^3}{\ep}\right)^{1/8}\right)\right] \\
&=& \sqrt{\frac{2g_0\ep \kappa}{\pi}} \left(\int_{\R^2} \frac12
\left(x_1^2 +
  x_2^2\right) |p(x)|^2 + \frac{\pi\gamma(\tau)}{4} |p|^4\right) \\
&&\left[1 +
  O\left(\left(\frac{\kappa^3}{\ep}\right)^{1/8}\right)\right]. \\
&=& \sqrt{\frac{2g_0\ep \kappa}{\pi}} \left(\int_{\R^2} \frac12
\left(x_1^2 +
  x_2^2\right) |p(x)|^2 + \frac{\pi\gamma(\tau)}{4} |p|^4\right) \\
&&+  O\left(\sqrt{\ep\kappa}\left(\frac{\kappa^3}{\ep}\right)^{1/8}\right).
\end{eqnarray*}
\end{proof}
\begin{proof}[Proof of Theorem~\ref{th:cv-faible}:] We first prove the lower
bound in \eqref{eq:limnrjfaible}: this is done by noticing that
$$J(\ep,\kappa) \leq I(\ep,\kappa),$$
where
$$J(\ep,\kappa) = \inf\left\{E(u), \quad u \in L^2\left(\R^2,
    (1+|x|^2)dx\right)\cap L^4(\R^2), \quad \int_{\R^2} |u|^2 =
  1\right\}.$$
In addition, the minimizer of $J(\ep,\kappa)$ may be explicitly
computed (up to the multiplication by a complex function of modulus one):
\begin{equation}\label{eq:TF1}
u(x) = \sqrt\frac{2}{\pi R_1 R_2} \left(1 -\frac{x_1^2}{R_1^2}
- \frac{x_2^2}{R_2^2}  \right)_+^{1/2},
\end{equation}
with $R_1,R_2$ defined by \eqref{eq:alpha}.
Inserting \eqref{eq:TF1} in the energy, one finds the lower bound of
\eqref{eq:limnrjfaible}. In addition, the inverted parabola \eqref{eq:TF1} is
compactly supported, so it cannot be in $\Lambda_0$. Hence, the inequality is
strict.
\par
In order to prove the upper bound, we apply Proposition~\ref{pr:bsup},
with
$$p(x) = \sqrt{\frac 2 {\pi\sqrt{\gamma(\tau)}}} \left(1 -
    \frac{|x|^2}{\sqrt{\gamma(\tau)}}\right)^{1/2}_+,$$ and $\tau = j.$
  This corresponds to minimizing the leading order term of
  \eqref{eq:bsup} with respect to $\tau$ and $p$, with the constraint
  $\int |p|^2 = 1$.
\end{proof}
\section{Strong anisotropy}
We give in this Section the proof of Theorem~\ref{th:cv-fort}. We deal
here with the strongly asymmetric case that is, \eqref{eq:fort}, which
we recall here:
\begin{equation}
  \label{eq:fort2}
  {\kappa \gg \ep^{1/3}}
\end{equation}
We    first   prove    an   upper    bound   for    the    energy   in
Subsection~\ref{ssec:borne-sup-fort},    then     a    lower    bound    in
Subsection~\ref{ssec:borne-inf-fort},    and   conlude    the    proof   in
Subsection~\ref{ssec:fin-fort}
\subsection{Upper bound for the energy}
\label{ssec:borne-sup-fort}
\begin{lem}
  \label{lm:ftest-fort-1}
Assume that $\rho\in L^2(\R)$. Then the function
\begin{equation}
  \label{eq:ftest-fort}
  u(x_1,x_2) = \frac 1 {2^{1/4}} e^{-\frac{\pi}{2}
x_2^2 } \int_{\R} e^{-\frac{\pi}{2} \left((x_1-y_1)^2 - 2iy_1x_2 \right)}
\rho(y_1)dy_1 ,
\end{equation}
satisfies $u\in \Lambda_0.$
\end{lem}
\begin{proof}
  We first write
$$u(x_1,x_2)e^{\frac{\pi}2 \left(x_1^2 + x_2^2\right)}  = \frac 1 {2^{1/4}}\int_\R
e^{-\frac{\pi}2 \left(y_1^2 - 2(x_1 + i x_2)y_1\right)}\rho(y_1) dy_1,$$
which is a holomorphic function of $x_1 + i x_2$. In addition, we have
$$|u(x_1,x_2) | \leq \frac 1 {2^{1/4}} e^{-\frac{\pi}2 x_2^2} \left| \rho *
  e^{-\frac{\pi}2 y_1^2}\right|(x_1),$$
Hence, using Young's inequality, we get
$$\|u\|_{L^2(\R^2)} \leq \frac 1 {2^{1/4}}\|\rho\|_{L^2(\R)} \left\| e^{-\frac{\pi}2
    y_1^2} \right\|_{L^1(\R)} = 2^{1/4} \|\rho\|_{L^2(\R)},$$
hence $u\in L^2(\R^2).$
\end{proof}
\begin{lem}
\label{lm:ftest-fort-2}
Let $p\in C^2(\R)$ have compact support with $\supp(p)\subset(-T,T)$,
and consider the function
\begin{equation}
  \label{eq:alpha-fort}
  \rho(t) = \frac 1 {\sqrt{R}} p\left( \frac t R \right).
\end{equation}
Then, for any $r\geq 1,$ there exists a constant $C_{r}$ depending
only on $r$ such that the function $u$ defined by
\eqref{eq:ftest-fort} satisfies, for $R\geq 1,$
\begin{multline}
  \label{eq:dl-fort}
 \left\| u(x_1,x_2) - 2^{1/4} \rho(x_1) e^{-\pi x_2^2 + i\pi x_1 x_2}
   - i 2^{1/4} x_2\rho'(x_1) e^{-\pi x_2^2 + i\pi x_1 x_2}
 \right\|_{L^r(\R^2)} \\ \leq C_{r} T^{1/r}\frac{\|p''\|_{L^\infty(\R)}} {R^{5/2-1/r}}.
\end{multline}
\end{lem}
\begin{proof}
We use a Taylor expansion of $p\left(\frac{y_1}R\right)$ around $\frac{x_1}R$, that is,
\begin{multline}
\label{eq:taylor-fort}
p\left(\frac{y_1}R\right) = p\left(\frac{x_1}R\right) + \frac 1 R
p'\left(\frac{x_1}R\right)(y_1 - x_1) \\ + \frac 1 {R^2}(x_1 -
y_1)^2\int_0^1 (1-t)p''\left(\frac{x_1}R + \frac{t(y_1-x_1)}R \right) dt.
\end{multline}
In addition we have
$$\frac 1 {2^{1/4}} e^{-\frac{\pi}{2}
x_2^2} \int_{\R} e^{-\frac{\pi}{2} \left((x_1-y_1)^2 - 2iy_1x_2 \right)}
\frac 1 {\sqrt R} p\left(\frac{x_1}R\right)dy_1 = \frac 1 {\sqrt R}
p\left(\frac{x_1}R\right) 2^{1/4} e^{-\pi x_2^2 + i\pi x_1 x_2},$$
and
\begin{multline*}
\frac 1 {2^{1/4}} e^{-\frac{\pi}{2}
x_2^2 } \int_{\R} e^{-\frac{\pi}{2} \left((x_1-y_1)^2 - 2iy_1x_2 \right)}
\frac 1 {R^{3/2}} p'\left(\frac{x_1}R\right)(y_1 - x_1)dy_1 \\=
\frac 1 {R^{3/2}}
i 2^{1/4} x_2 p'\left(\frac{x_1}R\right)  e^{-\pi x_2^2 + i\pi x_1 x_2}.
\end{multline*}
Setting
\begin{equation}\label{eq:dfv}
v(x_1,x_2) = u(x_1,x_2) - 2^{1/4} \rho(x_1) e^{-\pi x_2^2 + i\pi x_1 x_2}
   - i  2^{1/4} x_2\rho'(x_1) e^{-\pi x_2^2 + i\pi x_1 x_2},
\end{equation}
we infer
\begin{eqnarray*}
\left|v(x_1,x_2) \right|
&\leq& \frac 1 {2^{1/4} R^{5/2}} e^{-\frac{\pi}{2} x_2^2 }
\int_{\R}\int_0^1  y_1^2 e^{-\frac{\pi}2 y_1^2} (1-t) \left| p''\left(
  \frac{x_1}R + t\frac{y_1}R \right) \right| dt dy_1 \\
&\leq &\frac {\|p''\|_{L^\infty}} {2^{1/4} R^{5/2}} e^{-\frac{\pi}{2} x_2^2 }
\int_{\R}\int_0^1  y_1^2 e^{-\frac{\pi}2 y_1^2} (1-t) {\bf
  1}_{(-TR,TR)}(x_1 + ty_1) dt dy_1.
\end{eqnarray*}
Hence, using Jensen's inequality, we see that there is a constant $C_r$
depending only on $r$ such that
\begin{eqnarray*}
\left|v(x_1,x_2) \right|^r &\leq& C_r \frac {\|p''\|_{L^\infty}^r}
{R^{5r/2}} e^{-r\frac{\pi}{2} x_2^2 } \int_{\R}\int_0^1  y_1^2
e^{-\frac{\pi}2 y_1^2} (1-t) {\bf 1}_{(-TR,TR)}(x_1 + ty_1) dt dy_1,
\end{eqnarray*}
whence
\begin{eqnarray*}
  \left\| v \right\|^r_{L^r} &\leq & C_r \frac{\|p''\|_{L^\infty}^r} {R^{5r/2}} \int_\R \int_\R \int_0^1
  e^{-r\frac{\pi}{2} x_2^2 } y_1^2 e^{-\frac{\pi}2 y_1^2} (1-t) \int_\R
  {\bf 1}_{(-TR,TR)}(x_1 + ty_1)dx_1 dt
dx_2 dy_1 \\
&=& C_r \frac {\|p''\|_{L^\infty}^r} {R^{5r/2}} (2TR)  \int_\R \int_\R \int_0^1
  e^{-r\frac{\pi}{2} x_2^2 } y_1^2 e^{-\frac{\pi}2 y_1^2} (1-t)dt
dx_2 dy_1 \\
&=& C'_r \frac {\|p''\|_{L^\infty}^r} {R^{5r/2}} TR,
\end{eqnarray*}
which implies \eqref{eq:dl-fort}.
\end{proof}
\begin{lem}
\label{lm:ftest-fort-3}
Under the same assumptions as Lemma~\ref{lm:ftest-fort-2}, let $u$ be
defined by \eqref{eq:ftest-fort}. Then, there exists
a constant $C_T>0$ depending only on $T$ such that
$u$ satisfies
\begin{multline}
  \label{eq:dl-fort-x1}
 \int_{\R^2} x_1^2\left| u(x_1,x_2) - 2^{1/4} \rho(x_1) e^{-\pi x_2^2
 + i\pi x_1 x_2} - i 2^{1/4} x_2 \rho'(x_1) e^{-\pi x_2^2 + i\pi x_1 x_2}
 \right|^2 dx \\ \leq C_T\frac{\|p''\|_{L^\infty(\R)}^2} {R^2},
\end{multline}
and
\begin{multline}
  \label{eq:dl-fort-x2}
 \int_{\R^2} x_2^2\left| u(x_1,x_2) - 2^{1/4} \rho(x_1) e^{-\pi x_2^2 + i\pi x_1 x_2}
   - i 2^{1/4} x_2 \rho'(x_1) e^{-\pi x_2^2 + i\pi x_1 x_2}
 \right|^2 dx \\ \leq C_T\frac{\|p''\|_{L^\infty(\R)}^2} {R^4}.
\end{multline}
\end{lem}
\begin{proof}
  Here again, we use the Taylor expansion \eqref{eq:taylor-fort}. Hence,
  $v$ being defined by \eqref{eq:dfv}, we have
  \begin{eqnarray*}
     |x_1| |v(x_1,x_2)|  &\leq &\frac {\|p''\|_{L^\infty}} {2^{1/4}R^{5/2}}
     |x_1| e^{-\frac{\pi}{2} x_2^2 }
\int_{\R}\int_0^1  y_1^2 e^{-\frac{\pi}2 y_1^2} (1-t) {\bf
  1}_{(-TR,TR)}(x_1 + ty_1) dt dy_1 \\
&\leq & \frac {\|p''\|_{L^\infty}} {2^{1/4}R^{5/2}}  e^{-\frac{\pi}{2} x_2^2 }
\int_{\R}\int_0^1  y_1^2 e^{-\frac{\pi}2 y_1^2} (1-t)|x_1+ty_1| {\bf
  1}_{(-TR,TR)}(x_1 + ty_1) dt dy_1 \\
&& +  \frac {\|p''\|_{L^\infty}} {2^{1/4}R^{5/2}}  e^{-\frac{\pi}{2} x_2^2 }
\int_{\R}\int_0^1  |y_1|^3 e^{-\frac{\pi}2 y_1^2} t(1-t) {\bf
  1}_{(-TR,TR)}(x_1 + ty_1) dt dy_1.
  \end{eqnarray*}
Hence, using Jensen's inequality and arguing as in the proof of Lemma~\ref{lm:ftest-fort-2}, we have
$$
\left\| x_1 v \right\|_{L^2(\R^2)} \leq C  \frac {\|p''\|_{L^\infty}}
{R^{5/2}} \left( (RT)^{3/2} + \sqrt{RT} \right),
$$
where $C$ is a universal constant. This implies \eqref{eq:dl-fort-x1}.
A similar computation gives
$$\left\| x_2 v \right\|_{L^2(\R^2)} \leq C
\frac{\|p''\|_{L^\infty}}{R^{5/2}} \sqrt{RT},$$
which proves \eqref{eq:dl-fort-x2}.
\end{proof}
\subsection{Lower bound for the energy}
\label{ssec:borne-inf-fort}
We first recall an important result by Carlen \cite{MR1105661} about
wave functions in $\Lambda_0$ (defined by \eqref{lll}):
\begin{lem}[E. A. Carlen, \cite{MR1105661}]
\label{lm:carlen}
For any $u\in \Lambda_0$, $\nabla u \in L^2$, and we have
\begin{equation}
  \label{eq:carlen}
  \int_{\R^2} |\nabla |u||^2 = \pi\int_{\R^2} |u|^2.
\end{equation}
\end{lem}
\begin{oss}
  {\rm The result of Carlen is actually much more general than the one we
  cite here, but the special case \eqref{eq:carlen} is the only thing we
  need.}
\end{oss}
Lemma~\ref{lm:carlen} implies the following decomposition of the energy
in $\Lambda_0$:
\begin{lem}
  \label{lm:binf-fort}
Let $u\in \Lambda_0$ be such that $\|u\|_{L^2} = 1.$ Then, we have
\begin{eqnarray}
  \label{eq:decomp-fort}
  E(u) &=& - \frac{\kappa^2}{8\pi} + \frac{\kappa^2}2
  \left(\frac 1 {4\pi^2} \int_{\R^2} \left|\partial_2 |u|\right|^2 +
    \int_{\R^2} x_2^2 |u|^2\right) \nonumber \\
&&+ \frac{\kappa^2}{8\pi^2} \int_{\R^2} \left|\partial_1
|u|\right|^2 + \frac{\ep^2}2  \int_{\R^2} x_1^2 |u|^2 + \frac{g_0}2
\int_{\R^2} |u|^4.
\end{eqnarray}
\end{lem}
\begin{proof}
  We write
\begin{equation}\label{eq:decomp-fort-2}
E(u) =  - \frac{\kappa^2}{8\pi} + \frac{\kappa^2}{8\pi} +
\frac{\kappa^2}2 \int_{\R^2} x_2^2 |u|^2+ \frac{\ep^2}2  \int_{\R^2}
x_1^2 |u|^2 + \frac{g_0}2 \int_{\R^2} |u|^4.
\end{equation}
Hence, applying \eqref{eq:carlen}, we find \eqref{eq:decomp-fort}.
\end{proof}
Note that the first line is easily seen to be bounded from below by the
first eigenvalue of the corresponding harmonic oscillator, namely
$\kappa^2/(4\pi)$. Hence, \eqref{eq:decomp-fort} readily implies
\begin{equation}\label{eq:decomp-fort-3}
E(u) \geq \frac{\kappa^2}{8\pi}.
\end{equation}
This explains why we chose the constant $\frac{\kappa^2}{8\pi}$ in the
decomposition \eqref{eq:decomp-fort-2}: it is the constant which gives
the highest lower bound in \eqref{eq:decomp-fort-3}.
\subsection{Proof of Theorem \ref{th:cv-fort}}
\label{ssec:fin-fort}
\noindent{\em }\begin{proof}[Step 1: upper bound for the energy]
We
  pick  a real-valued function $p$ such that
$$p\in C^2(\R), \supp(p) \subset (-T,T), \quad \int_{\R} p^2 = 1,$$
and define $u$ by \eqref{eq:ftest-fort}, where $\rho$ is defined by
\eqref{eq:alpha-fort}, with
\begin{equation}\label{eq:dfR}
R= \ep^{-2/3}.
\end{equation}
Hence, setting
$v = \frac 1 {\|u\|_{L^2}} u,$
we know by Lemma~\ref{lm:ftest-fort-1} that $v$ is a test function for
$I(\ep,\kappa)$. Hence,
\begin{equation}
\label{eq:bsup-fort-1}
I(\ep,\kappa) \leq E(v).
\end{equation}
Next, we set
$$v_1 =2^{1/4} \rho(x_1) e^{-\pi x_2^2 + i\pi x_1 x_2}
   +  i 2^{1/4} x_2 \rho'(x_1) e^{-\pi x_2^2 + i\pi x_1 x_2} ,$$
and point out  that, applying Lemma~\ref{lm:ftest-fort-2} with $r=2$,
\begin{multline*}
\|u\|_{L^2}^2 = \|v_1\|_{L^2}^2 + O\left(\ep^{4/3}\right) = 1+ 2^{1/2} \int_{\R}
|\rho'(x_1)|^2 \int_{\R}x_2^2 e^{-2\pi x_2^2} dx_2 + O\left(\ep^{4/3}\right) \\= 1 + C
\ep^{4/3} \int_{\R} p'^2 + O\left(\ep^{4/3}\right),
\end{multline*}
where we have used that the two terms defining $v_1$ are orthogonal to
each other. Hence,
$$
 \|u\|_{L^2} = 1 + O\left(\ep^{4/3}\right),
$$
where the term $O\left(\ep^{4/3} \right)$ depends only on $\|p'\|_{L^2}$,
$\|p''\|_{L^\infty}$ and $T.$
According to \eqref{eq:bsup-fort-1} and the definition of $v$, we thus
have
\begin{equation}
  \label{eq:bsup-fort-2}
 I(\ep,\kappa) \leq E(u)\left[1 + O\left(\ep^{4/3}\right)\right],
\end{equation}
where the term $O\left(\ep^{4/3} \right)$ is independent of $\kappa$. We now compute the
energy of $u$: applying Lemma~\ref{lm:ftest-fort-3}, we have
$$
  \left|\int_{\R^2} x_1^2 |u|^2 - \int_{\R^2} x_1^2 |v_1|^2 \right| \leq
  C \ep^{2/3} \left(\|x_1 u\|_{L^2} + \|x_1 v_1\|_{L^2}\right) \leq
  C\ep^{2/3} \left( 2\|x_1 v_1\|_{L^2} + C\ep^{2/3} \right).
$$
Moreover, we have, since $\rho$ is real-valued,
$$\int_{\R^2} x_1^2 |v_1|^2 dx = \int_{\R} x_1^2 \rho(x_1)^2 dx_1 +
\frac{1}{4\pi} \int_{\R} x_1^2 \rho'(x_1)^2 dx_1 = \ep^{-4/3}
\int_{\R} t^2 p(t)^2 dt + O\left(1\right).
$$
Hence, we have
\begin{equation}
  \label{eq:bsup-fort-3}
  \int_{\R^2} x_1^2 |u|^2 = \ep^{-4/3} \int_{\R} t^2 p(t)^2 dt +
  O\left(1\right).
\end{equation}
The same kind of argument allows us to prove that
\begin{equation}
  \label{eq:bsup-fort-4}
  \int_{\R^2} x_2^2 |u|^2 = \int_{\R^2} x_2^2 v_1^2 +
  O\left(\ep^{4/3}\right) = \frac1{4\pi} +O\left(\ep^{4/3}\right).
\end{equation}
Next, we apply Lemma~\ref{lm:ftest-fort-2} with
$r=4$:
$$\left|\int_{\R^2} |u|^4 - \int_{\R^2} |v_1|^4\right| \leq 2 \| u - v_1
\|_{L^4} \left(\|u\|_{L^4}^3 + \|v_1\|_{L^4}^3 \right) \leq C \ep^{3/2}
\left(\|u\|_{L^4}^3 + \|v_1\|_{L^4}^3 \right).$$
Moreover, we have  $\|u\|_{L^4} \leq \|v_1\|_{L^4} + C \ep^{2/3},$ hence
$$\left|\int_{\R^2} |u|^4 - \int_{\R^2} |v_1|^4\right| \leq C
\ep^{3/2}\|v_1\|_{L^4}^3.$$
We also have
\begin{eqnarray*}
\int_{\R^2} |v_1|^4 &=& \int_{R^2} 2\rho(x_1)^4 e^{-4\pi x_2^2} +
4 \rho(x_1)^2 \rho'(x_1)^2 x_2^2 e^{-4\pi x_2^2} + 2 x_2^4
\rho'(x_1)^4 e^{-4\pi x_2^2} \\
&=& \ep^{2/3} \int_{\R} p^4 + \ep^2 \frac{1}{4\pi} \int_{\R} p(t)^2
p'(t)^2 dt + \ep^{10/3}\frac{3}{64 \pi^2} \int_{\R} p'^4.
\end{eqnarray*}
Hence, we obtain
\begin{equation}
  \label{eq:bsup-fort-5}
  \int_{\R^2} |u|^4 = \ep^{2/3} \int_{\R} p(t)^4 dt + O\left(\ep^2\right).
\end{equation}
Collecting \eqref{eq:bsup-fort-3}, \eqref{eq:bsup-fort-4} and
\eqref{eq:bsup-fort-5}, we thus have
$$E(u) = \frac{\kappa^2}{8\pi} + O\left(\kappa^2 \ep^{4/3}\right) +
\ep^{2/3}\left( \int_{\R} \frac12 t^2 p(t)^2dt +\frac{g_0}2
\int_{\R}
  p(t)^4dt\right)+ O\left(\ep^{2}\right). $$
Recalling \eqref{eq:bsup-fort-2}, this implies
$$\frac{I(\ep,\kappa) - \frac{\kappa^2}{8\pi}}{\ep^{2/3}} \leq \frac12\int_{\R} t^2
p(t)^2dt +\frac{g_0}2 \int_{\R} p(t)^4dt +O\left(\kappa^2
\ep^{2/3}\right) + O\left(\ep^{4/3}\right).$$ As a conclusion, we
have
$$\limsup_{\ep \rightarrow 0, \frac{\ep^{1/3}}{\kappa} \to 0} \frac{I(\ep,\kappa) -
  \frac{\kappa^2}{8\pi}}{\ep^{2/3}}  \leq \frac12 \int_{\R} t^2
p(t)^2dt +\frac{g_0}2 \int_{\R} p(t)^4dt,$$ for any real-valued
$p\in C^2(\R)$ having compact support, and such that
$\|p\|_{L^2}=1$. A density argument allows to prove that
$$\limsup_{\ep \rightarrow 0, \frac{\ep^{1/3}}{\kappa} \to 0} \frac{I(\ep,\kappa) -
  \frac{\kappa^2}{8\pi}}{\ep^{2/3}} \leq J,$$
where $J$ is defined by \eqref{eq:pb1d}. Thus, we get
$$\frac{I(\ep,\kappa) - \frac{\kappa^2}{8\pi}}{\ep^{2/3}} = J +c
\left(\ep, \frac{\ep^{1/3}}{\kappa} \right),$$
with $\lim_{\begin{array}{c}{\scriptscriptstyle (t,s)\rightarrow(0,0)}
    \vspace{-4pt} \\ {\scriptscriptstyle t,s>0}\end{array}} c(t,s) = 0.$
\par
\noindent{\em Step 2: convergence of minimizers.}
Let $u$ be a minimizer of $I(\ep,\kappa)$. Then, according to the first
step, we have
$$E(u) \leq \frac{\kappa^2}{8\pi} + J\ep^{2/3} +
\ep^{2/3}c \left(\ep, \frac{\ep^{1/3}}{\kappa} \right),$$
with
$\lim_{\begin{array}{c}{\scriptscriptstyle (t,s)\rightarrow(0,0)}
    \vspace{-4pt} \\ {\scriptscriptstyle t,s>0}\end{array}} c(t,s) = 0.$
Hence, applying Lemma~\ref{lm:binf-fort}, we obtain
\begin{multline}
\label{eq:binf-fort-1}
    \frac{\kappa^2}2
  \left(\frac 1 {4\pi^2} \int_{\R^2} \left|\partial_2 |u|\right|^2 +
    \int_{\R^2} x_2^2 |u|^2\right) \\
+ \frac{\kappa^2}{8\pi^2} \int_{\R^2} \left|\partial_1 |u|\right|^2
+ \frac{\ep^2}2  \int_{\R^2} x_1^2 |u|^2 + \frac{g_0}2 \int_{\R^2}
|u|^4 \leq \frac{\kappa^2}{4\pi} + J\ep^{2/3} + \ep^{2/3}c
\left(\ep, \frac{\ep^{1/3}}{\kappa} \right).
\end{multline}
We set
\begin{equation}
  \label{eq:v}
  v(x_1,x_2) = \frac 1 {\ep^{1/3}} \left| u\left( \frac{x_1}{\ep^{2/3}},
  x_2\right)\right|,
\end{equation}
so that $\|v\|_{L^2} = \|u\|_{L^2} = 1$, $v\geq 0$, and \eqref{eq:binf-fort-1} becomes
\begin{multline}
\label{eq:binf-fort-2}
    \frac{\kappa^2}2
  \left(\frac 1 {4\pi^2} \int_{\R^2} \left|\partial_2 v\right|^2 +
    \int_{\R^2} x_2^2 v^2\right) \\
+ \frac{\kappa^2 \ep^{4/3}}{8\pi^2} \int_{\R^2} \left|\partial_1
v\right|^2 + \frac{\ep^{2/3}}2 \left( \int_{\R^2} x_1^2 v^2 + g_0
\int_{\R^2} v^4\right) \leq \frac{\kappa^2}{4\pi} + J\ep^{2/3} +
\ep^{2/3}c \left(\ep, \frac{\ep^{1/3}}{\kappa} \right).
\end{multline}
This implies that
\begin{equation}
  \label{eq:cv01}
  \int_{\R^2} \left|\partial_2 v\right|^2 +
    \int_{\R^2} x_2^2 v^2 \leq C,
\end{equation}
where $C$ does not depend on $(\ep,\kappa)$. Moreover,
since the first eigenvalue of the operator $-\frac 1 {4\pi^2}
\frac{d^2}{dx_2^2} + x_2^2$ is equal to $1/(2\pi)$,
\eqref{eq:binf-fort-2} implies that
\begin{equation}
  \label{eq:cv02}
 \int_{\R^2} x_1^2 v^2 + g_0 \int_{\R^2} v^4 \leq C,
\end{equation}
where $C$ does not depend on $(\ep,\kappa)$. Hence, up to extracting a
subsequence, $v$ converges weakly in $L^4$ and weakly in $L^2$ to some
limit $v_0 \geq 0$. Using \eqref{eq:cv01} and \eqref{eq:cv02}, we see that
$$\int_{\R^2} |x|^2 v^2 \leq C,$$
hence $v$ converges strongly in $L^2$. Since in addition $\partial_2 v$
converges weakly in $L^2$, we have:
\begin{equation}
  \label{eq:cv03}
\left\{\begin{array}{l}
\displaystyle  v \mathop{\longrightarrow}_{(\ep,\ep^{1/3}\kappa^{-1})\to (0,0)}  v_0 \text{ strongly in } L^2(\R^2),
  \\
\displaystyle x_1 v \mathop{\longrightarrow}_{(\ep,\ep^{1/3}\kappa^{-1})\to (0,0)} x_1 v_0 \text{ weakly in } L^2(\R^2),
  \\
\displaystyle v  \mathop{\longrightarrow}_{(\ep,\ep^{1/3}\kappa^{-1})\to (0,0)} v_0 \text{ weakly in }
  L^4(\R^2), \\
\displaystyle
\partial_2 v \mathop{\longrightarrow}_{(\ep,\ep^{1/3}\kappa^{-1})\to (0,0)} \partial_2 v_0 \text{
  weakly in } L^2(\R^2).
\end{array}\right.
\end{equation}
Hence, we may pass to the liminf in the two first terms of
\eqref{eq:binf-fort-2}, getting
\begin{equation}
  \label{eq:cv04}
 \frac 1 {4\pi^2} \int_{\R^2} \left|\partial_2 v_0\right|^2 +
    \int_{\R^2} x_2^2 v_0^2\leq \liminf_{(\ep,\ep^{1/3}\kappa^{-1})\to (0,0)} \left(\frac 1 {4\pi^2} \int_{\R^2} \left|\partial_2 v\right|^2 +
    \int_{\R^2} x_2^2 v^2 \right) \leq \frac1{2\pi}.
\end{equation}
We use that the first eigenvalue of the operator $-\frac 1 {4\pi^2}
\frac{d^2}{dx_2^2} + x_2^2$ on $L^2(\R)$ is equal to $1/(2\pi)$, is simple, with an
eigenvector equal to $2^{1/4} \exp(-\pi x_2^2)$. Thus,
\begin{equation}
  \label{eq:cv05}
  v_0(x_1,x_2) = \xi(x_1) 2^{1/4} e^{-\pi x_2^2},
\end{equation}
with $\xi\geq 0$.
Next, \eqref{eq:binf-fort-2} and \eqref{eq:cv03} also imply
\begin{equation}
  \label{eq:cv06}
 \frac12 \int_{\R^2} x_1^2 v_0^2 + \frac {g_0} 2 \int_{\R^2} v_0^4 \leq \liminf_{\ep \to 0, \frac{\ep^{1/3}}{\kappa}\to 0}\left(
 \frac12 \int_{\R^2} x_1^2 v^2 + \frac {g_0}{ 2 }\int_{\R^2} v^4 \right) \leq J.
\end{equation}
Using \eqref{eq:cv05}, we infer
$$\frac12 \int_{\R} x_1^2 \xi^2 + \frac {g_0} 2 \int_\R \xi^4 \leq J.$$
Hence, recalling that, in view of \eqref{eq:cv03} and
\eqref{eq:cv05}, we have $\int \xi^2 = 1,$ the definition of $J$ implies
that $\xi$ is the unique non-negative minimizer of \eqref{eq:pb1d}.
This proves \eqref{eq:ulimfort}, with strong convergence in $L^2$
and weak convergence in $L^4$. Moreover, using \eqref{eq:cv06} again
and the fact that $\xi$ is a minimizer of \eqref{eq:pb1d}, we have
$$\lim_{(\ep,\ep^{1/3}\kappa^{-1})\to (0,0)} \left(\int_{\R^2} x_1^2(v_0^2 - v^2) + g_0\int_{\R^2}\left( v_0^4 -
    v^4 \right) \right) = 0.$$
Next, using the explicit formula giving $v_0$, a simple computation gives
$$\int_{\R^2} x_1^2(v^2 - v_0^2) + g_0(v^4 - v_0^4) \geq g\int_{\R^2}
 \left(v^2 - v_0^2\right)^2,$$
hence $v^2$ converges to $v_0^2$ strongly in $L^2(\R^2)$. Thus,
$$\int_{\R^2} v^4 \longrightarrow \int_{\R^2} v_0^4.$$
The space $L^4(\R^2)$ being uniformly convex, this implies strong
convergence in $L^4$, hence \eqref{eq:ulimfort}.
\par
\noindent{\em Step 3: lower bound for the energy.}
Using Lemma~\ref{lm:binf-fort}, we have
$$E(u) \geq \frac{\kappa^2}{4\pi} + \frac{\ep^{2/3}}2 \left(\int_{\R^2}
x_1^2 v^2 +  g_0 \int_{\R^2} v^4\right).$$ In addition, we already
proved \eqref{eq:ulimfort}, which implies
$$\frac12 \int_{\R^2} x_1^2 v^2 + \frac {g_0} 2 \int_{\R^2} v^4 \longrightarrow
\frac12 \int_{\R^2} x_1^2 v_0^2 + \frac {g_0} 2 \int_{\R^2} v_0^4 =
J,$$ which implies the lower bound for the energy.
\end{proof}
\section{Appendix}
\subsection{Glossary}
\subsubsection{The harmonic oscillator} The operator
\begin{equation}\label{12.harmonic}
\sum_{1\le j\le n} \pi (\xi_{j}^2+\lambda_{j}^2 x_{j}^2)^w=
\sum_{1\le j\le n} \pi (D_{x_{j}}^2+\lambda_{j}^2 x_{j}^2),\quad \lambda_{j}>0,\quad D_{x_{j}}=\frac{1}{2i\pi}\p_{x_{j}},
\quad
\end{equation}
has a discrete spectrum
\begin{equation}\label{12.specharm}
\frac12{\sum_{1\le j\le n}\lambda_{j}}+
\Bigl\{\sum_{1\le j\le n}\alpha_{j}\lambda_{j}\Bigr\}_{(\alpha_{1},\dots,\alpha_{n})\in\N^n
},
\end{equation}
and its ground state is one-dimensional generated by the Gaussian function
\begin{equation}\label{12.groundho}
\varphi_{\lambda}(x)=2^{n/4} \prod_{1\le j\le n} \lambda_{j}^{1/4}e^{-\pi \lambda_{j} x_{j}^2}.
\end{equation}
\subsubsection{Degenerate harmonic oscillator}
Let $r\in\{1,\dots, n\}$.
Using the identity
\begin{equation}\poscal{H_{r} u}{u}=
\sum_{1\le j\le r}  \poscal{(D_{x_{j}}^2+\lambda_{j}^2 x_{j}^2) u}{u}=
\sum_{1\le j\le r}  \norm{(D_{x_{j}}-i\lambda_{j}x_{j})u}_{L^2}^2+\frac{\lambda_{j}}{2\pi}\norm{u}^2_{L^2},
\end{equation}
we can define the ground state $E_{r}$
of the operator $H_{r}$ as
\begin{multline}
E_{r}=L^2(\R^n)\cap_{1\le j\le r}\ker(D_{x_{j}}-i\lambda_{j}x_{j})
\\=\{\varphi_{(\lambda_{1},\dots, \lambda_{r})}(x_{1},\dots,x_{r})\otimes v(x_{r+1},\dots, x_{n})\}
_{v\in L^2(\R^{n-r})}.
\end{multline}
The bottom of the spectrum of $\pi H_{r}$ is
$\frac{1}{2}\sum_{1\le j\le r}\lambda_{j}$.
\subsection{Notations  for the calculations of section \ref{sec.effective}}
\label{notations}
{\small
\begin{align}
\nu^2+\omega^2&\le 1,\quad \nu^2+\omega^2+\ep^2=1, \\
\alpha&=\sqrt{\nu^4+4\omega^2}=
\sqrt{4\omega^2+(1-\omega^2-\ep^2)^2}\quad\text{            (if            $\nu=0, \alpha=2\omega$).}
\\
\mu_{1}^2&= 1+\omega^2-\alpha=\frac{(1+\omega^2)^2-\alpha^2}{1+\omega^2+\alpha}=\frac{(1-\omega^2)^2-\nu^4}{\mu_{2}^2}=\frac{2\nu^2\ep^2+\ep^4}{\mu_{2}^2}\\
\mu_{2}^2&=1+\omega^2+\alpha\quad\text{             (if            $\nu=0, \mu_{2}=1+\omega$).}
\end{align}
\begin{oss}\label{12.rem.obvious}{\rm
If            $\nu=0$, $\mu_{1}=O(\ep^2)$ and if
$\nu\not=0$, $\mu_{1}=O(\ep)$.
Moreover, for $\nu^2+\omega^2\le 1$, $\mu_{2}^2\in[1,4]$ and for
$\nu^2+\omega^2=1$, $\mu_{2}^2\in[2,4]$:
we have indeed
\begin{equation}\label{12.jkl}
1\le 1+\omega^2+(\nu^4+4\omega^2)^{1/2}\le 4
\end{equation}
since
$
 \nu^4+10\omega^2\le (1-\omega^2)^2+10\omega^2=8\omega^2+1+\omega^4\le 9+\omega^4
$,
implying
$
(3-\omega^2)^2\ge \nu^4+4\omega^2
$
and \eqref{12.jkl}. If $\nu^2+\omega^2=1$,
we have
$(1-\omega^2)^2=\nu^4\le \nu^4+4\omega^2\Longrightarrow
2\le 1+\omega^2+(\nu^4+4\omega^2)^{1/2}.
$}
\end{oss}
We define the following set of parameters,
\begin{align}
&\beta_{1}=\frac{2\omega \mu_{1}}{\alpha-2\omega^2+\nu^2}=\frac{\alpha-2\omega^2-\nu^2}{2\omega \mu_{1}}\  \text{\tiny\               since        $ (\alpha-2\omega^2)^2-\nu^4=4\omega^2+4\omega^4-4\omega^2\alpha=4\omega^2\mu_{1}^2$ },\\
&\beta_{2}=\frac{2\omega \mu_{2}}{\alpha+2\omega^2+\nu^2}
=\frac{\alpha+2\omega^2-\nu^2}{2\omega \mu_{2}}\
 \text{\ \tiny               since        $ (\alpha+2\omega^2)^2-\nu^4=4\omega^2+4\omega^4+4\omega^2\alpha=4\omega^2\mu_{2}^2$ },\\
&\gamma=\frac{2\alpha}{\omega},\\
&\lambda_{1}^2=\frac{\mu_{1}}{\mu_{1}+\beta_{1}\beta_{2}\mu_{2}}=\frac{1}{1+\frac{\beta_{1}\beta_{2}\mu_{2}}{\mu_{1}}}=\frac{1}{1+\frac{\alpha+2\omega^2-\nu^2}{\alpha-2\omega^2+\nu^2}}=\frac{\alpha-2\omega^2+\nu^2}{2\alpha},\\
&\lambda_{2}^2=\frac{\mu_{2}}{\mu_{2}+\beta_{1}\beta_{2}\mu_{1}}
=\frac{1}{1+\frac{\beta_{1}\beta_{2}\mu_{1}}{\mu_{2}}}=
\frac{1}{1+\frac{\alpha-2\omega^2-\nu^2}{\alpha+2\omega^2+\nu^2}}=\frac{\alpha+2\omega^2+\nu^2}{2\alpha},
\\
&\text{             and we have
\quad }\lambda_{1}^2+\lambda_{2}^2=1+\frac{\nu^2}{\alpha},\qquad
\lambda_{1}^2\lambda_{2}^2=\frac{(\alpha+\nu^2)^2-4\omega^4}{4\alpha^2},
\\&
d=\frac{\gamma \lambda_{1}\lambda_{2}}{2},\quad c=\frac{\lambda_{1}^2+\lambda_{2}^2}
{2\lambda_{1}\lambda_{2}}\quad
\text{             so that }\quad \scriptstyle
cd=\frac{2\alpha(1+\nu^2/\alpha)}{4\omega}=\frac{\alpha+\nu^2}{2\omega}.
\end{align}
We have also
\begin{align*}
\frac{2\mu_{1}}{\gamma \beta_{1}}=\frac{\alpha-2\omega^2+\nu^2}{\omega\gamma}=
\frac{\alpha-2\omega^2+\nu^2}{2\alpha}=\lambda_{1}^2,
\\
\frac{2\mu_{2}}{\gamma \beta_{2}}=\frac{\alpha+2\omega^2+\nu^2}{\omega\gamma}=
\frac{\alpha+2\omega^2+\nu^2}{2\alpha}=\lambda_{2}^2,
\end{align*}
and
\begin{align*}
&c\lambda_{2}=\frac{\lambda_{1}^2+\lambda_{2}^2}{2\lambda_{1}}=(1+\nu^2\alpha^{-1})2^{-1}
\frac{2^{1/2}\alpha^{1/2}}{\sqrt{\alpha-2\omega^2+\nu^2}}
\\&=
(1+\nu^2\alpha^{-1})2^{-1}
\frac{2^{1/2}\alpha^{1/2}
\sqrt{\alpha+2\omega^2-\nu^2}
}{\sqrt{\alpha-2\omega^2+\nu^2}\sqrt{\alpha+2\omega^2-\nu^2}}
\\&
=(1+\nu^2\alpha^{-1})2^{-1}2^{1/2}\alpha^{1/2}
\frac{\sqrt{\alpha+2\omega^2-\nu^2}}{
\sqrt{\alpha^2-(2\omega^2-\nu^2)^2}}
\\&=(1+\nu^2\alpha^{-1})2^{-1/2}\alpha^{1/2}\sqrt{\alpha+2\omega^2-\nu^2}
(2\omega)^{-1}(2\nu^2+\ep^2)^{-1/2}.
\end{align*}
Moreover, we have
\begin{equation}\label{mat11}
c\lambda_{2}=2^{-3/2}(\alpha^{1/2}+\nu^2\alpha^{-1/2})\omega^{-1}
\sqrt{\frac{\alpha+2\omega^2-\nu^2}{2\nu^2+\ep^2}}
\quad\text{ \tiny            (if            $\nu=0, c\lambda_{2}=2^{-1/2}(1-\omega)^{-1/2}$),}
\end{equation}
$$
\lambda_{2}d^{-1}=\frac{c\lambda_{2}}{cd}=
2^{-3/2}(\alpha^{1/2}+\nu^2\alpha^{-1/2})\omega^{-1}
\sqrt{\frac{\alpha+2\omega^2-\nu^2}{2\nu^2+\ep^2}}\frac{2\omega}{\alpha+\nu^2},
$$
$$
\lambda_{2}d^{-1}= 2^{-1/2}(\alpha^{1/2}+\nu^2\alpha^{-1/2})
\sqrt{\frac{\alpha+2\omega^2-\nu^2}{2\nu^2+\ep^2}}(\alpha+\nu^2)^{-1},
$$
\begin{equation}
\lambda_{2}d^{-1}= (2\alpha)^{-1/2}
\sqrt{\frac{\alpha+2\omega^2-\nu^2}{2\nu^2+\ep^2}}
\quad\text{   \tiny          (if            $\nu=0, \lambda_{2}d^{-1}=2^{-1/2}(1-\omega)^{-1/2}$),}
\end{equation}
{\footnotesize\begin{equation}
c\lambda_{1}=\frac{\lambda_{1}^2+\lambda_{2}^2}{2\lambda_{2}}
=(1+\alpha^{-1}\nu^2)2^{-1/2}\alpha^{1/2}(\alpha+2\omega^2+\nu^2)^{-1/2}
\quad\text{   \tiny             (if            $\nu=0, c\lambda_{1}=
2^{-1/2}(1+\omega)^{-1/2}$),}
\end{equation}}
$$\lambda_{1}d^{-1}=\lambda_{1}c(cd)^{-1}
=
(1+\alpha^{-1}\nu^2)2^{-1/2}\alpha^{1/2}(\alpha+2\omega^2+\nu^2)^{-1/2}
2\omega(\alpha+\nu^2)^{-1},
$$
\begin{equation}
\lambda_{1}d^{-1}=2^{1/2}\alpha^{-1/2}\omega(\alpha+2\omega^2+\nu^2)^{-1/2}
\quad\text{             (if            $\nu=0, \lambda_{1}d^{-1}=2^{-1/2}
(1+\omega)^{-1/2}),
$}
\end{equation}
$$
\lambda_{1} cd =(\alpha+\nu^2)2^{-1}\omega^{-1}(\alpha-2\omega^2+\nu^2)^{1/2}2^{-1/2}\alpha^{-1/2},
$$
$$
\lambda_{1} cd =2^{-3/2}(\alpha+\nu^2)\omega^{-1}\alpha^{-1/2}(\alpha-2\omega^2+\nu^2)^{1/2}
\quad\text{           (if            $\nu=0, \lambda_{1} cd =2^{-1/2}
(1-\omega)^{1/2}),
$}
$$
$$
\frac{d}{\lambda_{2}}=\frac{\gamma \lambda_{1}}{2}=\alpha\omega^{-1}
(\alpha-2\omega^2+\nu^2)^{1/2}2^{-1/2}\alpha^{-1/2}
=2^{-1/2}\alpha^{1/2}\omega^{-1}(\alpha-2\omega^2+\nu^2)^{1/2},
$$
$$
\lambda_{1} cd- \frac{d}{\lambda_{2}}=
(\alpha-2\omega^2+\nu^2)^{1/2}
\bigl(
2^{-3/2}(\alpha+\nu^2)\omega^{-1}\alpha^{-1/2}
-2^{-1/2}\alpha^{1/2}\omega^{-1}
\bigr),
$$
$$
\lambda_{1} cd- \frac{d}{\lambda_{2}}=2^{-3/2}\omega^{-1}\alpha^{-1/2}
(\alpha-2\omega^2+\nu^2)^{1/2}
(\alpha+\nu^2-2\alpha),
$$
{\begin{equation}
\lambda_{1} cd- \frac{d}{\lambda_{2}}=
-2^{-3/2}\omega^{-1}\alpha^{-1/2}
(\alpha-2\omega^2+\nu^2)^{1/2}
(\alpha-\nu^2),
\end{equation}}
\text{  \footnotesize    (if            $\nu=0, \lambda_{1} cd- \frac{d}{\lambda_{2}}
=-2^{-1/2}
(1-\omega)^{-1/2}),
$}
\begin{equation}
\lambda_{1}=2^{-1/2}\alpha^{-1/2}(\alpha-2\omega^2+\nu^2)^{1/2}
\quad\text{             (if            $\nu=0, \lambda_{1}
=2^{-1/2}
(1-\omega)^{1/2}),
$}
\end{equation}
\begin{multline*}
\lambda_{2}cd -\frac d{\lambda_{1}}=\lambda_{1}^{-1}\lambda_{2}(\lambda_{1}cd-\frac{d}{\lambda_{2}})
\\=
-2^{-3/2}\omega^{-1}\alpha^{-1/2}
(\alpha-2\omega^2+\nu^2)^{1/2}
(\alpha-\nu^2)(\alpha+2\omega^2+\nu^2)^{1/2}(\alpha-2\omega^2+\nu^2)^{-1/2}
\\=
-2^{-3/2}\omega^{-1}\alpha^{-1/2}(\alpha-\nu^2)(\alpha+2\omega^2+\nu^2)^{1/2},
\end{multline*}
{\footnotesize\begin{equation}
\lambda_{2}cd -\frac d{\lambda_{1}}=-2^{-3/2}\omega^{-1}\alpha^{-1/2}(\alpha-\nu^2)(\alpha+2\omega^2+\nu^2)^{1/2}
\text{    \tiny        (if            $\nu=0,\lambda_{2}cd -\frac d{\lambda_{1}}
=-2^{-1/2}(1+\omega)^{1/2})
$}
\end{equation}}
\begin{equation}\label{mat44}
\lambda_{2}=2^{-1/2}\alpha^{-1/2}(\alpha+2\omega^2+\nu^2)^{1/2}
\text{             (if            $\nu=0,\lambda_{2}
=2^{-1/2}(1+\omega)^{1/2}),
$}
\end{equation}
\begin{equation}
\frac{\gamma\mu_{1}\beta_{1}}{2}=\frac{2\alpha}{\alpha+2\omega^2+\nu^2} \ep^2,\qquad
\frac{\gamma\mu_{1}}{2\beta_{1}}=\frac{4\alpha\omega(2\nu^2+\ep^2)}{
\alpha-\nu^2+2\omega^2}.
\end{equation}
}
\subsection{Some calculations}
\subsubsection{Proof of the lemma \ref{2.lem.keydiag}}\label{8.app.somec}
We have to calculate
\begin{align*}
&\widetilde Q=\chi^* Q \chi=
\chi^*  \begin{pmatrix}
1-\nu^2&0&0&-\omega\\
0&1+\nu^2&\omega&0\\
0&\omega&1&0\\
-\omega&0&0&1
\end{pmatrix}
\begin{pmatrix}
\lambda_{1}&0&0&-\frac{\lambda_{1}}{d}\\
0&\lambda_{2}&-\frac{\lambda_{2}}{d}&0\\
0&\frac{d}{\lambda_{1}}-\lambda_{2}cd&c\lambda_{2}&0\\
\frac{d}{\lambda_{2}}-\lambda_{1}cd&0&0&c\lambda_{1}\end{pmatrix}\\&=
\chi^*\scriptstyle
\begin{pmatrix}
\scriptstyle (1-\nu^2)\lambda_{1}-\frac{\omega d}{\lambda_{2}}+\lambda_{1}cd \omega&0&0& \scriptstyle
-\frac{\lambda_{1}(1-\nu^2)}{d}-c\lambda_{1}\omega\\
0&\scriptstyle (1+\nu^2)\lambda_{2}+\frac{\omega d}{\lambda_{1}}-\lambda_{2}cd \omega&
\scriptstyle  -\frac{(1+\nu^2)\lambda_{2}}{d}+\omega c \lambda_{2}&0\\
0& \scriptstyle \omega \lambda_{2}+\frac{d}{\lambda_{1}}-\lambda_{2}cd&
\scriptstyle -\frac{\omega \lambda_{2}}{d}+c\lambda_{2}&0\\
\scriptstyle -\omega \lambda_{1}+\frac{d}{\lambda_{2}}-\lambda_{1}cd
&0&0&
 \scriptstyle \frac{\omega \lambda_{1}}{d}+c\lambda_{1}
 \end{pmatrix}\\
 &=
 \begin{pmatrix}
 \lambda_{1}&
0&
0&
 \frac{d}{\lambda_{2}}-\lambda_{1}cd\\
 0&\lambda_{2}&\scriptstyle\frac{d}{\lambda_{1}}-\lambda_{2}cd&
 0\\
0&-\frac{\lambda_{2}}{d}& c\lambda_{2}&0\\
 -\frac{\lambda_{1}}{d}&0&0& c\lambda_{1}
 \end{pmatrix}
 \\& \times
 \begin{pmatrix}
\scriptstyle (1-\nu^2)\lambda_{1}-\frac{\omega d}{\lambda_{2}}+\lambda_{1}cd \omega&0&0& \scriptstyle
-\frac{\lambda_{1}(1-\nu^2)}{d}-c\lambda_{1}\omega\\
0&\scriptstyle (1+\nu^2)\lambda_{2}+\frac{\omega d}{\lambda_{1}}-\lambda_{2}cd \omega&
\scriptstyle  -\frac{(1+\nu^2)\lambda_{2}}{d}+\omega c \lambda_{2}&0\\
0& \scriptstyle \omega \lambda_{2}+\frac{d}{\lambda_{1}}-\lambda_{2}cd&
\scriptstyle -\frac{\omega \lambda_{2}}{d}+c\lambda_{2}&0\\
\scriptstyle -\omega \lambda_{1}+\frac{d}{\lambda_{2}}-\lambda_{1}cd
&0&0&
 \scriptstyle \frac{\omega \lambda_{1}}{d}+c\lambda_{1}
 \end{pmatrix}.
\end{align*}
We get easily
$
\tilde q_{12}=\tilde q_{13}=0=\tilde q_{24}=\tilde q_{34}.
$
To prove that the symmetric matrix $\tilde Q$ is diagonal,
it is thus sufficient to prove that
$\tilde q_{14}=0=\tilde q_{23}$.
We have
\begin{align*}
\tilde q_{14}&=-\frac{\lambda_{1}^2}{d}(1-\nu^2)-\omega c\lambda_{1}^2+\omega\frac{\lambda_{1}}{\lambda_{2}}+\frac{cd \lambda_{1}}{\lambda_{2}}-\lambda_{1}^2 c\omega
-c^2\lambda_{1}^2 d\\
&=\frac{\lambda_{1}^2}{d}\Big[
-1+\nu^2-2\omega cd+\frac{\omega d}{\lambda_{2}\lambda_{1}}+\frac{cd^2}{\lambda_{2}\lambda_{1}}-c^2 d^2
\Bigr]
\\&=
\frac{\lambda_{1}^2}{d}\Big[
-1+\nu^2-\alpha -\nu^2+\alpha+\frac{(\alpha+\nu^2)}{2\omega}\frac{\alpha}
{\omega}-\frac{(\alpha+\nu^2)^2}{4\omega^2}
\Bigr]
\\&=
\frac{\lambda_{1}^2}{d\omega^2}\Big[
-\omega^2+\frac{(\alpha^2+\nu^2\alpha)}{2}
-\frac{(\alpha+\nu^2)^2}{4}
\Bigr]
\\&=
\frac{\lambda_{1}^2}{d\omega^2}\Big[
-\omega^2+\frac{(\nu^4+4\omega^2+\nu^2\alpha)}{2}
-\frac{(\alpha^2+\nu^4+2\alpha \nu^2)}{4}
\Bigr]
\\&=
\frac{\lambda_{1}^2}{d\omega^2}\Big[
-\omega^2+\frac{(\nu^4+4\omega^2+\nu^2\alpha)}{2}
-\frac{(\nu^4+4\omega^2+\nu^4+2\alpha \nu^2)}{4}
\Bigr]
\\&=
\frac{\lambda_{1}^2}{d\omega^2}\Big[
-\omega^2+\frac{(2\nu^4+8\omega^2+2\nu^2\alpha)}{4}
-\frac{(2\nu^4+4\omega^2+2\alpha \nu^2)}{4}
\Bigr]=0,\hs \text{\footnotesize qed}.
\end{align*}
Moreover we have
\begin{align*}
\tilde q_{23}&=-\frac{\lambda_{2}^2}{d}(1+\nu^2)+\omega c\lambda_{2}^2
-\frac{\omega \lambda_{2}}{\lambda_{1}}+\frac{cd\lambda_{2}}{\lambda_{1}}
+\lambda_{2}^2\omega c-\lambda_{2}^2c^2 d
\\
&=\frac{\lambda_{2}^2}{d}\Big[
-1-\nu^2+2\omega cd-\frac{\omega d}{\lambda_{2}\lambda_{1}}+\frac{cd^2}{\lambda_{2}\lambda_{1}}-c^2 d^2
\Bigr]
\\&=
\frac{\lambda_{2}^2}{d}\Big[
-1-\nu^2+\alpha +\nu^2-\alpha+\frac{(\alpha+\nu^2)}{2\omega}\frac{\alpha}
{\omega}-\frac{(\alpha+\nu^2)^2}{4\omega^2}
\Bigr]
\\&=
\frac{\lambda_{2}^2}{d\omega^2}\Big[
-\omega^2+\frac{(\alpha^2+\nu^2\alpha)}{2}
-\frac{(\alpha+\nu^2)^2}{4}
\Bigr]=0,\hs
\text{\footnotesize from the previous computation}.
\end{align*}
We know now that $\tilde Q$ is indeed diagonal.
We calculate
$$
\tilde{q}_{44}=\frac{\lambda_{1}^2(1-\nu^2)}{d^2}+\frac{2c\lambda_{1}^2\omega}{d}+c^2\lambda_{1}^2
=
\frac{\lambda_{1}^2}{d^2}
\Bigl[
1-\nu^2+2\omega cd+c^2 d^2
\Bigr]
=
\frac{\lambda_{1}^2}{d^2}
\Bigl[
1-\nu^2+\alpha+\nu^2+\frac{(\alpha+\nu^2)^2}{4\omega^2}
\Bigr]
$$
$$
\tilde{q}_{44}=
\frac{\lambda_{1}^2}{\omega^2d^2}
\Bigl[
\omega^2+\alpha\omega^2+\frac{(\nu^4+4\omega^2+\nu^4+2\alpha \nu^2)}{4}
\Bigr]
=
\frac{\lambda_{1}^2}{\omega^2d^2}
\Bigl[
2\omega^2+\alpha\omega^2+\frac{(\nu^4+\alpha \nu^2)}{2}
\Bigr].
$$
Since
 $\frac{\lambda_{1}^2}{\omega^2d^2}=\frac{4}{\gamma^2\lambda_{2}^2\omega^2}
=\frac{1}{\alpha^2\lambda_{2}^2}=\frac{2\alpha}{\alpha^2(\alpha+2\omega^2+\nu^2)}$, we have
$$
\tilde{q}_{44}=
\frac{1}{\alpha(\alpha+2\omega^2+\nu^2)}
\Bigl[
4\omega^2+2\alpha\omega^2+{\nu^4+\alpha \nu^2}
\Bigr]=\frac{\alpha^2+2\alpha\omega^2+\alpha\nu^2}{\alpha^2+2\alpha\omega^2+\alpha\nu^2}=1.
$$
Analogously, we have
$$
\tilde{q}_{33}=\frac{\lambda_{2}^2(1+\nu^2)}{d^2}-\frac{2c\lambda_{2}^2\omega}{d}+c^2\lambda_{2}^2
=
\frac{\lambda_{2}^2}{d^2}
\Bigl[
1+\nu^2-2\omega cd+c^2 d^2
\Bigr]
=
\frac{\lambda_{2}^2}{d^2}
\Bigl[
1+\nu^2-\alpha-\nu^2+\frac{(\alpha+\nu^2)^2}{4\omega^2}
\Bigr]
$$
$$
\tilde{q}_{33}=
\frac{\lambda_{2}^2}{\omega^2d^2}
\Bigl[
\omega^2-\alpha\omega^2+\frac{(\nu^4+4\omega^2+\nu^4+2\alpha \nu^2)}{4}
\Bigr]
=
\frac{\lambda_{2}^2}{\omega^2d^2}
\Bigl[
2\omega^2-\alpha\omega^2+\frac{(\nu^4+\alpha \nu^2)}{2}
\Bigr].
$$
Since
 $\frac{\lambda_{2}^2}{\omega^2d^2}=\frac{4}{\gamma^2\lambda_{1}^2\omega^2}
=\frac{1}{\alpha^2\lambda_{1}^2}=
\frac{2\alpha}{\alpha^2(\alpha-2\omega^2+\nu^2)}$, we have
$$
\tilde{q}_{33}=
\frac{1}{\alpha(\alpha-2\omega^2+\nu^2)}
\Bigl[
4\omega^2-2\alpha\omega^2+{\nu^4+\alpha \nu^2}
\Bigr]=\frac{\alpha^2-2\alpha\omega^2+\alpha\nu^2}{\alpha^2-2\alpha\omega^2+\alpha\nu^2}=1.
$$
We calculate
\begin{align*}
&\tilde{q}_{11}=\lambda_{1}^2(1-\nu^2)-2\frac{\omega d\lambda_{1}}{\lambda_{2}}
+2\lambda_{1}^2 cd \omega+\frac{d^2}{\lambda_{2}^2}-2\frac{cd^2\lambda_{1}}{\lambda_{2}}+\lambda_{1}^2 c^2d^2
\\
&\tilde{q}_{11}=\lambda_{1}^2
\Bigl[(1-\nu^2)-\frac{2\omega d}{\lambda_{1}\lambda_{2}}
+2cd\omega+\frac{d^2}{\lambda_{1}^2\lambda_{2}^2}-2\frac{cd^2}{\lambda_{1}\lambda_{2}}+c^2d^2\Bigr]
\\
&\tilde{q}_{11}=\lambda_{1}^2
\Bigl[(1-\nu^2)-2\alpha
+\alpha+\nu^2+
\frac{\alpha^2}{\omega^2}
-2\frac{\alpha+\nu^2}{2\omega}\frac{\alpha}{\omega}
+\frac{(\alpha+\nu^2)^2}{4\omega^2}\Bigr]
\\
&\tilde{q}_{11}=\frac{\lambda_{1}^2}{\omega^2}
\Bigl[(1-\alpha)\omega^2
+
{\alpha^2}
-\alpha^2-\alpha\nu^2
+\frac{(\alpha+\nu^2)^2}{4}\Bigr]
\\
&\tilde{q}_{11}=\frac{\alpha-2\omega^2+\nu^2}{2\alpha\omega^2}
\Bigl[\omega^2-\alpha\omega^2
-\alpha\nu^2
+\frac{\nu^4+4\omega^2+\nu^4+2\alpha\nu^2}{4}\Bigr]
\\
&\tilde{q}_{11}=\frac{\alpha-2\omega^2+\nu^2}{2\alpha\omega^2}
\Bigl[2\omega^2-\alpha\omega^2
-\frac12\alpha\nu^2
+\frac{\nu^4}{2}\Bigr].
\end{align*}
More calculations:
\begin{multline*}
(\alpha-2\omega^2+\nu^2)(2\omega^2+\frac{\nu^4}{2}-\alpha(\omega^2+\frac{\nu^2}{2}))
\\=(\nu^2-2\omega^2)(2\omega^2+\frac{\nu^4}{2})
-(\nu^4+4\omega^2)(\omega^2+\frac{\nu^2}{2})
+\alpha\Bigl(
2\omega^2+\frac{\nu^4}{2}+
(\omega^2+\frac{\nu^2}{2})(2\omega^2-\nu^2)
\Bigr)
\\
=-8\omega^4-2\omega^2\nu^4+\alpha(2\omega^4+2\omega^2)
\end{multline*}
which is equal to
$$
2\alpha \omega^2(1+\omega^2-\alpha)=\alpha(2\omega^4+2\omega^2)-2\alpha^2\omega^2=\alpha(2\omega^4+2\omega^2)-2\omega^2(\nu^4+4\omega^2),
$$
proving thus that
$
\tilde{q}_{11}=1+\omega^2-\alpha.
$
The previous calculations and
\eqref{2.parone} give
$\varphi
\tilde{q}_{22}=1+\omega^2+\alpha,
$
completing the proof of the lemma.
\subsubsection{On the symplectic relationships in Lemma \ref{2.lem.effectq}}\label{a.double}
The reader is invited to check the following formulas\footnote{This is indeed double-checking since those formulas are  proven in  section 2.},
with the notations of lemma \ref{2.lem.effectq}:
{\small \begin{gather*}
\bigl\{\xi_{1}-\bigl(\frac{\alpha-\nu^2}{2\omega}\bigr)x_{2},
\xi_{2}+\bigl(\frac{\alpha+\nu^2}{2\omega}\bigr)x_{1}\bigr\}=\alpha\omega^{-1}
,
\bigl\{\xi_{2}-\bigl(\frac{\alpha-\nu^2}{2\omega}\bigr)x_{1},
\xi_{1}+\bigl(\frac{\alpha+\nu^2}{2\omega}\bigr)x_{2}\bigr\}=\alpha\omega^{-1},
\\
\bigl\{\xi_{1}-\bigl(\frac{\alpha-\nu^2}{2\omega}\bigr)x_{2},
\xi_{1}+\bigl(\frac{\alpha+\nu^2}{2\omega}\bigr)x_{2}\bigr\}=0
,
\bigl\{\xi_{1}-\bigl(\frac{\alpha-\nu^2}{2\omega}\bigr)x_{2},
\xi_{2}-\bigl(\frac{\alpha-\nu^2}{2\omega}\bigr)x_{1}\bigr\}=0,
\\
\bigl\{\xi_{2}+\bigl(\frac{\alpha+\nu^2}{2\omega}\bigr)x_{1},
\xi_{1}+\bigl(\frac{\alpha+\nu^2}{2\omega}\bigr)x_{2}\bigr\}=0
,
\bigl\{\xi_{2}+\bigl(\frac{\alpha+\nu^2}{2\omega}\bigr)x_{1},
\xi_{2}-\bigl(\frac{\alpha-\nu^2}{2\omega}\bigr)x_{1}\bigr\}=0,
\end{gather*}
}
as well as
\begin{multline*}
\bigl(\frac{\alpha-2\omega^2+\nu^2}{2\alpha}\bigr)^{1/2}
\bigl(\frac{\alpha+2\omega^2-\nu^2}{2\alpha\mu_{2}^2}\ep^2\bigr)^{1/2}
\alpha\omega^{-1}
=2^{-1}\ep\mu_{2}^{-1}\omega^{-1}\bigl(\alpha^2-(2\omega^2-\nu^2)^2\bigr)^{1/2}
\\
=2^{-1}\ep\mu_{2}^{-1}\omega^{-1}
\bigl(4\omega^2-4\omega^4+4\omega^2\nu^2\bigr)^{1/2}
=\ep\mu_{2}^{-1}(1-\omega^2+\nu^2)^{1/2}
=\ep\mu_{2}^{-1}(2\nu^2+\ep^2)^{1/2}=\mu_{1}
\end{multline*}
and
$$
\bigl(\frac{\alpha+2\omega^2+\nu^2}{2\alpha}\bigr)^{1/2}
2^{1/2}\omega
\bigl(\frac{1+\omega^2+\alpha}{\alpha(\alpha+2\omega^2+\nu^2)}\bigr)^{1/2}
\alpha\omega^{-1}
=(1+\omega^2+\alpha)^{1/2}=\mu_{2}.
$$

\bibliography{abl}
\nocite{*}
\bibliographystyle{amsplain}
\end{document}